\documentclass[a4paper,10pt]{article}
 \usepackage[
   margin=3.5cm,
   includefoot,
   footskip=20pt,
 ]{geometry}

\usepackage{layout}

\usepackage{hyperref}
\hypersetup{hypertexnames = false, bookmarksdepth = 2, bookmarksopen = true, colorlinks, linkcolor = black, citecolor = black, urlcolor = blue, pdfstartview={XYZ null null 1}}

\usepackage{amsfonts}
\usepackage{amsmath}
\usepackage{amsthm}
\usepackage{amssymb}
\usepackage[dvipsnames]{xcolor}

\usepackage{chngcntr}

\usepackage[capitalise]{cleveref}

\usepackage{colonequals}
\usepackage{mathtools}
\usepackage{thmtools}
\usepackage{tikz-cd}
\usepackage[colorinlistoftodos, textsize = footnotesize]{todonotes}
\usepackage{xparse}
\usepackage{xspace}
\usepackage{color}

\usepackage{enumitem}

\usepackage[utf8]{inputenc}
\usepackage[T1]{fontenc}

\usepackage[normalem]{ulem}

\usepackage{microtype}
\usepackage{lmodern}
\theoremstyle{definition}
\newtheorem{theorem}[subsubsection]{Theorem}
\newtheorem{corollary}[theorem]{Corollary}
\newtheorem{definition}[theorem]{Definition}
\newtheorem{example}[theorem]{Example}

\newtheorem{lemma}[theorem]{Lemma}
\newtheorem{proposition}[theorem]{Proposition}

\newtheorem{remark}[theorem]{Remark}

\declaretheoremstyle[
  spaceabove = 3pt,
  spacebelow = 3pt,
  bodyfont=\normalfont\itshape,
]{alpha}
\declaretheorem[numberwithin=section, style=alpha, title=Theorem]{alphatheorem}

\DeclareMathOperator\Aut{Aut}
\DeclareMathOperator\coker{coker}
\DeclareMathOperator\dimvect{\underline{\dim}}
\DeclareMathOperator\dual{D}
\DeclareMathOperator\Hom{Hom}
\DeclareMathOperator\End{End}
\DeclareMathOperator\Ext{Ext}
\DeclareMathOperator\Filt{Filt}
\DeclareMathOperator\Gr{Gr}
\DeclareMathOperator\gr{gr}

\DeclareMathOperator\HHom{\mathbb{H}\mathrm{om}} 
\DeclareMathOperator\lHom{\mathcal{H}\hspace{-1pt}\textit{om}} 
\DeclareMathOperator\im{im}
\DeclareMathOperator\IIsom{\mathbb{I}\mathrm{som}}
\DeclareMathOperator\Maps{Maps}
\DeclareMathOperator\Spec{Spec}
\DeclareMathOperator\relSpec{\underline{Spec}}
\DeclareMathOperator\supp{supp}
\DeclareMathOperator\Sym{Sym}

\DeclareMathOperator\codim{codim}

\newcommand\blank{\underline{\ \ }}
\newcommand{\sslash}{\mathbin{/\mkern-6mu/}}
\mathchardef\mhyphen="2D

\newcommand\can{\mathrm{can}}
\newcommand\Gm{\ensuremath{\mathbb{G}_{\mathrm m}}}
\newcommand\op{\mathrm{op}}
\newcommand\Sch{\mathrm{Sch}}
\newcommand\Sets{\mathrm{Sets}}
\newcommand\univ{\mathrm{univ}}
\newcommand\wt{\mathrm{wt}}

\newcommand\bbA{\ensuremath{\mathbb{A}}}
\newcommand\CC{\ensuremath{\mathbb{C}}}
\newcommand\FF{\ensuremath{\mathbb{F}}}
\newcommand\GG{\ensuremath{\mathbb{G}}}

\newcommand\NN{\ensuremath{\mathbb{N}}}
\newcommand\PP{\ensuremath{\mathbb{P}}}
\newcommand\QQ{\ensuremath{\mathbb{Q}}}
\newcommand\RR{\ensuremath{\mathbb{R}}}

\newcommand\ZZ{\ensuremath{\mathbb{Z}}}

\newcommand\cA{\mathcal{A}}
\newcommand\cB{\mathcal{B}}
\newcommand\cE{\mathcal{E}}
\newcommand\cF{\mathcal{F}}
\newcommand\cG{\mathcal{G}}
\newcommand\cK{\mathcal{K}}
\newcommand\cL{\mathcal{L}}
\newcommand\cM{\mathcal{M}}
\newcommand\cO{\mathcal{O}}
\newcommand\cQ{\mathcal{Q}}
\newcommand\cS{\mathcal{S}}
\newcommand\cT{\mathcal{T}}
\newcommand\cU{\mathcal{U}}
\newcommand\cV{\mathcal{V}}
\newcommand\cX{\mathcal{X}}

\DeclareDocumentCommand\Catrep{O{k}O{Q}}{\operatorname{rep}_{#1}#2}
\DeclareDocumentCommand\CatRep{O{k}O{Q}}{\operatorname{Rep}_{#1}#2}
\DeclareDocumentCommand\Catmod{O{Q}}{k#1\text{-mod}}
\DeclareDocumentCommand\CatMod{O{Q}}{k#1\text{-Mod}}

\DeclareDocumentCommand\Rep{m}{\mathrm{R}_{#1}}

\DeclareMathOperator\GL{GL}

\DeclareMathOperator\Pic{Pic}
\DeclareMathOperator\rk{rk}

\DeclareDocumentCommand\modulistack{om}{\IfNoValueTF{#1}{\mathcal{M}_{#2}}{\mathcal{M}^{#1}_{#2}}}
\DeclareDocumentCommand\modulispace{om}{\IfNoValueTF{#1}{\mathrm{M}_{#2}}{\mathrm{M}^{#1}_{#2}}}

\newcommand\semistable[1]{#1\mhyphen\mathrm{ss}}
\newcommand\stable[1]{#1\mhyphen\mathrm{s}}

\newcommand\STbar{\ensuremath{\overline{\mathrm{ST}}}}

\newcommand*\quot[2]{\left.\raisebox{.0em}{$#1$}\big/\raisebox{-.0em}{$#2$}\right.}

\title{Projectivity and effective global generation of determinantal line bundles on quiver moduli}

\author{Pieter Belmans \and Chiara Damiolini \and Hans Franzen \and Victoria Hoskins \and Svetlana Makarova \and Tuomas Tajakka}

\begin{document}

\maketitle

\begin{abstract}
  We give a moduli-theoretic treatment of the existence and properties of moduli spaces of semistable quiver representations, avoiding methods from geometric invariant theory.
  Using the existence criteria of Alper--Halpern-Leistner--Heinloth, we show that for many stability functions, the stack of semistable  representations admits an adequate moduli space,
  and prove that this moduli space is proper over the moduli space of semisimple representations.
  We construct a natural determinantal line bundle that descends to a semiample line bundle on the moduli space and provide new effective bounds for global generation.
  For an acyclic quiver, we show that this line bundle is ample, thus giving a modern proof of the fact that the moduli space is projective.
\end{abstract}

\tableofcontents

\section{Introduction}

Many important natural problems in linear algebra such as Jordan normal forms, the classification of tuples of matrices up to simultaneous conjugation, and the classification of tuples of subspaces in a fixed vector space can all be encoded as moduli problems for representations of a quiver, where properties of the classification problem are described by the geometry of the moduli space. Moreover, quiver moduli spaces closely interact with other important moduli spaces and play a fundamental role in representation theory.

The theory of quivers and their representations goes back to Gabriel,
who showed that the quivers for which the moduli problem is \emph{discrete}
are precisely the Dynkin quivers \cite{MR0332887}.
After pioneering work of Kac \cite{MR0557581}, the study of \emph{continuous} moduli problems
began with King's construction \cite{MR1315461} of moduli spaces of quiver representations
using geometric invariant theory (GIT),
in which the GIT notion of stability obtains a
reinterpretation as stability of representations with respect to a stability function.
For the trivial stability function, all quiver representations are semistable and the moduli space is just an affine GIT quotient which classifies semisimple quiver representations. For a nontrivial stability function, the moduli space of semistable representations is projective over the affine GIT quotient. The ring of invariants turns out to be generated by taking traces along oriented cycles \cite{MR958897}, so in particular when the quiver is acyclic, the semisimple moduli space is a point and moduli spaces of semistable quiver representations are projective varieties. An excellent survey of the geometry of these moduli spaces is given in \cite{MR2484736}; for further details on their GIT construction, see also \cite{MR3202702}.

In this paper, we instead take an approach that combines Alper's theory of adequate moduli spaces with determinantal line bundle techniques to provide a construction of projective moduli spaces of semistable representations of an acyclic quiver that avoids the methods of GIT.
Our method follows the blueprint set out in recent papers on the projectivity of moduli spaces of stable curves \cite{sponge-Mgbar,MR1064874}
and semistable vector bundles on curves \cite{sponge}.
Let us recall the basic steps in these approaches:
\begin{enumerate}
  \item[(1)] Interpret the moduli problem as an algebraic stack $\cM$ of finite type.
  \item[(2)] Prove that $\mathcal{M}$ admits an adequate moduli space $M$, which is a proper algebraic space.
  \item[(3)] Find a line bundle on $\mathcal{M}$ which descends to an ample line bundle on $M$.
\end{enumerate}
In the case of stable curves \cite{sponge-Mgbar},
the second step uses the Keel--Mori theorem for proper Deligne--Mumford stacks,
whereas the analogue for the moduli space of semistable vector bundles on a curve \cite{sponge} relies on
the recent existence result
on adequate moduli spaces of algebraic stacks \cite{1812.01128} due to Alper--Halpern-Leistner--Heinloth.

In the case of stable curves,
the construction of the ample line bundle
in the third step above
is due to Koll\'ar \cite{MR1064874},
who considers the determinant of the direct image of a relative pluricanonical bundle
on the universal family over the moduli stack.
The case of vector bundles on a curve \cite{sponge},
which similarly centers around proving ampleness of a certain determinantal line bundle
constructed from the universal vector bundle,
follows arguments in Esteves \cite[Section~5]{MR1695802} and Esteves--Popa \cite[Section~3]{esteves.popa:2004:veryampleness}
improving upon the original GIT-free approach of Faltings \cite{MR1211997}.

There are several good reasons for pursuing such an approach:
i) it provides an intrinsic moduli-theoretic proof,
ii) it illustrates the theory of adequate moduli spaces,
and iii) it can yield insight into how to
construct projective moduli spaces in situations where GIT cannot be applied. It should be noted however that this approach does not automatically yield projectivity, and we need to rely on the properties of the specific moduli problem.

\paragraph{Main results.}
Working over a noetherian base scheme $S$, we will moduli-theoretically define the stack $\modulistack{d,S}$ parameterizing families of representations of a quiver $Q$ of dimension vector $d$. 
To a stability function $\theta$ we associated the open substack $\modulistack[\semistable{\theta}]{d,S} \subseteq \modulistack{d,S}$ of $\theta$-semistable representations and construct a natural determinantal line bundle $\cL_\theta$ on it.

We then give moduli-theoretic proofs that these stacks are $\Theta$-reductive and S-complete
and deduce that both stacks admit adequate moduli spaces by applying the existence result of Alper--Halpern-Leistner--Heinloth \cite{1812.01128}.
A reader who is used to working in characteristic~$0$ can read the word ``adequate'' as ``good'', since the two notions only differ in positive characteristic.

The following is a summary of the main results of the paper.
\begin{alphatheorem}
  \label{theorem:main-thm}
  Let~$Q$ be an acyclic quiver
  and let~$S$ be a noetherian scheme.
  \begin{enumerate}[label=\normalfont(\roman*)]
    \item\label{enumerate:main-proper}
      The stack $\modulistack[\semistable{\theta}]{d,S}$ of $\theta$-semistable quiver representations
      over $S$
      admits an adequate moduli space $\modulispace[\semistable{\theta}]{d,S}$ which is proper over $S$.

    \item\label{enumerate:main-ample}
      The line bundle $\mathcal{L}_\theta$ on $\modulistack[\semistable{\theta}]{d,S}$
      descends to a relatively ample line bundle $L_\theta$ on $\modulispace[\semistable{\theta}]{d,S}$.
      In particular, $\modulispace[\semistable{\theta}]{d,S}$ is 
      projective
      over $S$.
  \end{enumerate}
\end{alphatheorem}
The last part of \ref{enumerate:main-ample} is of course well-known (when working over a field usually),
but the novelty here is that we illustrate how it can be obtained using the modern methods of algebraic stacks and adequate moduli spaces.

\medskip
Our methods yield partial results also when $Q$ has oriented cycles. In this case we require that the stability function $\theta$ is of the form $-\langle \blank, \beta \rangle$ for a dimension vector $\beta$ with $\langle d, \beta \rangle =0$, where $\langle \blank, \blank \rangle$ denotes the Euler pairing; this condition holds for every stability function when $Q$ is acyclic by \cref{lemma:stability-conditions-are-dim-vectors}.
Under this assumption we prove the analogues of \ref{enumerate:main-proper} and \ref{enumerate:main-ample}:
\begin{enumerate}[label=\normalfont(\roman*$^\prime$)]
  \item
    \label{enumerate:main-separated}
    {\it The stack $\modulistack[\semistable{\theta}]{d,S}$ of $\theta$-semistable quiver representations admits a separated adequate moduli space $\modulispace[\semistable{\theta}]{d,S}$, and the semisimplification map $\modulispace[\semistable{\theta}]{d,S} \to \modulispace{d,S}$ on adequate moduli spaces is proper.}
  \item
    \label{enumerate:main-semiample}
    {\it When $S = \Spec{k}$ is the spectrum of a field,
    the line bundle $\mathcal{L}_\theta$ on $\modulistack[\semistable{\theta}]{d,k}$ 
    descends to a semiample line bundle $L_\theta$ on $\modulispace[\semistable{\theta}]{d,k}$.}
\end{enumerate}

In this setting, there are two limitations to our approach. First, our methods are unable to handle the true analogue of \ref{enumerate:main-ample} in the non-acyclic case --
that $\modulispace{d,k}$ is affine and
the semisimplification map $\modulispace[\semistable{\theta}]{d,k} \to \modulispace{d,k}$ is projective; this statement usually follows from the methods of GIT. Second, we can only develop \ref{enumerate:main-semiample} over a field.

It should moreover be possible to remove the condition on~$\theta$ in the non-acyclic case.
It is currently used to produce sections of determinantal line bundles, which we use to check local reductivity of the stack of semistable representations in order to apply the existence criteria \cite[Theorem 5.4]{1812.01128} in arbitrary characteristic.
If~$S$ has characteristic~$0$, we can use good moduli spaces instead of adequate moduli spaces,
and the required local reductivity follows from \cite{MR4088350,1812.01128}. We also use determinantal sections in the proof of semiampleness.

\medskip
We prove \ref{enumerate:main-proper} and \ref{enumerate:main-separated} in
\cref{corollary:Md-ss-separated-gms} and
\cref{proposition:gms-map-proper}.
The main projectivity result \ref{enumerate:main-ample} is given in \cref{theorem:projectivity}.
We prove \ref{enumerate:main-semiample} in \cref{proposition:semiampleness}.

Our proof that the semisimplification map $\modulispace[\semistable{\theta}]{d,S} \to \modulispace{d,S}$ is proper
follows an adaptation of Langton's semistable reduction argument for semistable coherent sheaves \cite{MR0364255},
see~\cref{proposition:langton}.
It says that if a representation of~$Q$ over a discrete valuation ring~$R$ has semistable generic fiber,
then there exists a subrepresentation which agrees at the generic point
such that its special fiber is semistable.
For an acyclic quiver,
we moreover argue moduli-theoretically that
$\modulispace[\semistable{\theta}]{d,S}$ is proper over $S$ 
(\cref{corollary:adequate-proper-arbitrary-base}) and that
the adequate moduli space $\modulispace{d,S}$ of all representations
is isomorphic $S$ (\cref{proposition:acyclic-full-ams-is-S}).

\medskip

The proof of projectivity in \cref{theorem:main-thm}~\ref{enumerate:main-ample} (when~$Q$ is acyclic)
is obtained by bootstrapping from~$\Spec k$ to~$\Spec\ZZ$ and finally to an arbitrary base~$S$.
The main idea over a field~$k$ is
to show that the line bundle~$\cL_\theta$ is semiample,
and that the induced map~$\modulispace[\semistable{\theta}]{d,k} \to \PP_k^n$
is finite,
and thus the proper algebraic space $\modulispace[\semistable{\theta}]{d,k}$ is in fact a projective variety. Instead of appealing to methods from GIT,
we give a new moduli-theoretic proof of global generation of a power of $\mathcal{L}_\theta$
inspired by the approach of Esteves \cite{MR1695802}
and Esteves--Popa \cite{esteves.popa:2004:veryampleness} for moduli of vector bundles on curves
using dimension-counting techniques.

Let us outline how we produce sections of $\cL_\theta$. Since $Q$ is acyclic, we can write $\theta = -\langle \blank, \beta \rangle$ for a dimension vector $\beta$ with $\langle d, \beta \rangle =0$. For $m > 0$ and a representation $N$ of dimension vector $m \beta$,
we define a 2-term complex $\mathcal{E}^\bullet_N$ on $\modulistack{d,S}$
whose associated determinant line bundle is $\mathcal{L}_\theta^{\otimes m}$ and,
since $\langle d, \beta \rangle = 0$,
comes with a section $\sigma_N$ which is nonzero at a rep\-re\-sen\-ta\-tion~$M \in \modulistack{d,S}$
if and only if $\Hom(M,N) = \Ext(M,N) = 0$.

The sections constructed in this way are often called \emph{determinantal semi-invariants}
in the representation theory literature,
and were first studied by Schofield~\cite{MR1113382}.
A key result is that determinantal  semi-invariants
span the global sections of powers of determinantal line bundles;
GIT-based proofs are due to
Derksen--Weyman \cite{MR1758750},
Domokos--Zubkov \cite{MR1825166},
and Schofield--Van den Bergh \cite{MR1908144}.

\medskip
Interestingly,
this approach enables us to produce new effective bounds for global generation,
by keeping track of the estimates in the dimension counting.
This is analogous to the effective bounds from \cite{MR1695802,esteves.popa:2004:veryampleness} for moduli of vector bundles on curves.
Moreover, the bound is independent of the orientation of the quiver.
We show the following result combining \cref{proposition:effective-bounds} and \cref{proposition:semiampleness}.

\begin{alphatheorem}
  \label{theorem:effective}
  Let $k$ be a field and let~$Q$ be a quiver.
  Let $\lambda$ be the negative of the minimal eigenvalue of the Tits form.
  If $m$ is a positive integer greater than~$\lambda\|d\|^2$,
  then $\cL_\theta^{\otimes m}$ is globally generated on $\modulistack[\semistable{\theta}]{d,k}$.
\end{alphatheorem}

In \cref{remark:effective-bpf} we comment further on the context for this result.

\paragraph{Similarities and differences with vector bundles on curves.}

There is an important parallel between
moduli of semistable quiver representations
and moduli of semistable vector bundles on a smooth projective curve:
both parameterize objects in a category of global dimension~1 and have smooth moduli stacks.
Some aspects of this dictionary, including similarities between their constructions via GIT and symplectic reduction (over $k = \mathbb{C}$),
are described in \cite{MR3882963}.

However, we also see several instances where this dictionary breaks down.
One example, mentioned above, is the preservation of stability under natural dualities:
Serre duality preserves stability and semistability of vector bundles on curves,
but the Auslander--Reiten translations are only a partial duality, and,
although they preserve semistability,
they only preserve stability under some additional assumptions (see \cref{lemma:auslander-reiten-stable}).

As a second example, the theory of elementary Hecke modifications for vector bundles on curves
does not have an immediate analogue for quiver representations (lacking a notion of torsion sheaves),
meaning that the proof in \cite{sponge} cannot be directly translated to quiver representations.
We can nevertheless find a close enough analogue of Hecke modifications (as in the proof of \cref{theorem:separation-base-case}),
so that we can stay close to the global structure of the proof for vector bundles.

A third difference is that for a curve of genus $g\geq 2$,
the moduli space of stable vector bundles of any rank and degree is non-empty with dimension determined by the Euler pairing,
the analogue of which fails for quiver representations -- the stable locus may well be empty,
in which case the dimension of the semistable locus is difficult to control.
This results in us adopting an alternative approach to the dimension counts for quiver representations in \cref{section:vanishing-results}.

\paragraph{Structure of the paper}

In \cref{section:quiver-background} we recall quiver representations,
their stability properties, and the Auslander--Reiten translations.
In \cref{section:construction-moduli-stack} we construct the stack of quiver representations from the ground up
and prove its algebraicity moduli-theoretically, as well as explain how to produce determinantal line bundles and their sections.
\Cref{section:vanishing-results} is the technical heart of the paper, where we prove the key vanishing results required for later sections.
In \cref{section:good-moduli-spaces} we
construct adequate moduli spaces for stacks of representations
by verifying the conditions for the existence result of \cite{1812.01128}.
Here we also discuss semistable reduction.
The main results, namely the projectivity of the adequate moduli space for~$Q$ acyclic
as well as the effective bounds on global generation,
are finally established in \cref{section:projectivity}.
In \cref{section:appendix} we explain how one can obtain the same results
using Halpern--Leistner's theory of stability for stacks; this gives an alternative modern proof, albeit one that still relies on methods of GIT and only works in characteristic $0$.

\paragraph{Acknowledgements}
This project started during the ``Moduli problems beyond geometric invariant theory'' workshop at the American Institute in Mathematics, in 2021.
We would like to thank Jarod Alper, Andres Fernandez Herrero, Daniel Halpern-Leistner, Jochen Heinloth and Markus Reineke for interesting discussions. 

P.B.~was partially supported by a postdoctoral fellowship from the Research Foundation--Flanders (FWO) in the earlier stages of this project.
T.T.~was partially supported by the Knut and Alice Wallenberg foundation grant no. 2019.0493, as well as the Swedish Research Council grant no. 2016-06596 during the research program \textit{Moduli and Algebraic Cycles} at Institut Mittag-Leffler.

\section{Background on quiver representations}
\label{section:quiver-background}

We recall here some terminology from the theory of quiver representations that will be used throughout this paper. We refer to \cite{MR3727119} for more details.

\subsection{Quiver representations}
\label{sec:quiverrepr}

A \textit{quiver} $Q=(Q_0,Q_1, s,t)$ is a finite directed graph with vertex set $Q_0$, arrow set $Q_1$, and maps $s,t \colon Q_1 \to Q_0$ that assign to each arrow its source and target. A \textit{path} in a quiver is a sequence of composable arrows; that is, the target of the previous arrow is the source of the next arrow. We formally include a path of length 0 at every vertex $i \in Q_0$. A vertex $i \in Q_0$ is a \textit{source} (respectively a \textit{sink}) of $Q$ if there are no arrows whose target (respectively source) equals $i$.
We will assume that~$Q$ is connected throughout, and in \cref{section:construction-moduli-stack} we will see why this is a harmless assumption to make.

An \textit{oriented cycle} in a quiver is a path of positive length starting and ending at the same vertex; a special case is a \textit{loop}, which is an arrow whose source and target are equal. A quiver is \textit{acyclic} if it has no oriented cycles. Note that if $Q$ is acyclic, there is an \emph{admissible ordering} of $Q_0$, meaning that $i < j$ whenever there is an arrow $i \to j$; in particular, an acyclic quiver has both a source and a sink.

Given a field $k$, a $k$-\textit{representation} of $Q$ is a tuple $M=((M_i)_{i \in Q_0}, (M_a)_{a \in Q_1})$ consisting of finite-dimensional $k$-vector spaces $M_i$ for each vertex $i \in Q_0$ and $k$-linear maps $M_a \colon M_{s(a)} \to M_{t(a)}$ for each arrow $a \in Q_1$.
The \textit{dimension vector} $\dimvect(M) \in \NN^{Q_0}$ of $M$ is the tuple $(\dim(M_i))_{i \in Q_0}$.
A \textit{morphism} of $k$-representations $\phi\colon M \to N$ is a tuple of $k$-linear maps $(\phi_i\colon M_i \to N_i)_{i\in Q_0}$ such that $\phi_{t(a)} \circ M_a = N_a \circ \phi_{s(a)}$ for every arrow $a \in Q_1$.
The repre\-sen\-tations of $Q$ over $k$ form an abelian category $\Catrep$.
If $k \subset k'$ is a field extension, there is a base change functor $(\blank) \otimes_k k'\colon \mathrm{rep}_k Q \to \mathrm{rep}_{k'} Q$ that
preserves dimension vectors.

A representation $M \neq 0$ is called \textit{simple} if it has exactly two subrepresentations $0$ and $M$, and \textit{semisimple} if it is the direct sum of simple representations. Any representation has a filtration by simple representations, making $\Catrep$ a category of \textit{finite length}. For each vertex $i \in Q_0$, there is a simple representation $S(i)$ with $S(i)_i = k$ and $S(i)_j = 0$ for $j \neq i$. If $Q$ is acyclic, these are the only simple representations.

A representation $M \neq 0$ is called \textit{indecomposable} if it cannot be written as the direct sum of two nonzero subrepresentations.
By the Krull--Remak--Schmidt theorem, every representation can be written as a direct sum of indecomposable subrepresentations in an essentially unique way; we will use this fact without further mention.

The \textit{Euler pairing} or \textit{Euler form} of $Q$
is $\langle\blank,\blank \rangle\colon \ZZ^{Q_0} \times \ZZ^{Q_0} \to \ZZ$ where
\[
    \langle \alpha,\beta \rangle\colonequals \sum_{i \in Q_0} \alpha_i\beta_i - \sum_{a \in Q_1} \alpha_{s(a)}\beta_{t(a)}.
\]
For representations $M$ and $N$ of $Q$, we write $\langle M,N \rangle\colonequals \langle \dimvect(M), \dimvect(N) \rangle$.
Given an ordering $Q_0 = \{1, 2, \ldots, n\}$ of the vertices of $Q$, one has an isomorphism $\mathbb{Z}^{Q_0} \cong \mathbb{Z}^n$. Hence the Euler pairing is represented by a matrix $A \in \text{Mat}_n(\mathbb{Z})$, called the \textit{Euler matrix}, which satisfies
\[
  \langle \alpha,\beta \rangle = \alpha^\mathrm{T} A\beta \qquad \mbox{for all}\;\; \alpha,\beta \in \ZZ^{n}.
\]
If $Q$ is acyclic, then for an admissible ordering of the vertices, the Euler matrix is upper unitriangular and hence invertible over $\ZZ$.
In particular, when $Q$ is acyclic, the Euler pairing is perfect.

\subsection{Modules over the path algebra}

Let $k$ be a field. The \emph{path algebra} of $Q$ is the $k$-algebra $kQ$ with basis given by all paths in $Q$, including a path $\epsilon_i$ of length $0$ at each vertex $i$,
and multiplication given by the concatenation of paths; see for example~\cite[Section~1.5]{MR1758750}.
The category $\Catrep$ of $k$-representations of $Q$ is equivalent to
the category $\Catmod$ of finite-dimensional left modules over $kQ$.
The path algebra is finite-dimensional if and only if $Q$ is acyclic.
If~$Q$ is not acyclic we will also need to consider the category $\CatRep$
of not necessarily finite-dimensional representations,
which is equivalent to the category $\CatMod$ of all left modules over the path algebra.

For each $i \in Q_0$, there are projective and injective representations $P(i) = kQ\epsilon_i$ and $I(i) = (\epsilon_ikQ)^*$; thus
for each $j \in Q_0$
\begin{itemize}
  \item $P(i)_j$ is the $k$-vector space with basis the set of paths from $i$ to $j$,
  \item $I(i)_j$ is the $k$-vector space dual to the one whose basis is the set of paths from $j$ to $i$.
\end{itemize}

The representation $P(i)$ is finite-dimensional if and only if there is no path from $i$ to any vertex in an oriented cycle. Similarly, $I(i)$ is finite-dimensional if and only is there is no path from any vertex in an oriented cycle to $i$.

We will focus on the case of projective modules, as injective modules are the dual notion, see for example~\cite[Section~2.2]{MR3727119} and \cref{subsection:auslander-reiten-translation} below.
The projective representations $P(i)$ are indecomposable, and in the case when $Q$ is acyclic, these are the only indecomposable projective representations of $Q$ up to isomorphism.
For any representation $M$ of $Q$, we have a canonical isomorphism
\begin{equation}
  \label{eq:hom-out-of-projective}
  \Hom(P(i),M) \cong M_i.
\end{equation}

The category $\Catrep$ is \textit{hereditary},
meaning that any subrepresentation of a projective representation is again projective
and dually any quotient of an injective representation is injective.
In particular, $\Catrep$ has homological dimension at most one,
so we will write $\Ext(M,N)$ for $\Ext^1(M,N)$.
In fact, any representation $M$ has a \textit{canonical projective resolution}
\begin{equation}
  \label{eq:canonical-projective-resolution}
  0 \to \bigoplus_{a \in Q_1} M_{s(a)} \otimes P(t(a)) \to \bigoplus_{i \in Q_0} M_i \otimes P(i) \to M \to 0
\end{equation}
and an analogous canonical injective resolution,
working in $\CatRep$ in case $Q$ is not acyclic.
Applying the functor $\Hom(\blank, N)$ to \eqref{eq:canonical-projective-resolution}
and using \eqref{eq:hom-out-of-projective} gives the useful exact sequence
\begin{equation}
  \label{eq:exact-seq-hom-ext}
  0 \to \Hom(M,N) \to \bigoplus_{i \in Q_0} \Hom_k(M_i,N_i) \to \bigoplus_{a \in Q_1} \Hom_k(M_{s(a)},N_{t(a)}) \to \Ext(M,N) \to 0,
\end{equation}
where the middle morphism is given by $(f_i)_{i \in Q_0} \mapsto (f_{t(a)} \circ M_a - N_a \circ f_{s(a)})_{a \in Q_1}$.
From this we deduce that the Euler pairing computes the Euler characteristic:
\[
  \langle M,  N \rangle = \dim \Hom(M,N) - \dim \Ext(M,N).
\]
In particular, $\langle P(i),M \rangle = \dim(M_i)$
since $\Ext(P(i),M) = 0$ as $P(i)$ is projective.

\subsection{Stability of representations}
\label{subsection: background on stability of reps}

Following \cite{MR1315461,reineke:2003:harder} we introduce a standard notion of stability for a representation of $Q$.
A \emph{stability function}
on $Q$ is a linear map $\theta \colon \ZZ^{Q_0} \to \ZZ$.
By convention we will write $\theta(M)$ instead of $\theta(\dimvect M)$ for a representation $M$.

Given a stability function $\theta$, a $k$-representation $M$ of $Q$ is called
  \begin{itemize}
    \item \emph{$\theta$-semistable} if $\theta(M) = 0$ and $\theta({M'}) \leq 0$ for every subrepresentation $M' \subseteq M$;
    \item \emph{$\theta$-stable} if it is $\theta$-semistable and it has exactly two subrepresentations $M' \subseteq M$ with $\theta(M') = 0$, namely $M' = 0$ and $M' = M$ with $M \neq 0$;
    \item \emph{geometrically $\theta$-stable} if $M\otimes_k k'$ is stable for every field extension $k'/k$;
    \item \emph{$\theta$-polystable} if it is a direct sum of $\theta$-stable representations;
    \item \emph{geometrically $\theta$-polystable} if $M\otimes_k k'$ is polystable for every field extension $k'/k$
  \end{itemize}

\begin{remark}\label{remark:stability-under-field-extension}
  There is no need for a notion of geometric $\theta$-semistability
  since a representation $M$ is $\theta$-semistable if and only if $M \otimes_k k'$ is for some field extension $k\subset k'$ \cite[Proposition~2.4]{MR4117060}. A representation $M$ is geometrically $\theta$-stable if either one of $M \otimes_k k^s$ or $M \otimes_k \bar{k}$ is $\theta$-stable, where $k^s$ and $\bar{k}$ denote a separable and an algebraic closure of $k$ respectively, see \cite[Corollary~2.12]{MR4117060}. In particular, over an algebraically closed field, $\theta$-stability and geometric $\theta$-stability coincide. Over a perfect field, $\theta$-polystability and geometric $\theta$-polystability coincide, although a field extension may give rise to more stable factors \cite[Lemma 4.2]{MR2250021}.
\end{remark}

\begin{example}\label{example:Jordan-quiver-polystability-not-preserved-field-ext}
This example shows that polystability may not be preserved under field extension when the base field is not perfect.
Consider the Jordan quiver~$Q$:
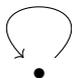
\begin{equation*}
  \begin{tikzcd}[every label/.append style={font=\small}]
    \bullet \arrow[out=45, in=135, loop]
  \end{tikzcd}
\end{equation*}
Isomorphism classes of representations of $Q$ are given by square matrices up to conjugation. We consider the trivial stability function $\theta = 0$ and give an example of a 2-dimensional representation $M$ of $Q$ which is stable over $\FF_2(t)$ but not polystable over $\FF_2(\sqrt t)$. The representation $M$ is given over $\FF_2(t)$ by the matrix
\[
\begin{bmatrix}
1 & t \\ 1 & 1
\end{bmatrix}
\]
whose only eigenvalue is $1+\sqrt t$, so $M$ is simple and in particular stable over $\FF_2(t)$.
Over $\FF_2(\sqrt t)$, this matrix is similar to
\[
\begin{bmatrix}
1 & 1 + \sqrt t \\ 1 & \sqrt t
\end{bmatrix}
\begin{bmatrix}
1 & t \\ 1 & 1
\end{bmatrix}
\begin{bmatrix}
\sqrt t & 1+\sqrt t \\ 1 & 1
\end{bmatrix}
=
\begin{bmatrix}
1+\sqrt t & 1 \\ 0 & 1+\sqrt t
\end{bmatrix}
\]
and since it is a $2\times 2$ Jordan block, it gives a semistable representation that no longer splits into a direct sum of simple
representations.
\end{example}

Every $\theta$-semistable representation $M$ has a \textit{Jordan-H\"older filtration}
\[
  0 = M^0 \subsetneq M^1 \subsetneq M^2 \subsetneq \cdots \subsetneq M^{r-1} \subsetneq M^r = M
\]
with the property that the quotients $M^\ell/M^{\ell-1}$ are $\theta$-stable for $\ell=1,\ldots,r$.
The filtration is not unique, but the isomorphism type of the \emph{associated graded}
representation
\[
  \gr{M} \coloneqq \bigoplus_{\ell=1}^r M^\ell/M^{\ell-1}
\]
is independent of the filtration.

To define a second type of filtration,
we introduce a slope function $\mu_\theta$
which associates to a dimension vector $d \in \mathbb{N}^{Q_0}\setminus\{0\}$ the rational number
\[
  \mu_\theta(d) \coloneqq \frac{\theta(d)}{\sum_{i\in Q_0} d_i}.
\]
We say that a representation $M$ is $\mu_\theta$-semistable
if $\mu_\theta(M') \leq \mu_\theta(M)$ for every subrepresentation $M' \subseteq M$.
Observe that when $\theta(M)=0$,
stability with respect to $\mu_\theta$ and $\theta$ coincide.
Moreover, an arbitrary representation $M$ has a \textit{Harder-Narasimhan filtration}
\[
  0 = M^0 \subsetneq M^1 \subsetneq M^2 \subsetneq \cdots \subsetneq M^{r-1} \subsetneq M^r = M
\]
such that $M^i/M^{i-1}$ is $\mu_\theta$-semistable for every $i=1,\ldots,r$ and
\[
  \mu_\theta(M^1/M^0) > \mu_\theta(M^2/M^1) > \ldots > \mu_\theta(M^r/M^{r-1}).
\]
See for example \cite[Section~4]{MR2484736}.
This filtration is unique and the representation $M^1$ is called the \textit{maximally destabilizing} subrepresentation. More generally, one can define different slope functions by giving positive weights to the dimensions in the denominator; these give the same notion of semistability for dimension vectors $d$ for which $\theta(d) = 0$ but possibly different Harder--Narasimhan filtrations \cite[Section~5.2]{MR3261979}.

We record some elementary properties of semistable representations.

\begin{proposition}
  \label{proposition:stability-properties}
  Let $M$ and $N$ be $\mu_\theta$-semistable $k$-representations of the same slope.
  \begin{enumerate}[label=(\roman*)]
    \item If $f\colon M \to N$ is any morphism, then $\ker(f)$, $\im(f)$ and $\coker(f)$ are $\mu_\theta$-semistable.
    \item If $M$ and $N$ are $\theta$-stable, then any nonzero morphism $M \to N$ is an isomorphism. In particular, if $M$ is geometrically $\theta$-stable, then the canonical map $k \to \End(M)$ is an isomorphism.
  \end{enumerate}
\end{proposition}

Given $\alpha, \beta \in \ZZ^{Q_0}$ we define stability functions $\theta_\alpha,\eta_\beta\colon\ZZ^{Q_0} \to \ZZ$ by
\begin{equation}\label{eq:theta-alpha-and-eta-beta}
    \theta_\alpha(d) \colonequals \langle \alpha, d \rangle  \qquad \text{ and } \qquad \eta_\beta(d) \colonequals - \langle d, \beta  \rangle.
\end{equation}
When $Q$ is an acyclic quiver, the matrix of the Euler form is invertible and therefore defines an isomorphism $\ZZ^{Q_0} \cong (\ZZ^{Q_0})^\vee$. This implies the following.

\begin{lemma}\label{lemma:every-theta-is-theta-alpha}
  Suppose $Q$ is an acyclic quiver and let $\theta\colon \ZZ^{Q_0} \to \ZZ$ be a stability function.
  \begin{enumerate}[label=(\alph*)]
    \item \label{enumerate:every-theta-is-theta-alpha}
      There is a unique $\alpha \in \ZZ^{Q_0}$ such that $\theta = \theta_\alpha = \langle \alpha,\blank \rangle$, given by $\alpha_i = \theta(I(i))$.
    \item \label{enumerate:every-theta-is-theta-beta}
      There is a unique $\beta \in \ZZ^{Q_0}$ such that $\theta = \eta_\beta = - \langle \blank,\beta \rangle$, given by $\beta_i = -\theta(P(i))$.
  \end{enumerate}
\end{lemma}

The next lemma shows how stability functions for acyclic quivers
are in fact given by dimension vectors.
\begin{lemma}
  \label{lemma:stability-conditions-are-dim-vectors}
  Suppose $Q$ is an acyclic quiver and let $\theta\colon\ZZ^{Q_0} \to \ZZ$ be a stability function
  for which there exists a semistable representation $M$ such that $\operatorname{supp} M = Q_0$.
  \begin{enumerate}[label=(\alph*)]
    \item\label{enumerate:stability-conditions-are-dim-vectors-alpha}
      Let $\alpha\in\ZZ^{Q_0}$ be the unique vector such that~$\theta=\theta_\alpha$.
      Then~$\alpha\in\NN^{Q_0}$.
    \item\label{enumerate:stability-conditions-are-dim-vectors-beta}
      Let $\beta\in\ZZ^{Q_0}$ be the unique vector such that~$\theta=\eta_\beta$.
      Then~$\beta\in\NN^{Q_0}$.
  \end{enumerate}
\end{lemma}

\begin{proof}
  We show that the entries of $\beta$ are non-negative.
  The proof for $\alpha$ is analogous.
  As $\beta_i = -\theta(P(i))$,
  we need to show that $\theta(P(i)) \leq 0$ for every $i$.
  We choose an admissible ordering of the vertices of $Q$
  and write $Q_0 = \{1, \dots, n\}$.
  We show the claim by descending induction on $i \in \{1, \dots, n\}$.
  For $i=n$, the vertex $i$ is a sink and therefore $P(i)$ is simple.
  Choose a nonzero vector $x \in M_i$ and let $f \in \Hom(P(i),M) \cong M_i$
  be the corresponding morphism.
  As $f \neq 0$ and $P(i)$ is simple, $f$ is injective.
  Since $M$ is $\theta$-semistable, we see that $\theta(P(i)) \leq 0$.

  Now let $i < n$.
  Again, choose any nonzero element $x \in M_i$ and consider the corresponding morphism $f\colon P(i) \to M$.
  The kernel $K$ of $f$ is again projective and a proper subrepresentation of $P(i)$, and thus $K \cong \bigoplus_{j > i} P(j)^{\oplus m_j}$ for some multiplicities $m_j \geq 0$.
  By the inductive assumption, we know that $\theta(P(j)) \leq 0$ for all $j > i$, which implies that $\theta(K) \leq 0$.
  Now $f\colon P(i) \to M$ gives rise to an injective homomorphism
  \[
    P(i)/K \hookrightarrow M
  \]
  which by $\theta$-semistability of $M$ implies that $0 \geq \theta(P(i)/K) = \theta(P(i)) - \theta(K)$, and so
  $\theta(P(i)) \leq \theta(K) \leq 0$.
\end{proof}

Without the acyclicity assumption \cref{lemma:stability-conditions-are-dim-vectors} can fail.
\begin{example}
  Consider the quiver
  \begin{equation*}
    Q\colon
    \begin{tikzcd}
      1 \arrow[r, bend left] \arrow[r, bend left=45] & 2 \arrow[l,bend left] \arrow[l, bend left=45]
    \end{tikzcd}.
  \end{equation*}
  The stability function $\theta$ given by the inner product with $(-3,3)$ equals $\theta_\alpha$ for $\alpha=(-1,1)$, and $\eta_\beta$ for $\beta=(1,-1)$.
  Note that since the Euler
  matrix~$\begin{psmallmatrix} 1 & -2 \\ -2 & 1 \end{psmallmatrix}$
  is invertible 
  over the rationals, both $\alpha$ and $\beta$ are uniquely determined, but neither of them is a dimension vector.
  Consider the dimension vector~$d=(1,1)$, so that~$\theta(d)=0$.
  There exists a~$\theta$-semistable representation with dimension vector~$d$:
  it suffices that one of the maps from vertex $2$ to vertex $1$ is non-zero.
  In this case there is only one non-trivial subrepresentation,
  which necessarily is of dimension $(1,0)$ and hence does not destabilize the representation.
\end{example}

\subsection{The Auslander--Reiten translations}
\label{subsection:auslander-reiten-translation}

The final standard construction for quivers and their representations that we need to recall is that of the Auslander--Reiten translations.
These are endofunctors of the category of quiver representations,
and their interaction with $\theta$-stability will be discussed in \cref{subsection:ar-translations-vs-semistability}.

The \emph{opposite quiver} of a quiver $Q = (Q_0, Q_1, s, t)$ is the quiver $Q^{\op}\colonequals(Q_0,Q_1,t,s)$,
where all arrows have been reversed.
The path algebra of $Q^{\op}$
is canonically isomorphic to the opposite algebra of $kQ$.
Taking the dual $k$-vector space gives a contravariant functor $\dual$
from the category $\Catmod$ to $\Catmod[Q^{\text{op}}]$,
or equivalently between the category of representations of $Q$ and $Q^{\op}$.
The duality functor $\dual$ is an antiequivalence of categories
and exchanges injective and projective modules.

Let $Q$ be an acyclic quiver.
Consider a representation $M$ as a left module over the path algebra $kQ$ and $kQ$ as a bimodule over itself.
The \emph{Auslander--Reiten translates} of~$M$ are the left $kQ$-modules
\begin{align*}
  \tau M &:= \dual\Ext(M,kQ) \\
  \tau^-M &:= \Ext(\dual M,kQ)
\end{align*}
and these constructions provide two endofunctors of the category $\Catmod \cong \Catrep$,
called the \emph{Auslander--Reiten translations}.
It follows from the construction that
\[
  \tau P = \tau^- I = 0
\]
whenever $P$ is projective and $I$ is injective.
The following proposition records the key properties we will need later;
for the proof, see \cite[Section~6.4]{MR1758750}.

\begin{proposition}\label{proposition:AR-translation}
  Let $M$ and $N$ be representations of $Q$.
  The Auslander--Reiten translations $\tau$ and $\tau^-$ satisfy the following properties:
  \begin{enumerate}[label=(\roman*)]
    \item\label{item:AR-inverse} \textbf{Partial inverse property:} We have $\tau^- \tau M \cong M$ and $\tau \tau^- N \cong N$,
      provided that $M$ has no projective summands and $N$ has no injective summands.
    \item\label{item:AR-duality} \textbf{Auslander--Reiten duality:} We have isomorphisms of $k$-vector space valued functors
    \[
      \Hom(\blank, \tau M) \cong \Ext(M,\blank)^\vee \quad \text{and}  \quad \Hom(\tau^- N, \blank) \cong \Ext(\blank,N)^\vee.
    \]
    In particular $\tau^-$ is the left adjoint to $\tau$.
    If $M$ has no projective summands and $N$ has no injective summands, we have additional isomorphisms
    \[
      \Ext(\blank, \tau M) \cong \Hom(M, \blank)^\vee \quad \text{and}  \quad \Ext(\tau^- N, \blank) \cong \Hom(\blank, N)^\vee,
    \]
      and in particular
      \[ \langle \blank, \tau M \rangle = - \langle M, \blank \rangle \quad \text{and}  \quad  \langle \tau^- N, \blank \rangle = - \langle \blank, N \rangle. \]
  \end{enumerate}
\end{proposition}

\section{Moduli stacks of representations and determinantal line bundles}
\label{section:construction-moduli-stack}

In this section, we introduce moduli stacks parameterizing representations of a quiver with fixed dimension vector and study their first properties. Throughout $d$ will denote a dimension vector for the fixed quiver $Q$.

\subsection{The moduli stack of all representations}
Let $S$ be a fixed base scheme. For
an $S$-scheme
$T$, a \emph{family $\cF$ of quiver representations} of dimension vector $d$
consists of locally free sheaves $\cF_i$ of rank $d_i$ over
$T$ for each $i \in Q_0$
and homomorphisms $\cF_a \colon  \cF_{s(a)} \to \cF_{t(a)}$ of
$\cO_T$-modules for each $a \in Q_1$. If $f\colon T' \to T$ is a morphism of $S$-schemes,
we obtain a family of representations $\cF_{T'}:=f^* \cF$ on $T'$  by pullback along $f$.
If $x \in T$ is a point, then pulling back along the inclusion of
the residue field $\Spec \kappa(x) \hookrightarrow T$ gives a $\kappa(x)$-representation
which we denote by $\cF|_x$.

\begin{definition}
\label{definition: moduli stack M_d}
  The moduli stack
  $\modulistack{d,S}$ of representations of $Q$ of dimension vector $d$
  is the category fibered in groupoids over the big \'etale site of the category of
  $S$-schemes
  whose objects are pairs
  $(T, \cF)$, where $T$ is an $S$-scheme
  and $\cF$ is a family of representations of $Q$ of dimension vector $d$ over
  $T$.
  A morphism
  $(T', \cF') \to (T, \cF)$ is the data of a morphism
  $f \colon T' \to T$ of $S$-schemes
  together with morphisms $\phi_i\colon \cF_i \to f_*\cF'_i$ of
  $\cO_T$-modules such that the squares
  \begin{equation*}
    \begin{tikzcd}
      \cF_{s(a)} \arrow[r, "\phi_{s(a)}"] \arrow[d, swap, "f_* \cF'_a"] & f_* \cF'_{s(a)} \arrow[d, "\cF_a"] \\
      \cF_{t(a)} \arrow[r, swap, "\phi_{t(a)}"] & f_* \cF'_{t(a)}
    \end{tikzcd}
  \end{equation*}
  commute for every $a \in Q_1$ and whose adjoints are isomorphisms $f^*\cF_i \to \cF'_i$.
\end{definition}
We will frequently omit one of the subscripts in $\modulistack{d,S}$ when either the base scheme or the dimension vector is clear from the context. In \cref{section:vanishing-results} and \cref{section:projectivity}, we will take our base scheme to be $S = \Spec A$ for a ring $A$, in which case we may denote $\modulistack{d,S}$ by $\modulistack{d,A}$ or simply $\modulistack{A}$. Similarly, when we below define the substacks $\modulistack[\semistable{\theta}]{d,S}$, the moduli spaces $\modulispace{d,S}$ and $\modulispace[\semistable{\theta}]{d,S}$, and the representation space $\Rep{d,S}$, we apply that same conventions for the subscripts.

We will make the harmless assumption that~$Q$ is connected: 
if~$Q = Q' \sqcup Q''$
with corresponding decomposition~$d = (d', d'')$
with $d' \in \NN^{Q'_0}, d'' \in \NN^{Q''_0}$, then
$\modulistack{d,S}$
is isomorphic to the product of moduli stacks of representations of $Q'$ and $Q''$ of dimension vectors $d'$ and $d''$, respectively,
and similarly for all other constructions.

The stack $\modulistack{d,S}$ commutes with base change, meaning that if $S' \to S$ is a morphism of schemes, then we have a cartesian diagram
  \begin{equation*}
    \begin{tikzcd}
      \modulistack{d,S'} \arrow[r] \arrow[d] & \modulistack{d,S} \arrow[d] \\
      S' \arrow[r] & S
    \end{tikzcd}
  \end{equation*}
where in the top horizontal map, a family of representations $\cF$ on $T \to S'$ is viewed as a family on $T \to S$ via the composition $T \to S' \to S$. In particular, the stack $\modulistack{d,\ZZ}$ over the final object $\Spec \ZZ$ is universal in the sense that for any scheme $S$, the stack $\modulistack{d,S}$ is obtained by base change from $\modulistack{d,\ZZ}$ by the structure morphism $S \to \Spec \ZZ$.

There is a universal family of representations $\cF^\univ$ of dimension vector $d$ on the stack $\modulistack{d,S}$.
If $T$ is
an $S$-scheme and $\cF$ is a family of representations of dimension vector $d$ on $T$,
there exists a unique morphism $f\colon T \to \modulistack{d,S}$ such that $\cF \cong f^*\cF^\univ$.

\begin{proposition}
  \label{proposition:affine-diagonal}
  The diagonal of the stack $\modulistack{d,S}$ is represented by affine
  $S$-schemes.
\end{proposition}

\begin{proof}
  We follow a similar argument in \cite[Tag \href{https://stacks.math.columbia.edu/tag/08K9}{08K9}]{stacks-project}. Let $\cF$ and $\cG$ be two representations over
  an $S$-scheme $T$ and consider the $2$-cartesian diagram
  \begin{equation*}
    \begin{tikzcd}
      \IIsom(\cF, \cG) \arrow[r] \arrow[d] & T \arrow[d, "(\cF{,}\cG)"] \\
      \modulistack{d,S} \arrow[r, "\Delta"] & \modulistack{d,S} \times \modulistack{d,S}
    \end{tikzcd}
  \end{equation*}
  The fiber product $\IIsom(\cF, \cG)$
  is the functor associating to a $T$-scheme $U$
  the set of isomorphisms $\cF_U \cong \cG_U$ of representations of $Q$ over $U$.
  This functor is represented by an affine $T$-scheme constructed as follows.
  First, for each $i \in Q_0$,
  the functor of maps $\cF_i \to \cG_i$ of $\cO_T$-modules
  is represented by the affine $T$-scheme
  \[
    \mathbb{V}_i \coloneqq \relSpec_{\cO_T}\Sym^\bullet {\lHom_{\cO_T}(\cF_i, \cG_i)^\vee}.
  \]
  Next, the functor parameterizing homomorphisms of representations $\cF \to \cG$
  is represented by the closed subscheme
  \[
    \HHom(\cF, \cG) \coloneqq \relSpec_{\cO_T} \Sym^\bullet \cK \subseteq \prod_{i \in Q_0} \mathbb{V}_i
  \]
  where $\cK\colonequals\coker{\phi^\vee}$
  is the cokernel of the morphism dual to the map
  \begin{align*}
      \phi \colon \bigoplus_{i \in Q_0} \lHom(\cF_i, \cG_i)
      &\to \bigoplus_{a \in Q_1} \lHom(\cF_{s(a)}, \cG_{t(a)}) \\
      \left( \alpha_i \right)_{i\in Q_0}
      &\mapsto
      \left( \cG_a \circ \alpha_{s(a)} - \alpha_{t(a)} \circ \cF_a \right)_{a \in Q_1}
  \end{align*}
  Finally, $\IIsom(\cF, \cG)$ is represented by the base change of the morphism
  \[
    \HHom(\cF, \cG) \times_T \HHom(\cG, \cF) \to \HHom(\cF, \cF) \times_T \HHom(\cG, \cG), (\phi, \psi) \mapsto (\psi \circ \phi, \phi \circ \psi)
  \]
  along the section $\sigma\colon T \to \HHom(\cF, \cF) \times_T \HHom(\cG, \cG)$
  given by $(\mathrm{id}_\cF, \mathrm{id}_\cG)$.
  As a section of a separated morphism,
  $\sigma$ is a closed embedding,
  and so $\IIsom(\cF, \cG)$ is a closed subscheme of an affine $T$-scheme,
  hence itself affine over $T$.
\end{proof}

We will next construct a smooth atlas for $\modulistack{d,S}$.
For each vertex $i \in Q_0$, let
$V_i^d\colonequals \cO_S^{\oplus d_i}$ denote the standard free $\cO_S$-module of rank $d_i$.
For each arrow $a \in Q_1$, we set
\[
  \bbA_{a,S}= \bbA_{a}  \coloneqq \relSpec_{\cO_S} \Sym_{\cO_S}^\bullet \left(\lHom(V^d_{s(a)}, V^d_{t(a)})^\vee \right) \cong \bbA_{S}^{d_{s(a)}d_{t(a)}}.
\]
\begin{definition}
  \label{definition:representation-space}
  The space of representations of $Q$ of dimension vector $d$ is the affine
  $S$-space
  \[
    \Rep{d,S} = \Rep{d} \colonequals \prod_{a \in Q_1} \bbA_{a}.
  \]
\end{definition}
Each affine space $\bbA_{a}$ parameterizes a universal linear map $\varphi_a\colon \cO_{\bbA_a}^{\oplus s(a)} \to \cO_{\bbA_a}^{\oplus t(a)}$.
Letting $\pi_a\colon\Rep{d} \to \bbA_a$ denote the projection,
we obtain a canonical family of representations $\cF^\can$ on $\Rep{d}$ as follows.
For each vertex $i \in Q_0$, we set $\cF^\can_i = \cO_{\Rep{d}}^{\oplus d_i}$
and for each arrow $a \in Q_1$ we set
\[
  \cF^\can_a = \pi_a^*\varphi_a\colon\cO_{\Rep{d}}^{\oplus d_{s(a)}} \to \cO_{\Rep{d}}^{\oplus d_{t(a)}}.
\]
This induces a morphism $\varphi\colon\Rep{d} \to \modulistack{d}$.
We can view $\Rep{d}$ as representing the functor on
$S$-schemes that to
an $S$-scheme $T$ associates the set of representations $\cF$ over $T$
with $\cF_i = \cO_T^{\oplus d_i}$ for each $i \in Q_0$.

The $S$-group scheme
\begin{equation}
  \label{equation:definition-Gd}
  \textrm{G}_{d,S} = \textrm{G}_{d}\colonequals \prod_{i \in Q_0} \GL_{d_i}
\end{equation}
acts on $\Rep{d}$ as follows.
The $T$-points of $\textrm{G}_d$ are tuples $g = (g_i)_{i \in Q_0}$
where $g_i$ is an automorphism of $\cO_T^{\oplus d_i}$.
If $T \to \Rep{d} $ corresponds to a representation $\cF=(\cF_a\colon \cO_T^{\oplus d_{s(a)}} \to \cO_T^{\oplus d_{t(a)}})_{a \in A}$, then
the action of $g$ sends $\cF$ to the representation $\cF'=(\cF'_a = g_{t(a)} \circ \cF_a \circ g_{s(a)}^{-1})_{a \in A}$.

\begin{proposition}
  \label{proposition:smooth-cover-by-Rep}
  The morphism $\varphi\colon \Rep{d} \to \modulistack{d}$ is schematic, smooth, and surjective, and induces an isomorphism of stacks
  \begin{equation}
    \label{equation:identification-with-quotient-stack}
    [\Rep{d}/\textrm{G}_d] \cong \modulistack{d}.
  \end{equation}
  In particular, the stack $\modulistack{d}$ is smooth and of finite type over~$S$.
\end{proposition}

\begin{proof}
   By \cref{proposition:affine-diagonal}, the diagonal of $\modulistack{d}$ is affine, in particular representable by schemes,
   and this implies that $\varphi$ is schematic.

  Let $T$ be an $S$-scheme and $T \to \modulistack{d}$ a morphism corresponding to
  a representation $\cG$ over $T$.
  The fiber product $\cT \coloneqq \Rep{d} \times_{\modulistack{d}} T$
  is isomorphic to the functor that sends a morphism $g\colon U \to T$ of
  $S$-schemes
  to the set of isomorphisms $\phi = (\phi_i\colon \cO_{U}^{\oplus d_i} \xrightarrow{\sim} g^*\cG_i)_{i \in Q_0}$ of $\cO_U$-modules.
  The group $\textrm{G}_d(U)$ acts on $\cT$ by sending $\phi$ to $g \cdot \phi = (\phi_i \circ g_{i}^{-1})_{i \in Q_0}$ for $g = (g_i)_{i \in Q_0} \in \textrm{G}_d(U)$, making $\cT$ into a $\textrm{G}_d$-torsor over $T$.
  This implies that $\phi$ is smooth and surjective.
  In particular, $\modulistack{d}$ is smooth and of finite type over $S$ since $\Rep{d}$ is.

  The induced morphism $[\Rep{d}/\textrm{G}_d] \to \modulistack{d}$ is given as follows. For an
  $S$-scheme $T$, a map $T \to [\Rep{d}/\textrm{G}_d]$ corresponds to a $\textrm{G}_d$-torsor $\cT$ over $T$ and a $\textrm{G}_d$-equivariant morphism $\pi\colon \cT \to \Rep{d}$. This induces a representation $\cG'$ on $\cT$ given by $\cG'_i = \cO_{\cT}^{\oplus d_i}$ and $\cG'_a = \pi^*\cF^\can_a$. From the equivariance, we deduce descent data for the sheaves $\cG'_i$ and the maps $\cG'_a$ on the fppf cover $\cT \to T$, giving a representation $\cG$ over $T$ which determines a map $T \to \modulistack{d}$.

  A quasi-inverse is given as follows. To
  an $S$-scheme $T$ and a map $T \to \modulistack{d}$
  corresponding to a representation $\cG$,
  we assign the $\textrm{G}_d$-torsor $\cT = \prod_{i \in Q_0} \IIsom(\cO_T^{\oplus d_i}, \cG_i)$ over $T$.
  Each vector bundle $\cG_i$ trivializes after pulling back to $\cT$, so the pullback of $\cG$ provides a $\textrm{G}_d$-equivariant morphism $\cT \to \Rep{d}$.
\end{proof}

We observe that since $\Rep{d}$ is separated and $\textrm{G}_d$ is affine over $S$,
the above proposition gives a second proof that $\modulistack{d}$ has affine diagonal.
We gave a direct argument in \cref{proposition:affine-diagonal} to emphasize the moduli-theoretic point of view.

\subsection{The moduli stack of semistable representations}
\label{subsection:mst_ss_reps}

For a stability function $\theta$, a family $\cF$ of representations over a scheme $T$ is said to be a family of $\theta$-semistable, respectively geometrically $\theta$-stable, representations if for each point $x \in T$, the restriction $\cF|_x$ is a $\theta$-semistable, respectively geometrically $\theta$-stable, over the residue field $\kappa(x)$. By \cref{remark:stability-under-field-extension}, this is equivalent to requiring that given a field $k$ and a morphism $\Spec k \to T$ with image $x$, the $k$-representation $\cF|_x \otimes_{\kappa(x)} k$ is $\theta$-semistable, respectively geometrically $\theta$-stable.
\begin{definition}
  The substack $\modulistack[\semistable{\theta}]{d,S}$, respectively $\modulistack[\stable{\theta}]{d,S}$, of $\modulistack{d,S}$ is the full subcategory whose objects are those $(T, \cF) \in \modulistack{d,S}$ for which $\cF$ is a family of $\theta$-semistable, respectively geometrically $\theta$-stable, representations.
\end{definition}

Since $\theta$-semistablility and geometric $\theta$-stability are pointwise conditions on families of representations, it is clear that both moduli stacks are preserved under base change along morphisms $S' \to S$. If $\theta = 0$ is the trivial stability function, then $\modulistack[\semistable{\theta}]{d} = \modulistack{d}$ since every representation is $\theta$-semistable. We will next see that~$\modulistack[\semistable{\theta}]{d,S}\subseteq\modulistack{d,S}$  and~$\modulistack[\stable{\theta}]{d,S}\subseteq\modulistack{d,S}$ are open substacks, or equivalently that in a family $\cF$ of representations over $T$, the locus of points $x \in T$ such that $\cF|_x$ is $\theta$-semistable (respectively geometrically $\theta$-stable) is open.

For this, we first construct a relative version of a quiver Grassmannian as follows.
For a family $\cF$ of $d$-dimensional representations over $T$ and a dimension vector  $d' < d$,
consider the functor $(\Sch/T)^{\op} \to (\Sets)$ that sends $f\colon U \to T$ to the set of families of $d'$-dimensional subrepresentations $\cF' \subseteq f^*\cF$ with locally free quotient.
We claim that this functor is representable by a projective
$T$-scheme
$\Gr(d', \cF/T)$. Indeed, it is represented by a closed subscheme of the
$T$-fiber product $\prod_{i \in Q_0} \Gr_i$ of the Grassmannian bundles
$\Gr_i:=\Gr(d_i',\cF_i) \to T$ parameterizing rank $d_i'$ locally free subsheaves of $\cF_i$
with locally free quotients.
More precisely,
\[
  \Gr \coloneqq
  \Gr(d',\cF/T) \subset \prod_{i \in Q_0} \Gr_i
\]
is the scheme-theoretic intersection over all $a \in Q_1$ of the scheme-theoretic vanishing loci of the compositions
\[
  p_{s(a)}^* \cS_{s(a)} \xrightarrow{p_{s(a)}^* \iota_{s(a)}} p^* \cF_{s(a)} \xrightarrow{p^* \cF_a} p^* \cF_{t(a)} \xrightarrow{p_{t(a)}^* \eta_{t(a)}} p_{t(a)}^* \cQ_{t(a)},
\]
where $0 \to \cS_i \xrightarrow{\iota_i} \pi_i^* \cF \xrightarrow{\eta_i} \cQ_i \to 0 $ denotes the universal exact sequence on $\Gr_i$ and $p_i\colon \Gr \to \Gr_i$ and
$p\colon\Gr \to T$ are the natural projections.

\begin{lemma}\label{lemma:openness-of-stability}
  In a family of quiver representations $\cF$ of dimension vector $d$ over a scheme $T$,
  the set of $\theta$-semistable,
  respectively geometrically $\theta$-stable,
  representations is open.
\end{lemma}

\begin{proof}
  For each nonzero dimension vector $d' < d$, with
  let $\pi_{d'} \colon \Gr(d',\cF/T) \rightarrow T$ denote the relative quiver Grassmannian of $d'$-dimensional subrepresentations of $\cF$.
  Since $\Gr(d',\cF/T)$ is proper over $T$,
  the image $T_{d'} \subseteq T$ of $\pi_{d'}$ is closed.
  Consider the two subsets
  \[ Z_1 = \bigcup_{\begin{smallmatrix}0 < d' < d: \\ \theta(d') \ge \theta(d) \end{smallmatrix}} T_{d'} \quad \mathrm{and} \quad Z_2 = \bigcup_{\begin{smallmatrix}0 < d' < d: \\ \theta(d') > \theta(d) \end{smallmatrix}} T_{d'}. \]
  As there are only finitely many dimension vectors $d'$ with $d' < d$,
  the subsets $Z_1$ and $Z_2$ are closed in $T$,
  and consequently their complements
  $U_1 = T\setminus Z_1$ and $U_2 = T\setminus Z_2$
  are open. We claim that $U_1$ and $U_2$ are the loci of geometrically $\theta$-stable and $\theta$-semistable
  representations parameterized by $\cF$ respectively.
  Let $x \in T$ be a point.
  If $x \in Z_1$, then $x = \pi_{d'}(y)$ for some $d' < d$ with $\theta(d') \ge \theta(d)$
  and some $y \in \Gr(d', \cF/T)$, meaning that
  the representation $\cF|_x \otimes_{\kappa(x)} \kappa(y)$
  has a subrepresentation of dimension vector $d'$.
  Thus, $\cF|_x$ is not geometrically $\theta$-stable.

  Conversely, if $\cF|_x$ is not geometrically $\theta$-stable, then for some field extension $k/\kappa(x)$, the representation $\cF|_x \otimes_{\kappa(x)} k$ has a subrepresentation of dimension vector $d'$ with $\theta(d') \geq \theta(d)$,
  which gives a map $\Spec k \to \Gr(d', \cF/T)$ whose composite with $\pi_{d'}$ has image $x$; thus $x \in Z_1$.

  The argument for $Z_2$ is exactly the same, except by \cref{remark:stability-under-field-extension} the question of whether $\cF|_x$ is $\theta$-semistable is insensitive to extending the residue field $\kappa(x)$.
\end{proof}

\begin{remark}
  \label{remark:geometrically-stable}
  The following example shows that $\theta$-stability is \emph{not} an open condition. Consider the Jordan quiver from \cref{example:Jordan-quiver-polystability-not-preserved-field-ext} and the family of 2-dimensional representations over $\RR$ parameterized by $\mathbb{A}_\RR^1$ with coordinate $t$ given by the matrix $M_t:=\begin{psmallmatrix} t & -1 \\ 1 & 0 \end{psmallmatrix}$. For the trivial stability function $\theta =0$, all representations are semistable and $M_t \otimes_\mathbb{R} \mathbb{C}$ is polystable unless $t = \pm 2$, but over $\RR$ a point $M_t$ is stable if and only if $M_t$ has no real eigenvalues i.e. $|t| < 2$; this set is not even constructible, let alone Zariski-open.
\end{remark}

The next result directly follows from  \cref{lemma:openness-of-stability}.

\begin{corollary}\label{corollary:open-substacks}
  For a dimension vector $d$ and a stability function $\theta$, the moduli stacks
  \begin{equation}
    \label{equation:substacks}
    \modulistack[\stable{\theta}]{d,S} \subseteq \modulistack[\semistable{\theta}]{d,S}\subseteq\modulistack{d,S}
  \end{equation}
  are open substacks of $\modulistack{d,S}$. In particular, both $\modulistack[\stable{\theta}]{d,S}$ and $\modulistack[\semistable{\theta}]{d,S}$ are smooth and of finite type and have affine diagonal over $S$.
\end{corollary}

We can express these moduli stacks as quotient stacks: the open subschemes
\[
    \Rep{d,S}^{\stable{\theta}} \subseteq \Rep{d,S}^{\semistable{\theta}} \subseteq \Rep{d,S},
\]
parameterizing geometrically $\theta$-stable and $\theta$-semistable representations 
are subschemes invariant under the action of $\mathrm{G}_d$,
and we obtain identifications
\[
    \modulistack[\stable{\theta}]{d,S} \cong [\Rep{d,S}^{\stable{\theta}}/\textrm{G}_d],
    \quad
    \modulistack[\semistable{\theta}]{d,S} \cong [\Rep{d,S}^{\semistable{\theta}}/\textrm{G}_d]
\]
for the open substacks in \cref{corollary:open-substacks},
similar to \eqref{equation:identification-with-quotient-stack} in \cref{proposition:smooth-cover-by-Rep}.

\begin{remark}
  \label{remark:emptiness}
  One interesting way in which the moduli theory of
  quiver representations
  differs from that of
  vector bundles on a curve of genus $g \geq 2$
  is that the analogous moduli stacks for stable vector bundles are always non-empty.
  For any integers~$r\geq 1$ and $d$ there always exists a \emph{stable} vector bundle of rank $r$
  and degree $d$ \cite[Lemma~4.3]{MR0242185},
  making the analogue of \eqref{equation:substacks} an inclusion of dense open substacks. However, the following standard example shows that whether or not $\modulistack[\stable{\theta}]{d,S}$ or $\modulistack[\semistable{\theta}]{d,S}$ is empty may depend on~$\theta$.
\end{remark}

\begin{example}
  Consider a representation of the $\mathrm{A}_2$-quiver ${\bullet} \xrightarrow{\, a \,} {\bullet}$ over a field~$k$ of dimension vector $d_n=(n,n)$ for a positive integer $n$. We can reduce the study of stability functions $\theta$ such that $\theta(d_n)=0$ to three cases:
  \begin{description}
    \item[$\theta=(0,0)$:] all representations of dimension vector $d_n$ are semistable, but none of them can be geometrically stable, as the unique subrepresentation of dimension vector $(0,n)$ has slope $0$.
    \item[$\theta=(-1,1)$:] no representation of dimension vector $d_n$ can be geometrically stable and even semistable, as the unique subrepresentation of dimension vector $(0,n)$ is a destabilizing subrepresentation.
    \item[$\theta=(1,-1)$:] a representation $M_a \colon k^n \to k^n$ is semistable if and only if $M_a$ is injective, while it is geometrically stable if and only if $n=1$ and $M_a$ is injective.
  \end{description}
\end{example}

For an acyclic quiver, a general criterion for the existence of a stable representation
is given in \cite{MR1972892}.

\subsection{Determinantal line bundles} \label{subsection:determinantal-line-bundles}
In this section, we explain how to construct line bundles $\mathcal{L}_\theta$ on the stack $\modulistack{d}$ depending on a stability function $\theta\colon \ZZ^{Q_0} \to \ZZ$. For an $S$-scheme $T$ and a family $\cF$ of representations of $Q$ over $T$, we define a line bundle on $T$
\[
  \cL_{\theta,\cF} \coloneqq
  \bigotimes_{i \in Q_0} \det(\cF_i)^{\otimes -\theta_i}
\]
where $\theta_i = \theta(S(i))$ for $i \in Q_0$. We list some basic properties of this construction.
\begin{proposition}
  \label{proposition:determinantal-properties}
  The following statements hold.
  \begin{enumerate}[label=(\roman*)]
    \item\label{enumerate:determinantal-1} The assignment
    $\theta \mapsto  \cL_{\theta,\cF}$ defines a homomorphism
    $\Hom(\ZZ^{Q_0}, \ZZ) \to \Pic(T)$.
    \item\label{enumerate:determinantal-2} If $0 \to \cF' \to \cF \to \cF'' \to 0$ is a short exact sequence of families of representations on $T$, then there is a canonical isomorphism
    $\cL_{\theta,\cF'} \otimes_{\cO_T} \cL_{\theta,\cF''} \cong \cL_{\theta,\cF}$.
    \item\label{enumerate:determinantal-3} If $f\colon T' \to T$ is a morphism of $S$-schemes, then there is a canonical isomorphism
    $f^*\cL_{\theta, \cF} \cong \cL_{\theta, f^*\cF}$.
  \end{enumerate}
\end{proposition}
\begin{proof}
  Part (i) is clear from the definition. For part (ii), see for example  \cite[Tag \href{https://stacks.math.columbia.edu/tag/0FJB}{0FJB}]{stacks-project}, and for part (iii), see \cite[Tag \href{https://stacks.math.columbia.edu/tag/0FJY}{0FJY}]{stacks-project}.
\end{proof}

Since $\modulistack{d}$ parameterizes families of representations, we have described how to construct a line bundle on $T$ associated to every map $T \to \modulistack{d}$. From property \ref{enumerate:determinantal-3}, it follows that this assignment gives rise to a line bundle on $\modulistack{d}$ which we denote $\mathcal{L}_\theta$.
Equivalently, we obtain $\mathcal{L}_\theta$
by applying the above construction to the universal representation $\cF^{\mathrm{univ}}$. Moreover, this construction defines a group homomorphism
$\Hom(\ZZ^{Q_0}, \ZZ) \to \Pic(\modulistack{d})$,
which is analogous to the one constructed in \cite[Definition~8.1.1]{MR2665168}.

\begin{proposition}
  The map
  $\Hom(\ZZ^{Q_0}, \ZZ) \to \Pic(\modulistack{d,\ZZ}), \theta \mapsto \cL_\theta$
  is an isomorphism.
\end{proposition}

\begin{proof} 
    The map
    corresponds to the isomorphisms
    \[
    \Hom(\ZZ^{Q_0}, \ZZ) \cong X^*(\textrm{G}_d) \cong \Pic^{\textrm{G}_d}(\Rep{d}) \cong \Pic( \modulistack{d}),
    \]
    where $X^*(\textrm{G}_d)$ denotes the character group of $\textrm{G}_d$.
\end{proof}

Suppose now that we are given two families of representations $\cF$ and $\cG$ on an $S$-scheme $T$. We can construct a $2$-term complex of locally free sheaves concentrated in degrees $-1$ and $0$ by setting
\begin{equation} \label{eq:2term-hom-complex}
  \cE_{\cF, \cG}^\bullet \colon \bigoplus_{i \in Q_0} \lHom_{\cO_T}(\cF_i, \cG_i) \xrightarrow{d_{\cF,\cG}} \bigoplus_{a \in Q_1} \lHom_{\cO_T}(\cF_{s(a)}, \cG_{t(a)})
\end{equation}
where the differential of degree one is given by
\[
  d_{\cF,\cG} \colon (\phi_i)_{i \in Q_0} \mapsto (\phi_{t(a)} \circ \cF_a - \cG_a \circ \phi_{s(a)})_{a \in Q_1}.
\]
As a first application of \eqref{eq:2term-hom-complex}, we see that the dimensions
\[
  \dim_{\kappa(x)} \Hom(\cF|_x, \cG|_x) \quad \mathrm{and} \quad \dim_{\kappa(x)} \Ext(\cF|_x, \cG|_x)
\]
are upper semicontinuous functions on $T$ -- by  comparing with \eqref{eq:exact-seq-hom-ext}, we see that the latter vector space is nothing but the fiber of the quasi-coherent sheaf $\coker(d_{\cF,\cG})$, while the former is dual to the fiber of $\coker(d_{\cF,\cG}^\vee)$, both of which are finitely presented.

Our second application of \eqref{eq:2term-hom-complex} is a moduli-theoretic construction of determinantal semi-invariants mentioned in the introduction. If the two families $\cF$ and $\cG$ satisfy the condition $\langle \dimvect{\cF}, \dimvect{\cG} \rangle = 0$, then
\[
  \rk(\cE_{\cF, \cG}^{-1}) = \sum_{i \in Q_0} \rk(\cF_i) \rk(\cG_i) = \sum_{a \in Q_1} \rk(\cF_{s(a)}) \rk(\cG_{t(a)}) = \rk(\cE_{\cF, \cG}^{0}),
\]
and thus we obtain a global section of the determinant of $\cE_{\cF, \cG}^\bullet$:
\[ \sigma_{\cF, \cG} = \det(d_{\cF, \cG})\colon \cO_T \to \det(\cE_{\cF, \cG}^\bullet) = \det(\cE_{\cF, \cG}^{-1})^\vee \otimes \det(\cE_{\cF, \cG}^0). \]

Suppose that the stability function $\theta\colon\ZZ^{Q_0} \to \ZZ$ is of the form $\eta_\beta = - \langle \blank, \beta \rangle$ for a dimension vector $\beta \in \NN^{Q_0}$; that is, we have
\[
  \theta_i = - \langle \dimvect S(i), \beta \rangle = -\beta_i + \sum_{a \in Q_1: s(a) = i} \beta_{t(a)}.
\]
By applying the construction of the $2$-term complex to the universal family $\mathcal{F}:=\cF^{\mathrm{univ}}$ on $\modulistack{d}$ and a family $\mathcal{G}$ of representations on $\modulistack{d}$ of dimension vector $m\beta$ for $m \in \NN$ such that $\mathcal{G}_i$ is \textit{free} for each $i \in Q_0$,
we get an associated determinantal line bundle that is isomorphic to $\mathcal{L}_\theta^{\otimes m}$, as we have
\[
  \lHom_{\cO_{\modulistack{d}}}(\cF_i, \cG_i) \cong \cF_i^{\oplus -m \beta_i}, \quad \lHom_{\cO_{\modulistack{d}}}(\cF_{s(a)}, \cG_{t(a)}) \cong \cF_{s(a)}^{\oplus -m \beta_{t(a)}},
\]
and so
\[
  \det(\cE_{\cF^{\mathrm{univ}}, \cG}^\bullet) \coloneqq \det(\cE_{\cF^{\mathrm{univ}}, \cG}^{-1})^\vee \otimes \det(\cE_{\cF^{\mathrm{univ}}, \cG}^0) \cong \bigotimes_{i \in Q_0} \det(\cF_i)^{\otimes -m \theta_i} = \mathcal{L}_\theta^{\otimes m}.
\]
If moreover $\theta(d) = -\langle d, \beta \rangle = 0$
then the ranks of the two terms in the complex $\cE_{\cF^{\mathrm{univ}}, \cG}^\bullet$ are equal, so the determinant of the differential $d_{\cF^{\mathrm{univ}}, \cG}$ gives a section
\[
  \sigma_{\cG} = \sigma_{\cF^{\mathrm{univ}}, \cG} \coloneqq \det(d_{\cF^{\mathrm{univ}}, \cG}) \colon \; \cO_{\modulistack{d}} \to \cL_\theta^{\otimes m}.
\]

The most important case is when
$\cG = V \otimes \cO_{\modulistack{d}}$ is a ``constant family'', that is, the pullback of a family $V$ on $S$,
in which case we denote the section by
$\sigma_V\colon \cO_{\modulistack{d}} \to \mathcal{L}_\theta^{\otimes m}$.
The following result, corresponding to \cite[Theorem~1.1]{MR1113382}, describes the vanishing locus of $\sigma_V$.

\begin{proposition} \label{proposition:vanishing-sections}
Let $d$ be a dimension vector and $\theta$ a stability function with $\theta(d) = 0$.
Let $k$ be a field and $x \colon \Spec{k} \to \modulistack[\semistable{\theta}]{d}$ a morphism
corresponding to a representation $M$ over $k$, and let $y \in S$ denote its image in $S$.
\begin{enumerate}[label=(\alph*)]
    \item If $\theta = \theta_\alpha$ for a dimension vector $\alpha$ and $V$ is a
    family of representation of dimension vector $m\alpha$ on $S$, then the section $\sigma_V$ of $\mathcal{L}_\theta^{\otimes m}$ is nonzero at
    $x$ if and only if
    \[\Hom_k(V|_y \otimes_{\kappa(y)} k, M) = 0, \; \text{ or equivalently } \Ext(V|_y \otimes_{\kappa(y)} k, M) = 0.
    \]
    \item If $\theta = \eta_\beta$ for a dimension vector $\beta$ and $V$ is a
    family of representations of dimension vector $m\beta$ on $S$, then the section $\sigma_V$ of $\mathcal{L}_\theta^{\otimes m}$ is nonzero at
    $x$ if and only if
    \[\Hom(M, V|_y \otimes_{\kappa(y)} k) = 0, \; \text{ or equivalently } \Ext(M, V|_y \otimes_{\kappa(y)} k) = 0.
    \]
\end{enumerate}
\end{proposition}

\begin{proof}
To prove (b), we note that the fiber of the complex $\cE_{\cF^{\mathrm{univ}}, \cG}^\bullet$ at $x$ is identified with the complex
\[  \begin{tikzcd}
\bigoplus\limits_{i \in Q_0} \Hom_{k}(M_i, V_i|_y \otimes_{\kappa(y)} k) \arrow{r} & \bigoplus\limits_{a \in Q_1} \Hom_{k}(M_{s(a)}, V_{t(a)}|_y \otimes_{\kappa(y)} k)\end{tikzcd}, \]
defined in \eqref{eq:exact-seq-hom-ext} whose kernel and cokernel are the $k$-vector spaces $\Hom(M, V)$ and $\Ext(M, V)$ respectively, and these vanish if and only if the differential of the complex is invertible, if and only its determinant is nonzero.

For part (a) we instead consider the complex $\cE_{\cG, \cF^{\mathrm{univ}}}^\bullet$ and proceed similarly.
\end{proof}

\section{Vanishing results}
\label{section:vanishing-results}

In this section we take our base scheme to be $S = \Spec k$ where $k = \overline{k}$ is an algebraically closed field. The results obtained here will be used in later sections to construct adequate moduli spaces for stacks of semistable representations and to show that they are projective when the quiver is acyclic.

The results in this section are inspired by the proof of global generation of determinantal line bundles on the moduli of vector bundles on a curve using dimension-counting techniques in \cite{esteves.popa:2004:veryampleness}. We will obtain parallel results in the case of moduli of quiver representations. The starting point is \cref{proposition:vanishing-sections}, which describes the non-vanishing locus of determinantal semi-invariants. In \cref{subsection:characterizing-semistable-reps}, for a given semistable representation $M$ over $k$, we will find another representation $N$ satisfying $\Hom(M, N) = \Ext(M, N) = 0$ by showing that in the appropriate representation space, the locus of those $N$ for which $\Hom(M, N) \neq 0$ has positive codimension; this will enable us to show a power of the determinantal line bundle on the moduli space of semistable representations is globally generated in \cref{section:projectivity}. For acyclic quivers, we study the preservation of (semi)stability under the Auslander-Reiten translations in \cref{subsection:ar-translations-vs-semistability} and use this to establish generic vanishing of $\Ext$ groups in \cref{subsection:generic-ext-vanishing}, in order to ultimately prove a key result in \cref{subsection:sep-det-inv} required to later show ampleness of the determinantal line bundle.

Throughout this section, we will formulate many results for stability functions of the form (a) $\theta =\theta_\alpha$ and (b) $\theta =\eta_\beta$, where $\alpha$ and $\beta$ are dimension vectors. This condition is automatically satisfied if $Q$ is acyclic and there exists a semistable representation supported on $Q_0$, see \cref{lemma:stability-conditions-are-dim-vectors}. Moreover, we only give the proof in one of these two cases, as the other is proved analogously. Since we will ultimately be concerned with case (b) $\theta = \eta_\beta$ in \cref{section:projectivity}, we mostly apply these results in this case and thus give the proofs in this case.

\subsection{Characterizing semistable representations}
\label{subsection:characterizing-semistable-reps}

We begin with the key dimension estimates. Let $d', d'' \in \NN^{Q_0}$ and write $d = d' + d''$. Let $M$ and $N$ be representations of dimension vector $d'$ and $d''$ respectively. Define the following subsets of $\Rep{d}$:
\begin{align}
    \label{eq:rep-subset-B} B_{d'}=B_{d',d} & = \{ V \in \Rep{d} \mid \text{there exists } V' \subset V \text{ with } \dimvect{V'} = d'  \}, \\
    \label{eq:rep-subset-K} K_M=K_{M,d} & = \{ V \in \Rep{d} \mid \text{there exists an injection } M \hookrightarrow V \},\\
    \label{eq:rep-subset-Q} Q_N=Q_{N,d} & = \{ V \in \Rep{d} \mid \text{there exists a surjection } V \twoheadrightarrow N \}.
\end{align}

\begin{lemma}\label{lemma:estimates}
   We have the following estimates.
   \begin{enumerate}[label=(\roman*)]
       \item $\codim_{\Rep{d}} B_{d'} \ge -\langle d', d'' \rangle$,
       \item $\codim_{\Rep{d}} K_M \ge 1 - \langle d', d \rangle$,
       \item $\codim_{\Rep{d}} Q_N \ge 1 - \langle d, d'' \rangle$.
   \end{enumerate}
\end{lemma}

\begin{proof}
  View $\Rep{d}$ as parameterizing representations $V$ where $V_i = k^{\oplus d_i}$ for each $i$. Let $V'_i \subseteq V_i$ be the subspace spanned by the $d'_i$ first standard basis vectors and define
  the subset of all representations for which $V'$ is a subrepresentation:
  \[
    S \colonequals \{ V     \in \Rep{d} \mid V_a(V'_{s(a)}) \subseteq V'_{t(a)} \text{ for all $a \in Q_1$}\}.
  \]
  A representation $V \in \Rep{d}$ has a subrepresentation of dimension vector $d'$ if and only if it lies in the $\textrm{G}_d$-saturation of $S$, that is,
  \[ B_{d'} = \textrm{G}_d \cdot S. \]

  Consider the parabolic subgroup $P \subseteq \textrm{G}_d$ given by
  \[
    P \colonequals \{ g \in \textrm{G}_d \mid g_i(V'_i) \subseteq V'_i \text{ for all $i \in Q_0$} \}.
  \]
  The subgroup $P$ acts on $S$, which implies that the action map $\textrm{G}_d \times S \to \Rep{d}$ factors as
  \[
  \begin{tikzcd}
      \textrm{G}_d \times S \arrow{d}{} \arrow{dr}{} & \\
      \textrm{G}_d \times^P S \arrow[dashed]{r}{\exists !} & \Rep{d}
  \end{tikzcd}
  \]
  where $\textrm{G}_d \times^P S$ is the associated fiber bundle
  \cite[Proposition 4]{serre:1958}.
  Thus, from the surjection $\textrm{G}_d \times^P S \to B_{d'}$ we obtain the bound
  \begin{equation*}
    \dim B_{d'} \le \dim
    \left( \textrm{G}_d \times^P S \right)
    = \dim \textrm{G}_d + \dim S - \dim P
  \end{equation*}
  and hence
  \begin{equation*}
    \codim_{\Rep{d}} B_{d'} \ge \dim \Rep{d} - \dim S - (\dim \textrm{G}_d - \dim P) = \codim_{\Rep{d}} S - \codim_{\textrm{G}_d} P.
  \end{equation*}
  Now $S$ is a linear subspace of $\Rep{d}$ of codimension $\sum_{a \in Q_1} d'_{s(a)} d''_{t(a)}$ and $P$ is a subgroup of $\textrm{G}_d$ of codimension $\sum_{i \in Q_0} d'_i d''_i$, so we have
  \[ \codim_{\Rep{d}} B_{d'} \ge \sum_{a \in Q_1} d'_{s(a)} d''_{t(a)} - \sum_{i \in Q_0} d'_i d''_i = - \langle d', d'' \rangle, \]
  proving (i).

  We have projection maps $p\colon S \to \Rep{d'}$ and $q\colon S \to \Rep{d''}$ taking a representation $V$ to the subrepresentation $V'$ and the quotient $V'' = V/V'$ respectively. Identifying $M$ and $N$ with points in $\Rep{d'}$ and $\Rep{d''}$ and letting $O_M \subset \Rep{d'}$ and $O_N \subset \Rep{d''}$ denote their orbits under the actions of $G_{d'}$ and $G_{d''}$ respectively, we see that
  \[ K_M = \textrm{G}_{d} \cdot p^{-1}(O_M), \quad Q_N = \textrm{G}_{d} \cdot q^{-1}(O_N) . \]

  To prove (ii) and (iii), we note the group $P$ acts on $\Rep{d'}$ and $\Rep{d''}$ via its natural projections to $\mathrm{G}_{d'}$ and $\mathrm{G}_{d''}$, and that the projections $p$ and $q$ are $P$-equivariant under these actions. Thus, $P$ acts on the preimages $p^{-1}(O_M)$ and $q^{-1}(O_N)$, so we have surjections
  \[
    \textrm{G}_d \times^P p^{-1}(O_M) \to K_M, \quad \textrm{G}_d \times^P q^{-1}(O_N) \to Q_N,
  \]
  from which we obtain the bounds
  \begin{equation*}
    \codim_{\Rep{d}} K_M \ge \dim \Rep{d} - \dim p^{-1}(O_M) - \codim_{\textrm{G}_d} P
  \end{equation*}
  and
  \begin{equation*}
    \codim_{\Rep{d}} Q_N \ge \dim \Rep{d} - \dim q^{-1}(O_N) - \codim_{\textrm{G}_d} P.
  \end{equation*}
  Now the map $p$ is a projection along a linear subspace of dimension $\sum_{a \in Q_1} d''_{s(a)} d_{t(a)}$ so
  \[ \dim p^{-1}(O_M) = \dim O_M + \sum_{a \in Q_1} d''_{s(a)} d_{t(a)}. \]
  On the other hand,
  \[ \dim O_M = \dim \textrm{G}_{d'} - \dim \textrm{Stab}_{\textrm{G}_{d'}}(M) \le \dim \textrm{G}_{d'} - 1 = \sum_{i \in Q_0} (d'_i)^2 - 1,  \]
  as the stabilizer of any representation contains the multiplicative group $\mathbb{G}_\mathrm{m}$. Thus,
  \begin{align*}
      \codim_{\Rep{d}} K_M & \ge \dim \Rep{d} - \dim \textrm{G}_{d'} + 1 - \codim_{\textrm{G}_d} P \\
      & = \sum_{a \in Q_1} d_{s(a)} d_{t(a)} - \sum_{i \in Q_0} (d'_i)^2 + 1 - \sum_{a \in Q_1} d''_{s(a)} d_{t(a)} - \sum_{i \in Q_0} d'_i d''_i \\
      & = 1 + \sum_{a \in Q_1} d'_{s(a)} d_{t(a)} - \sum_{i \in Q_0} d'_i d_i \\
      & = 1 - \langle d', d \rangle,
  \end{align*}
  proving (ii).

  Similarly, as $q$ is a projection along a linear subspace of dimension $\sum_{a \in Q_1} d_{s(a)} d'_{t(a)}$ and $\dim O_N \le \dim \mathrm{G}_{d''} - 1$, the corresponding calculation yields (iii).
\end{proof}

As in \cite{MR1162487} we say that a property of representations holds for a \emph{general} representation of dimension vector~$d$ if there exists a nonempty open substack~$\mathcal{U}$ of~$\modulistack{d}$ such that the property holds for all~$M$ in~$\mathcal{U}$.
This is equivalent to giving a nonempty $\mathrm{G}_{d}$-invariant open subscheme of~$\Rep{d}$.
Note that since $\modulistack{d}$ is irreducible, any nonempty open substack is dense, and in particular any finitely many general properties will hold simultaneously on a dense open substack.

Let $\theta\colon \ZZ^{Q_0} \to \ZZ$ be a stability function. We will use \cref{lemma:estimates} to relate various vanishing results with $\theta$-semistability for representations.
\begin{lemma}\label{lemma:image-with-bound}
  Let $M$ be a $\theta$-semistable representation and let $\epsilon \in \NN^{Q_0}$.
  \begin{enumerate}[label=(\alph*)]
      \item\label{enumerate:image-with-bound-1} Suppose $\theta = \theta_\alpha$ for a dimension vector $\alpha$. If $m$ is sufficiently large, then for a general representation $V$ of dimension vector $m \alpha + \epsilon$, every nonzero map $f\colon V \to M$ satisfies $\theta(\im f) = 0$. In fact, it suffices to take
      \begin{equation}\label{eq:m-inequality-alpha}
        m > \frac{\langle \gamma, \gamma \rangle - \langle \epsilon, \gamma \rangle}{\langle \alpha, \gamma \rangle}
      \end{equation}
      for the finitely many dimension vectors $0 < \gamma < \dimvect M$ such that $\langle \alpha, \gamma \rangle < 0$.

      \item\label{enumerate:image-with-bound-2} Suppose $\theta = \eta_\beta$ for a dimension vector $\beta$. If $m$ is sufficiently large, then for a general representation $V$ of dimension vector $m \beta + \epsilon$, every nonzero map $f\colon M \to V$ satisfies $\theta(\im f) = 0$. In fact, it suffices to take
      \begin{equation}\label{eq:m-inequality-beta}
        m > \frac{\langle \gamma, \gamma \rangle - \langle \gamma, \epsilon \rangle}{\langle \gamma, \beta \rangle}
      \end{equation}
      for the finitely many dimension vectors $0 < \gamma < \dimvect M$ such that $\langle \gamma, \beta \rangle < 0$.
  \end{enumerate}
\end{lemma}
\begin{proof}
  We only prove (b), as (a) is dual.
  Since $M$ is $\theta$-semistable, if $f\colon M \to V$ is any nonzero map, then $\theta(\im f) \ge 0$, so it suffices to rule out the case $\theta(\im f) > 0$.

  Let $0 < \gamma < \dimvect M$ be the dimension vector of a quotient representation of $M$ such that $\theta(\gamma) = -\langle \gamma, \beta \rangle > 0$. The subset
  \[
    B \colonequals \{ V \in \Rep{m\beta+\epsilon} \mid \text{there exists } f \in \Hom(M, V) \text{ such that } \dimvect \im f = \gamma \}
  \]
  is contained in the set $B_{\gamma, m\beta+\epsilon}$
  defined in \eqref{eq:rep-subset-B} so from \cref{lemma:estimates} we deduce that
  \begin{align*}
    \codim_{\Rep{m\beta+\epsilon}} B & \ge \codim_{\Rep{m\beta+\epsilon}} B_{\gamma, m\beta+\epsilon} \\
    & \ge - \langle \gamma, m \beta + \epsilon - \gamma \rangle \\
    & = - m \langle \gamma, \beta \rangle - \langle \gamma, \epsilon \rangle + \langle \gamma, \gamma \rangle.
  \end{align*}
  If $m$ satisfies the inequality in \eqref{eq:m-inequality-beta}, we see that $\codim_{\Rep{m\beta+\epsilon}} B > 0$, so for a general representation $V$ of dimension vector $m\beta+\epsilon$, there are no maps $f\colon M \to V$ with $\dimvect \im f = \gamma$.
\end{proof}

Using the above lemma, we obtain $\Hom$-vanishing conditions for stable and semistable representations with respect to $\theta_\alpha$ and $\eta_\beta$.

\begin{corollary}\label{corollary:hom-vanishing-for-stables}
  Let $M$ be a $\theta$-stable representation of dimension $d$ and let $\epsilon \in \NN^{Q_0}$.
  \begin{enumerate}[label=(\alph*)]
      \item Suppose $\theta = \theta_\alpha$ for some $\alpha \in \NN^{Q_0}$ and assume that $\langle \epsilon, d \rangle \le 0$. If $m$ satisfies \eqref{eq:m-inequality-alpha}, then $\Hom(V, M) = 0$ for a general representation $V$ of dimension vector $m \alpha + \epsilon$.
      \item Suppose $\theta = \eta_\beta$ for some $\beta \in \NN^{Q_0}$ and assume that $\langle d, \epsilon \rangle \le 0$. If $m$ satisfies \eqref{eq:m-inequality-beta}, then $\Hom(M, V) = 0$ for a general representation $V$ of dimension vector $m \beta + \epsilon$.
  \end{enumerate}
\end{corollary}
\begin{proof}
  We prove (b). By \cref{lemma:image-with-bound} we have $\theta(\im f) = 0$ for a general representation $V$ of dimension vector $m \beta + \epsilon$ and any nonzero map $f\colon M \to V$, so since $M$ is $\theta$-stable, any such nonzero map is injective. However, by \cref{lemma:estimates}(ii), the locus $K_M \subseteq \Rep{m \beta + \epsilon}$ defined in \eqref{eq:rep-subset-K} of representations $V$ for which there exists an injection $M \hookrightarrow V$ has codimension
  \[ \codim_{\Rep{m \beta + \epsilon}} K_M \ge
  1 - \langle d, m \beta + \epsilon \rangle = 1 - \langle d, \epsilon \rangle \ge 1 \]
  since by assumption $\langle d, \beta \rangle = 0$ and $\langle d, \epsilon \rangle \le 0$. Thus, a general representation $V$ does not admit such an injection.

  The proof of (a) is dual and uses $Q_M$ as defined in \eqref{eq:rep-subset-Q} in place of $K_M$.
\end{proof}

\begin{corollary} \label{corollary:hom-vanishing}
  Let $M$ be a $\theta$-semistable representation and let $\epsilon \in \NN^{Q_0}$.
  \begin{enumerate}[label=(\alph*)]
    \item\label{enumerate:hom-vanishing-alpha}
      Suppose that $\theta = \theta_\alpha$ for a dimension vector $\alpha$
      and assume that $\langle \epsilon, \gamma \rangle \leq 0$ for the dimension vectors $\gamma$
      of all $\theta$-stable subquotients of $M$.
      If $m$ satisfies \eqref{eq:m-inequality-alpha}, then
      for a general representation $V$ of dimension vector~$m\alpha+\epsilon$,
      we have $\Hom(V,M) = 0$.
    \item\label{enumerate:hom-vanishing-beta}
      Suppose that $\theta = \eta_\beta$ for a dimension vector $\beta$
      and assume that $\langle \gamma, \epsilon \rangle \leq 0$ for the dimension vectors $\gamma$
      of all $\theta$-stable subquotients of $M$.
      If $m$ satisfies \eqref{eq:m-inequality-beta}, then
      for a general representation $V$ of dimension vector~$m\beta+\epsilon$,
      we have $\Hom(M,V) = 0$.
  \end{enumerate}
\end{corollary}

\begin{proof}
    We show the (b), as (a) is analogous.
    Let $M^1, \ldots, M^r$ denote the $\theta$-stable subquotients
    of a Jordan-H\"older filtration of $M$.
    By assumption
    $\langle \dimvect M^\ell , \epsilon \rangle \le 0$,
    so by \cref{corollary:hom-vanishing-for-stables},
    a general representation $V$ of dimension vector~$m\beta+\epsilon$ satisfies
    \[
      \Hom(M^\ell, V) = 0
    \]
    for each $\ell$.
    By breaking up the Jordan-H\"older filtration of $M$ into
    short exact sequences, we inductively deduce that
    $\Hom(M, V)$ vanishes for a general $V$ of dimension vector~$m\beta+\epsilon$.
\end{proof}

\Cref{corollary:hom-vanishing} can be used to derive a characterization of semistability in terms of vanishing of $\Hom$ and $\Ext$. Note that for representations~$M$ and~$N$ such that $\langle\dimvect M,\dimvect N \rangle = 0$, we have $\Hom(M,N) = 0$ if and only if $\Ext(M,N) = 0$.

\begin{proposition} \label{proposition:sst}
  Let $\theta$ be a stability function and let $M$ be a representation with $\theta(M)=0$.
  \begin{enumerate}[label=(\alph*)]
    \item\label{enumerate:sst-alpha}
      If $\theta = \theta_\alpha$ for a dimension vector $\alpha$,
      then $M$ is $\theta$-semistable if and only if there exists $m > 0$ and a representation $V$ of dimension vector $m \alpha$ such that $\Hom(V, M) = 0$.
    \item\label{enumerate:sst-beta}
      If $\theta = \eta_\beta$ for a dimension vector $\beta$,
      then $M$ is $\theta$-semistable if and only if there exists $m > 0$ and a representation $V$ of dimension vector $m \beta$ such that $\Hom(M, V) = 0$.
  \end{enumerate}
\end{proposition}

\begin{proof}
  We will prove \ref{enumerate:sst-beta},
  as the proof of \ref{enumerate:sst-alpha} is dual.
  The forward implication of \ref{enumerate:sst-beta} is a special case of \cref{corollary:hom-vanishing} with $\epsilon = 0$.
  For the other direction,
  let $M' \subseteq M$ be a subrepresentation of dimension vector $d'$.
  After applying $\Hom(\blank,V)$ to the short exact sequence
  \[ 0 \to M' \to M \to M/M' \to 0,
  \]
  we get $\Ext(M', V) = 0$, and therefore
  \begin{align*}
    m \eta_\beta(\dimvect M') &=
    - \langle \dimvect M', m \beta \rangle
    = - \dim \Hom(M',V) + \dim \Ext(M',V) \\
    &= - \dim \Hom(M', V) \leq 0.
  \end{align*}
  This shows that $M$ is $\eta_\beta$-semistable.
\end{proof}

\begin{remark}
  \Cref{proposition:sst} appears as \cite[Corollary~1.1]{MR1908144} in characteristic $0$,
  using possibly infinite-dimensional representations when $Q$ is not acyclic,
  and with a proof using GIT methods.
  Our result stays completely within the realm of finite-dimensional representations
  and holds in all characteristics.
  In arbitrary characteristic the forward implication is proved in \cite[Corollary~2]{MR1348149}.

  In addition, \cref{proposition:sst} is the quiver analogue of
  Faltings's characterization of semistability for vector bundles, and more generally Higgs bundles, on a curve \cite[Theorem~I.2]{MR1211997}.
\end{remark}

\subsection{Effective bounds for vanishing of \texorpdfstring{$\Hom$}{Hom}}

We use the inequalities of \cref{lemma:image-with-bound} to derive an
upper bound for the vanishing of $\Hom$ and $\Ext$ that only depends on the underlying undirected graph of $Q$. This will turn into an upper bound for global generation of determinantal line bundles in \cref{section:projectivity}.

Recall from \cref{sec:quiverrepr} that we denote the Euler matrix of $Q$ by $A$ for a choice of an ordering of the vertices. Let
$B = \frac{1}{2} \left( A + A^{\mathrm{T}} \right)$ be the symmetrization of $A$.
These matrices $A$ and $B$ define the same quadratic form, called the \emph{Tits form}, that for any vector $x\in \ZZ^{Q_0}$ associates its self-pairing:
\[
  \langle x , x \rangle = x^{\mathrm{T}} A x = x^{\mathrm{T}} B x.
\]
Notice that the Tits form is independent of the orientation of the quiver.
In addition to the Euler pairing, we will also use the standard
inner product on $\ZZ^{Q_0}$ and write the induced norm of $x$ as $\|x\|$.

\begin{proposition}
  \label{proposition:effective-bounds}
  Let $d \in \NN^{Q_0}$ be a dimension vector
  and let $\theta$ be a stability function such that~$\theta(d)=0$.
  Denote
  \[
    \lambda\colonequals-\min\{\mu\mid\text{$\mu$ eigenvalue of $B$}\}
  \]
  and let~$m$ be a positive integer greater than~$\lambda\|d\|^2$.

  \begin{enumerate}[label=(\alph*)]
    \item If $\theta = \theta_\alpha = \langle \alpha , \blank \rangle$ for a dimension vector $\alpha$, then for every $\theta$-semistable representation $M$ of dimension vector $d$, a general representation $V$ of dimension vector $m \alpha$ satisfies
    \[
      \Hom(V,M) = \Ext(V,M) = 0.
    \]
    \item If $\theta = \eta_\beta = - \langle \blank , \beta \rangle$ for a dimension vector $\beta$, then for every $\theta$-semistable representation $M$ of dimension vector $d$, a general representation $V$ of dimension vector $m \beta$ satisfies
    \[
      \Hom(M,V) = \Ext(M,V) = 0.
    \]
  \end{enumerate}
\end{proposition}

\begin{proof}
We will only prove (b), as the argument for (a) is identical and leads to the same bound.
Given a $\theta$-semistable representation $M$ of dimension vector $d$ and a positive integer $m$, \cref{corollary:hom-vanishing} with $\epsilon = 0$ implies that as soon as
\[ m > f(\gamma) \coloneqq \frac{\langle \gamma, \gamma \rangle}{\langle \gamma, \beta \rangle} \]
for the finitely many dimension vectors $0 < \gamma < d$ for which $\langle \gamma, \beta \rangle < 0$, we have $\Hom(M, V) = \Ext(M, V) = 0$ for a general representation $V$ of dimension vector $m \beta$.

Clearly it is enough to consider only those $\gamma$ for which $\langle \gamma, \gamma \rangle < 0$, since otherwise 
$f(\gamma) \le 0$.
If $Q$ is either a Dynkin or extended Dynkin quiver,
there are no such $\gamma$, and we may take any $m \ge 1$.
Hence we will assume that $Q$ is not of these types, or equivalently~$\lambda>0$.
We now prove the claim by showing that $f(\gamma) \leq \lambda ||d||^2$ for all dimension vectors $0 < \gamma < d$ for which $\langle \gamma, \beta \rangle < 0$ and $\langle \gamma,\gamma \rangle < 0$.

Since the denominator of $f(\gamma)$ is assumed to be negative, in order to obtain an upper bound for $f(\gamma)$, we need to minimize the numerator.
Notice that since $B$ is symmetric, the minimal value of the Tits form on the unit sphere is
\[ -\lambda = \min\left\{ \langle \rho, \rho \rangle \mid \rho \in \RR^{Q_0}, \|\rho\| = 1 \right\}. \]
For~$\gamma < d$ that satisfies $\langle \gamma, \beta \rangle < 0$ and $\langle \gamma, \gamma \rangle \leq 0$, write $\gamma = \|\gamma\| \gamma_0$ where $\|\gamma_0\| = 1$.
We now have the following inequalities between nonnegative numbers
\[ \|\gamma\| < \|d\| ,
    \quad \frac{1}{|\langle \gamma, \beta \rangle|} \leq 1 ,
    \quad | \langle \gamma_0, \gamma_0 \rangle| \leq \lambda,
\]
where the third one follows from the assumption $\langle \gamma, \gamma \rangle < 0$. Combining these inequalities gives the estimate
\[ f(\gamma) = \frac{\langle \gamma, \gamma \rangle}{\langle \gamma, \beta \rangle}
\le
\frac{\|\gamma\|^2 \cdot |\langle \gamma_0, \gamma_0 \rangle|}
{|\langle \gamma, \beta \rangle|}
\le
\frac{\|d\|^2 \lambda}{|\langle \gamma, \beta \rangle|} \le
\lambda \|d\|^2. \]
\end{proof}

\begin{example}
  \label{example:nKronecker}
  The $n$-Kronecker quiver
  \begin{equation*}
    Q\colon
    \begin{tikzcd}[every label/.append style={font=\small}]
     1 \arrow[r, draw=none, "\raisebox{+1.5ex}{\vdots}" description] \arrow[r, bend left] \arrow[r, bend right, swap] & 2
    \end{tikzcd}
  \end{equation*}
  has Tits matrix $\begin{psmallmatrix} 1 & -\frac{n}{2} \\ -\frac{n}{2} & 1 \end{psmallmatrix}$. It follows that its eigenvalues are $1\pm\frac{n}{2}$ and so
  $\lambda= \frac{n}{2}-1$, which is positive when $n=1$ and zero when $n=2$.
\end{example}

\begin{remark}
  \label{remark:extended-dynkin}
  In \cref{example:nKronecker}, the case where~$n=1,2$ is a Dynkin, resp.~extended Dynkin quiver.
  As observed in the proof of \cref{proposition:effective-bounds} any~$m\geq 1$ will work.
  This is consistent with the effective basepoint-freeness results following from \cref{proposition:effective-bounds}
  discussed in \cref{subsection:global-generation},
  because by \cite[Theorem~3.1]{MR2747139} the moduli spaces are all affine or projective spaces.
  
  In certain cases, the effective bounds described in \cref{proposition:effective-bounds} are not optimal. For instance let $Q$ be the $(n+1)$-Kronecker quiver, $d=(1,1)$, and $\theta=\eta_\beta$ with $\beta=(n,1)$. \Cref{proposition:effective-bounds} shows that $\Hom(M,V) = 0$ for $M$ a $\theta$-semistable representation of $Q$ of dimension $d$ and $V$ a general representation of dimension $m \beta$ for $m \geq 2(n-1)$. A direct computation shows that it is in fact enough to let $m$ be any positive integer. 
\end{remark}

\subsection{Auslander--Reiten translations and semistability}
\label{subsection:ar-translations-vs-semistability}

In this section we investigate how stability behaves under the Auslander--Reiten translations, and therefore assume throughout that $Q$ is acyclic. For this reason, whenever we write $\theta = \theta_\beta$ or $\theta = \eta_\beta$, we implicitly assume that $\beta$ is a dimension vector.

Only \cref{lemma:auslander-reiten-semistable} from this section is used as an ingredient to our main theorem. The other results are included to give a complete picture of the interaction between stability and the translation functors, which does not seem to appear in the literature.

Let $\tau, \tau^-\colon \Catrep \to \Catrep$ denote the Auslander--Reiten translations defined in \cref{subsection:auslander-reiten-translation}.
Let $\theta$ be a stability function for $Q$; recall from \cref{lemma:every-theta-is-theta-alpha} that since $Q$ is acyclic, the stability function $\theta$ can be identified with $\theta_\alpha$ (resp. $\eta_\beta$) for a unique dimension vector $\alpha$ (resp. $\beta$).

\begin{lemma}
  \label{lemma:auslander-reiten-semistable}
  Let $M$ be a $\theta$-semistable representation of $Q$.
  \begin{enumerate}[label=(\alph*)]
    \item \label{item:auslander-reiten-semistable-a}
    If $\theta = \theta_\alpha$, then $\tau^- M$ is $\eta_\alpha$-semistable.
    \item \label{item:auslander-reiten-semistable-b}
    If $\theta = \eta_\beta$, then $\tau M$ is $\theta_\beta$-semistable.
  \end{enumerate}
\end{lemma}

\begin{proof}
  We only show the second assertion, as the first follows similarly.
  By \cref{proposition:sst}~\ref{enumerate:sst-beta} there exists $m>0$
  and a representation $V$ of dimension vector $m \beta$ such that $\Hom(M,V) = 0$.
  Assume first that $M$ has no projective summands.
  By Auslander--Reiten duality we have
  \[
    \Ext(V, \tau M) \cong \Hom(M, V)^\vee = 0.
  \]
  Moreover, from \cref{proposition:AR-translation}~\ref{item:AR-duality} we have
  \[
    \dim \Hom(V, \tau M) - \dim \Ext(V, \tau M) = \langle V, \tau M \rangle = - \langle M, V \rangle = - m \langle \dimvect M, \beta \rangle = 0
  \]
  since $M$ is $\eta_\beta$-semistable. Hence also $\Hom(V, \tau M) = 0$, and so $\tau M$ is $\theta_\beta$-semistable by \cref{proposition:sst}~\ref{enumerate:sst-alpha}.

  In the general case, we can decompose $M$ into indecomposable summands
  and this decomposition is unique up to isomorphism. Thus, we can write $M = U \oplus P$, where $P$ is projective and $U$ has no projective summands.
  Since $M$ is $\eta_\beta$-semistable, both summands $U$ and $P$ are also $\eta_\beta$-semistable. As $U$ has no projective summands, we conclude that $\tau U$ is $\theta_\beta$-semistable. Moreover,
  \[
    \tau M = \tau U \oplus \tau P = \tau U,
  \] and so $\tau M$ is $\theta_\beta$-semistable as well.
\end{proof}

\begin{lemma}\label{lemma:analysis-stability}
  Let $M$ be a $\theta$-stable representation.
  \begin{enumerate}[label=(\alph*)]
    \item \label{item:analysis-stability-a} Suppose $\theta = \theta_\alpha$. If $\supp M \not\subset \supp \alpha$, then $\supp M \setminus \supp \alpha = \{j\}$ and $M \cong I_{Q'}(j)$, where $I_{Q'}(j)$ is the indecomposable injective of the full subquiver $Q'$ supported on $\supp \alpha \cup \{ j \}$, viewed as a representation of $Q$.
    \item\label{item:analysis-stability-b} Suppose $\theta = \eta_\beta$. If $\supp M \not\subset \supp \beta$, then $\supp M \setminus \supp \beta = \{j\}$ and $M \cong P_{Q'}(j)$, where $P_{Q'}(j)$ is the indecomposable projective of the full subquiver $Q'$ supported on $\supp \beta \cup \{ j \}$, viewed as a representation of $Q$.
  \end{enumerate}
\end{lemma}

\begin{proof}
  We again only show (b). Assume that $\supp M \not\subset \supp \beta$. Let $Q'' \subset Q' \subseteq Q$ denote the full subquivers on $\supp M \setminus \supp \beta$ and $\supp M \cup \supp \beta$ respectively. We first observe that if $V$ and $W$ are representations supported on $Q'$, then it follows from \eqref{eq:exact-seq-hom-ext} that
  \[ \Hom_{Q}(V, W) = \Hom_{Q'}(V, W) \quad \mathrm{and} \quad \Ext_{Q}(V, W) = \Ext_{Q'}(V, W), \]
  so we may drop the subscripts. The subquiver $Q''$ is also acyclic so it has a sink $j$, and by assumption $\beta_j = 0$ and $\dim M_j > 0$. Consider the projective representation $P_{Q'}(j)$ of $Q'$. For any representation $V$ supported on $Q'$, we have
  \[
    \langle P_{Q'}(j), V \rangle = \dim \Hom(P_{Q'}(j), V) = \dim V_j
  \]
  because $\Ext(P_{Q'}(j), V) = 0$ and $\Hom(P_{Q'}(j), V) \cong V_j$. This implies that
  \[
    \eta_\beta(P_{Q'}(j)) = -\langle \dimvect P_{Q'}(j), \beta \rangle = -\beta_j = 0
  \]
  whereas $\Hom(P_{Q'}(j), M) \cong M_j \neq 0$. Let $f \in \Hom(P_{Q'}(j), M)$ be a nonzero homomorphism and consider its kernel $P$. We have $P \subsetneq P_{Q'}(j)$, and since the category $\Catrep[k][Q']$ is hereditary, $P$ is again a projective representation of $Q'$.

  Suppose that $P \neq 0$. Any indecomposable direct summand of $P$ is of the form $P_{Q'}(i)$ for some vertex $i \in Q'_0$ for which there exists a path $j \to i$ in $Q'$ because it must embed into $P_{Q'}(j)$, and this path cannot have length $0$ because $P_j = 0$. As $j$ is a sink of $Q''$, we must have $i \in \supp \beta$. This shows that
  \[
    \eta_\beta(P_{Q'}(i)) = -\beta_i < 0.
  \]
  We conclude that $\eta_\beta(P) < 0$ and therefore $\eta_\beta(P_{Q'}(j)/P) > 0$. However, $P_{Q'}(j)/P$ embeds into $M$ via $f$, which contradicts the  fact that $M$ is stable. Thus, we have $P = 0$ and the map $f\colon P_{Q'}(j) \to M$ is injective. Since the image of $f$ is a nonzero subrepresentation of $M$ with $\eta_\beta(\im f) = \eta_\beta(P_{Q'}(j)) = 0$, we conclude that $f$ must also be surjective and thus $M \cong P_{Q'}(j)$.

  If $j'$ is another sink of $Q''$, then the same argument shows that also $M \cong P_{Q'}(j')$, which implies that $j = j'$ and so $Q'' = \{j\}$ as claimed.
\end{proof}

\begin{lemma}\label{lemma:dimension-inequality}
    Let $M$ be a $\theta$-stable representation of dimension vector $d$ and let $\epsilon \in \NN^{Q_0}$.
    \begin{enumerate}[label=(\alph*)]
    \item
    If $\theta = \theta_\alpha$ and $\langle \epsilon, d \rangle > 0$, then $m\alpha+\epsilon \ge d$ for $m \gg 0$.
    \item
    If $\theta = \eta_\beta$ and $\langle d, \epsilon \rangle > 0$, then $m\beta+\epsilon \ge d$ for $m \gg 0$.
  \end{enumerate}
\end{lemma}

\begin{proof}
  As above, we just prove the second claim. The result is clear if $\supp d \subseteq \supp \beta$. By \cref{lemma:analysis-stability}, the only other case is that $\supp M \setminus \supp \beta = \{ j \}$ and $M \cong P'(j)$, the indecomposable projective of the full subquiver on $\supp \beta \cup \{j\}$. In this case we have $d_j = (\dimvect{P'(j)})_j = 1$, while $\epsilon_j = \langle d, \epsilon \rangle > 0$, so in particular $\epsilon_j \ge d_j$.
\end{proof}

\begin{lemma} \label{lemma:auslander-reiten-stable}
Let $M$ be a $\theta$-stable representation.
  \begin{enumerate}[label=(\alph*)]
    \item If $\theta=\theta_\alpha$,
    then either $\tau^-M$ is $\eta_\alpha$-stable, or $M$ is isomorphic to an injective representation $I$ of the full subquiver $Q'$ supported on $\supp M \cup \supp \alpha$,
    viewed as a representation of~$Q$.
    \item If $\theta=\eta_\beta$, then either $\tau M$ is $\theta_\beta$-stable, or $M$ is isomorphic to a projective representation $P$ of the full subquiver $Q'$ supported on $\supp M \cup \supp \beta$,
    viewed as a representation of~$Q$.
  \end{enumerate}
\end{lemma}

\begin{proof}
  We again only show (b).
  Suppose that $M$ is not of the form $P$ as in the statement.
  We know by \cref{lemma:auslander-reiten-semistable}~\ref{item:auslander-reiten-semistable-b} that $\tau M$ is $\theta_\beta$-semistable and want to conclude that it is $\theta_\beta$-stable.
  Viewing $M$ as a representation of the full subquiver $Q'$ supported on $\supp M \cup \supp \beta$, our assumption means that $M$ is not projective.
  Therefore, by \cref{lemma:analysis-stability}~\ref{item:analysis-stability-b},
  we have $\supp \beta = Q'_0$.
  Let $\tau M \twoheadrightarrow U$ be a surjection such that $\theta_\beta(U) = 0$ and write $U = \bigoplus_\ell U^\ell$ as a direct sum of indecomposables.
  Each $U^\ell$ is a quotient of the semistable representation $\tau M$, so $\theta_\beta(U^\ell) \geq 0$, and since these quantities sum to $0$, we have
  $\theta_\beta(U^\ell) = 0$ for each $\ell$. Thus, we may assume that $U$ is itself indecomposable.

  The quotient $U$ cannot be injective, because if $U \cong I(i)$ for some $i \in Q'$, then
  \[
    0 = \theta_\beta(U) = \langle \beta , \dimvect I(i) \rangle = \beta_i,
  \]
  whereas $\beta_i > 0$ since $i \in \supp{\beta}$.
  Thus,~$U \cong \tau\tau^- U$ by \cref{proposition:AR-translation}~\ref{item:AR-inverse}.
  By \cref{proposition:AR-translation}~\ref{item:AR-duality}, the functor $\tau^-$ is a left adjoint, thus right exact,
  so $\tau^- U$ is a quotient of $M$, as $\tau^- \tau M$ is in any case a quotient of $M$.
  Using \cref{proposition:AR-translation}~\ref{item:AR-duality} we obtain
  \[
    \eta_\beta(\tau^- U) =
    - \langle \dimvect \tau^- U , \beta \rangle =
    \langle \beta , \dimvect U \rangle = \theta_\beta(U) = 0 .
  \]
  As $M$ is $\eta_\beta$-stable, we have either $\tau^- U = 0$ or $\tau^- U = M$, and since $U \cong \tau \tau^- U$, we conclude that either $U = 0$ or $U = \tau M$ which proves the claim.
\end{proof}

\subsection{Generic vanishing of Ext}\label{subsection:generic-ext-vanishing}

We continue to assume that $Q$ is acyclic.

\begin{lemma}\label{lemma:ext-vanishing-semistability}
  Let $M$ be a $\theta$-semistable representation and let $\epsilon \in \mathbb{N}^{Q_0}$.
  \begin{enumerate}[label=(\alph*)]
    \item If $\theta = \theta_\alpha$ and $\epsilon = \dimvect P$ for some
    projective representation $P$, then for all sufficiently large integers $m$, a general
    representation $V$ of dimension vector $m\alpha + \epsilon$
    satisfies $\Ext(V, M) = 0$. In fact, it is enough that
    \[ m > \frac{\langle\gamma,\gamma\rangle}{\langle\gamma,\alpha\rangle}
    \quad \mbox{for all dimension vectors} \quad
    \gamma < \dimvect \tau^- M.\]
    \item If $\theta = \eta_\beta$ and $\epsilon = \dimvect I$ for some
    injective representation $I$, then for all sufficiently large integers $m$, a general
    representation $V$ of dimension vector $m\beta + \epsilon$
    satisfies $\Ext(M, V) = 0$. In fact, it is enough that
    \[ m > \frac{\langle\gamma,\gamma\rangle}{\langle\beta,\gamma\rangle}
    \quad \mbox{for all dimension vectors} \quad
    \gamma < \dimvect \tau M.\]
  \end{enumerate}
\end{lemma}

\begin{proof}
  We prove (b). By \cref{lemma:auslander-reiten-semistable}, the representation $\tau M$ is $\theta_\beta$-semistable, so \cref{corollary:hom-vanishing} (a) with $\epsilon = 0$ implies that
  for $m$ satisfying the inequality, there exists a representation $V'$ of dimension vector $m \beta$ such that $\Hom(V', \tau M) = 0$. By Auslander-Reiten duality, this implies $\Ext(M, V') \cong \Hom(V', \tau M)^\vee = 0$.

  Now the representation $V = V' \oplus I$ has dimension vector $m \beta + \epsilon$ and satisfies
  \[ \Ext(M, V) = \Ext(M, V') \oplus \Ext(M, I) = 0 \]
  since $I$ is injective. Thus, by upper semicontinuity this must hold for a general representation of dimension vector $m\beta + \epsilon$.
\end{proof}

\subsection{Separating stable representations}\label{subsection:sep-det-inv}

In this section, we assume that $Q$ is acyclic and $\theta = \eta_\beta$ for concreteness, and leave formulating the dual statements for $\theta = \theta_\alpha$ to the reader. Our aim is to prove the following result.
\begin{theorem}
  \label{theorem:separation}
  Let~$M^0,M^1,\ldots,M^r$ be non-isomorphic~$\eta_\beta$-stable representations. For all sufficiently large integers~$m$, there exists a representation~$N$ of dimension vector~$m\beta$ such that
  \begin{equation}
    \Hom(M^0,N)\neq 0,\quad \text{and} \quad \Hom(M^\ell,N)=0\quad \mathrm{for} \; \ell=1,\ldots,r.
  \end{equation}
\end{theorem}
This will be used in \cref{theorem:projectivity} below by considering two polystable representations $M$ and $M'$ such that $M^1, \ldots, M^r$ are the non-isomorphic stable summands of $M'$ while $M^0$ appears as a stable summand of $M$.
In view of \cref{proposition:vanishing-sections}, we will see that $\sigma_N$ separates the polystable representations $M$ and $M'$; this will ultimately enable us to prove ampleness of the determinantal line bundle on the moduli space of semistable representations in Section~\ref{section:projectivity}.

Our argument is inspired by the proof of a similar statement for moduli of vector bundles on a curve,
due to Esteves \cite[Section~5]{MR1695802} and Esteves--Popa \cite[Section~3]{esteves.popa:2004:veryampleness}.

We break up the proof of \cref{theorem:separation} into several steps.

\begin{proposition}\label{theorem:separation-base-case}
  For~$M^0,M^1,\ldots,M^r$ as in \cref{theorem:separation} and for all sufficiently large integers~$m$, there exists a representation $N$
  such that
  \begin{enumerate}[label=(\roman*)]
    \item\label{enumerate:separation-1} $\Hom(M^0, N) \neq 0$;
    \item\label{enumerate:separation-2} $\Hom(M^\ell, N) = 0\quad \mathrm{for} \; \ell =1, \dots, r$;
    \item\label{enumerate:separation-3} $\dimvect N = m \beta + \epsilon$,
      where $\epsilon$ is the dimension vector of an injective representation and $\supp \epsilon \cap \supp M^0 = \varnothing$.
  \end{enumerate}
\end{proposition}
Before proving \cref{theorem:separation-base-case} we first use it to establish \cref{theorem:separation}.
\begin{proof}[Proof of \cref{theorem:separation}]
  Let $N$ be as in \cref{theorem:separation-base-case}. If $\epsilon = 0$, we are already done, so assume $\epsilon > 0$. It suffices to find a subrepresentation $N' \subset N$ that satisfies properties \ref{enumerate:separation-1}-\ref{enumerate:separation-3} with $\epsilon' < \epsilon$, since repeating the construction results in a sequence of subrepresentations with the same properties, and the sequence must terminate at a subrepresentation with $\epsilon = 0$.

  By assumption we have $\epsilon = \dimvect I$ for some nonzero injective representation $I$. Since $\langle V, I \rangle = \dim \Hom(V, I) \ge 0$ for any representation $V$, and since $\beta$ is a dimension vector since $Q$ is acyclic, we have
  \[
    \dim \Hom(N, I) = \langle N, I \rangle = \langle m \beta + \epsilon, \epsilon \rangle \ge \langle \epsilon, \epsilon \rangle = \dim \Hom(I,I) \ge 1,
  \]
  so there is a nonzero map $f \colon N \to I$. Let $N' \subset N$ and $I'$ denote the kernel and cokernel of $f$ respectively and set $\epsilon' = \dimvect I'$. Note that $\epsilon' < \epsilon$. We claim that $N'$ satisfies properties \ref{enumerate:separation-1}-\ref{enumerate:separation-3}.

  To verify \ref{enumerate:separation-1} and \ref{enumerate:separation-2} for $N'$, we apply $\Hom(M^\ell, \blank)$ to the exact sequence
  \[
    0 \to N' \to N \xrightarrow{f} I
  \]
  to get
  \[
    0 \to \Hom(M^\ell, N') \to \Hom(M^\ell, N) \to \Hom(M^\ell, I).
  \]
  Since $\Hom(M^\ell, N) = 0$ for $\ell \ge 1$, we also have $\Hom(M^\ell, N') = 0$,
  giving \ref{enumerate:separation-2}.
  For \ref{enumerate:separation-1}, we have $\Hom(M^0, I) = 0$ since $\supp I \cap \supp M^0 = \varnothing$, and so
  \[
    \Hom(M^0, N') \cong \Hom(M^0, N) \neq 0.
  \]

  For \ref{enumerate:separation-3}, from the exact sequence $0 \to N' \to N \to I \to I' \to 0$
  we obtain
  \[
    \dimvect N' = \dimvect N - \dimvect I + \dimvect I' = m \beta + \epsilon - \epsilon + \epsilon' = m \beta + \epsilon'.
  \]
  Clearly $\supp \epsilon' \cap \supp M^0 = \varnothing$ since $\supp \epsilon' \subseteq \supp \epsilon$, and moreover, as $I'$ is a quotient of $I$, it is also injective since $\Catrep$ is hereditary. Thus, we have proven \ref{enumerate:separation-3} for $N'$.
\end{proof}

The rest of the section is devoted to proving \cref{theorem:separation-base-case}.
Fix an admissible ordering of $Q_0$ as in \cref{sec:quiverrepr},
let $i_0 \in \supp M^0$ be the minimal vertex in the support of $M^0$,
and set $\epsilon_0\colonequals\dimvect I(i_0)$.
Note that for any dimension vector $\xi$ we have
\[
  \langle \xi, \epsilon_0 \rangle = \xi_{i_0}
\]
and that $\supp \epsilon_0 \cap \supp M^0 = \{i_0\}$ by the choice of $i_0$.

\begin{lemma}
  \label{lemma:Mell-to-V}
  Let~$M^0,M^1,\ldots,M^r$ be as in \cref{theorem:separation} and let $m_\ell = \dim (M^\ell)_{i_0}$, where~$i_0$ is the minimal vertex in $\supp M^0$. For all sufficiently large integers~$m$,
  there exists a representation $V$ of dimension vector $m \beta + \epsilon_0$ such that
  \begin{itemize}
    \item $\Ext(M^\ell, V) = 0$ for each $\ell = 0, \ldots, r$, and so
      \[
        \dim \Hom(M^\ell, V) = \langle \dimvect M^\ell, m \beta + \epsilon_0 \rangle = \langle \dimvect M^\ell, \epsilon_0 \rangle = m_\ell;
      \]
    \item For $\ell = 0, \ldots, r$, every nonzero map $M^\ell \to V$ is injective;
    \item For $\ell = 1, \ldots, r$, every nonzero map $f\colon M^0 \oplus M^\ell \to V$ satisfies
      \[
        \langle \dimvect \ker f, \beta \rangle = 0.
      \]
  \end{itemize}
\end{lemma}

\begin{proof}
  Since each $M^\ell$
  and each $M^0 \oplus M^\ell$ is $\eta_\beta$-semistable and $\epsilon_0$ is the dimension vector of an injective representation, we obtain the claim by applying
  \cref{lemma:ext-vanishing-semistability}(b)
  to each $M^\ell$ for $\ell = 0, \ldots, r$, and by applying
  \cref{lemma:image-with-bound}~\ref{enumerate:image-with-bound-2} to each $M^\ell$ for $\ell = 0, \ldots, r$ as well as each $M^0 \oplus M^\ell$ for $\ell = 1, \ldots, r$.
\end{proof}

Denote the cokernel of $\phi$ by $W$, so that we have an exact sequence
\[    \begin{tikzcd}
  0 \arrow[r]& M^0 \arrow[r, "\phi"] & V \arrow[r]&W  \arrow[r]& 0.\end{tikzcd}
\]

\begin{lemma}
  \label{lemma:sum-Mell-injects-into-V}
  If $\ell \ge 1$ and $\psi\colon M^\ell \to V$ is a nonzero map, then the induced map $\overline \psi\colon M^\ell \to W$ is injective.
\end{lemma}

\begin{proof}
  By \cref{lemma:Mell-to-V},
  the map $\psi$ is injective
  and the kernel $K$ of the induced map
  \[
    (\phi, \psi)\colon M^0 \oplus M^\ell \longrightarrow V
  \]
  satisfies $\langle \dimvect K, \beta \rangle = 0$.
  Since $M^0$ and $M^\ell$ are $\eta_\beta$-stable and non-isomorphic,
  the only nonzero subrepresentations of $M^0 \oplus M^\ell$ with this property
  are $M^0, M^\ell$, and $M^0 \oplus M^\ell$.
  Since both $\phi$ and $\psi$ are injective,
  we must have $K = 0$.
  The successive inclusions
  \[ \begin{tikzcd}
  M^0 \arrow[hookrightarrow]{r} & M^0 \oplus M^\ell \arrow{r}{(\phi,\psi)}  & V
  \end{tikzcd}
  \]
  induce a short exact sequence
  \[\begin{tikzcd}
    0 \arrow{r} & M^\ell \arrow{r}{\overline{\psi}} & W \arrow{r} & V / (M^0 \oplus M^\ell)  \arrow{r} & 0, \end{tikzcd}
  \]
  so we see that the induced map $\overline{\psi}\colon M^\ell \xrightarrow{\psi} V \to W$ is injective.
\end{proof}

\begin{lemma}
  \label{lemma:hyperplane-construction}
  There exists a hyperplane $H \subset V_{i_0}$
  such that $\phi_{i_0}(M^0_{i_0}) \subseteq H$
  but $\psi_{i_0}(M^\ell_{i_0}) \nsubseteq H$ for every $\ell \ge 1$ and every nonzero $\psi\colon M^\ell \to V$.
\end{lemma}

\begin{proof}
  As before, let $m_\ell\colonequals\dim (M^\ell)_{i_0}$.
  Consider the Grassmannian $\Gr(m_\ell, W_{i_0})$
  of $m_\ell$-dimensional subspaces of the vector space $W_{i_0}$.
  By \cref{lemma:sum-Mell-injects-into-V}, we obtain a morphism
  \[
    \mathbb{P}(\Hom(M^\ell, V)) \longrightarrow \Gr(m_\ell, W_{i_0})
  \]
  that sends a map $\psi \colon M^\ell \to V$ to the subspace $\im (\overline{\psi})_{i_0} \subseteq W_{i_0}$.
  The image $X_\ell \subseteq \Gr(m_\ell, W_{i_0})$ of this morphism
  has dimension at most $\dim \mathbb{P}(\Hom(M^\ell, V)) = m_\ell - 1$.

  For a hyperplane $P \subseteq W_{i_0}$,
  denote by $Y_{\ell, P} \subseteq \Gr(m_\ell, W_{i_0})$ the Schubert variety of subspaces contained in $P$.
  Since the codimension of $Y_{\ell, P}$ is $m_\ell$,
  by the Bertini--Kleiman theorem there exists a hyperplane $P$
  such that $Y_{\ell, P} \cap X_\ell = \varnothing$ for each $\ell = 1,\ldots, r$.
  The preimage $H \subseteq V_{i_0}$ of this $P$ satisfies the conditions in the lemma.
\end{proof}

\begin{proof}[Proof of \cref{theorem:separation-base-case}]
  Let $H \subset V_{i_0}$ be a hyperplane as in \cref{lemma:hyperplane-construction} and let $p\colon V_{i_0} \to k$ be a linear map with kernel equal to $H$.
  Recall that restricting to $V_{i_0}$ gives an isomorphism $\Hom_Q(V, I(i_0)) \xrightarrow{\sim} \Hom_k(V_{i_0},k)$.
  Let $\pi$ be the unique map $V \to I(i_0)$ corresponding to $p$ under this isomorphism.
  We claim that the representation $N = \ker\pi$ satisfies conditions \ref{enumerate:separation-1}-\ref{enumerate:separation-2} in the statement.

  Note that a morphism $f\colon M^\ell \to V$ factors through $N$ if and only if the composition $\pi \circ f\colon M^\ell \to I(i_0)$ is zero, if and only if the composition $p \circ f_{i_0}\colon M^\ell_{i_0} \to I(i_0)_{i_0}$ is zero, if and only if $f_{i_0}(M^\ell_{i_0}) \subset H$. Thus, it follows from the choice of $H$ that $\phi\colon M^0 \to V$ factors through $N$ but no nonzero map $M^\ell \to V$ does for $\ell = 1,\ldots, r$, which proves \ref{enumerate:determinantal-1} and \ref{enumerate:determinantal-2}.

  For \ref{enumerate:determinantal-3}, we note that the representation $I = \coker\pi$ is injective, as it is a quotient of $I(i_0)$ and the category of representations is hereditary. Setting $\epsilon = \dimvect I$, the exact sequence
  \[
    0 \to N \to V \xrightarrow{\pi} I(i_0) \to I \to 0.
  \]
  implies that
  \[ \dimvect N = \dimvect V - \dimvect I(i_0) + \dimvect I = m \beta + \epsilon_0 - \epsilon_0 + \epsilon = m\beta+\epsilon. \]
  Finally, since $\supp I(i_0) \cap \supp M^0 = \{i_0\}$ and $(\im \pi)_{i_0} = (I(i_0))_{i_0}$,
  we must have $\supp \epsilon \cap \supp M^0 = \varnothing$.
\end{proof}

\section{Moduli spaces of quiver representations}
\label{section:good-moduli-spaces}

In this section we show that under certain assumptions (see \cref{remark:stability-function}), stacks of semistable quiver representations
admit adequate moduli spaces by verifying the existence criteria of~\cite{1812.01128}.

Throughout this section we let~$S$ denote a noetherian scheme.
The locally noetherian hypothesis is required
for the notions of $\Theta$-reductivity and S-completeness to be well-defined
and for the existence criteria to be applicable as in \cref{subsection:existence-criteria},
while we add a quasi-compactness condition to ensure that points specialize to closed points,
so that the local reductivity in \cref{subsection:local-reductivity} is better behaved.

\subsection{Good and adequate moduli spaces}

We begin by recalling the definition of good and adequate moduli spaces due to Alper \cite{MR3237451, MR3272912}.

\begin{definition}
  \label{definition:moduli-space}
  Let $\cX$ be a quasi-separated algebraic stack over $S$. An \emph{adequate moduli space} is a  quasi-compact quasi-separated  morphism~$f\colon\cX\to X$ to an algebraic space~$X$ over $S$ such that
  \begin{enumerate}[label=(\roman*)]
    \item $f$ is \emph{adequately affine}, meaning that for every surjection $\cA \to \cB$ of quasi-coherent $\cO_\cX$-algebras and every section $s$ of $f_*\cB$ over an affine \'etale neighborhood $U$ of $X$, some power of $s$ lifts to a section of $f_*\cA$ over $U$, and
    \item\label{enumerate:moduli-space-structure-sheaf} the natural morphism $\cO_X\to f_*\cO_{\cX}$ is an isomorphism.
  \end{enumerate}
  The map $f$ is a \emph{good moduli space} if instead of (i), we have the stronger hypothesis
  \begin{enumerate}[label=(\roman*$^\prime$)]
    \item $f$ is \emph{cohomologically affine}, meaning that the functor~$f_*\colon\operatorname{Qcoh}\cX\to\operatorname{Qcoh}X$ is exact.
  \end{enumerate}
\end{definition}

The notion of a good moduli space is inspired by good quotients in GIT.

\begin{example}
  \cite[Section~13]{MR3237451}
  If $G$ is a linearly reductive algebraic group 
  over a field $k$ acting on an affine $k$-scheme $\Spec(A)$,
  then a good moduli space of the quotient stack is given by the affine GIT quotient:
  \[
    [\Spec(A)/G] \rightarrow \Spec (A) \sslash G \coloneqq \Spec(A^G).
  \]
  More generally, if $G$ acts on a quasi-projective $k$-scheme $X$ with a fixed ample $G$-linearization
  and $X^{\mathrm{ss}} \rightarrow X \sslash G$ denotes the corresponding GIT quotient, then
  \[
    [X^{\mathrm{ss}}/G] \rightarrow X \sslash G
  \]
  is a good moduli space.
\end{example}

In this example, it is vital that $G$ is linearly reductive, since $\pi\colon \mathrm{B}G \rightarrow \Spec(k)$ is cohomologically affine if and only if $G$ is linearly reductive, that is, the functor $\pi_*$, which corresponds to taking $G$-invariants, is exact \cite[Proposition 12.2]{MR3237451}. In characteristic $0$, the notions of linearly reductive and reductive coincide; however, in positive characteristic, many reductive groups, $\GL_n$ among them, are not linearly reductive.

In order to bridge this gap,
Alper introduced the broader notion of adequate moduli spaces in \cite{MR3272912} which also covers many interesting cases in positive characteristic. In particular a flat, separated, finitely presented group scheme $G$ over $S$ is reductive if and only if the morphism $\mathrm{B}G \rightarrow S$ is adequately affine.
Consequently, Alper's notion of adequate moduli space enables a generalization of GIT to stacks for all reductive groups in arbitrary characteristic.

\begin{example}
  \cite[Section~9]{MR3272912}
  If $G$ is a smooth affine reductive group over a field $k$ acting on an affine $k$-scheme $\Spec(A)$, then
  \[
    [\Spec(A)/G] \rightarrow \Spec (A) \sslash G \coloneqq \Spec(A^G)
  \]
  is an adequate moduli space.
  More generally, for a reductive $k$-group $G$ acting on a quasi-projective $k$-scheme $X$ with an ample $G$-linearization,
  \[
    [X^{\mathrm{ss}}/G] \rightarrow X \sslash G
  \]
  is an adequate moduli space.
\end{example}

Since the stacks $\modulistack{d}$ and $\modulistack[\semistable{\theta}]{d}$
can both be described as quotient stacks via such a GIT setup (see \cref{proposition:smooth-cover-by-Rep}),
one can conclude that these moduli stacks have adequate moduli spaces given by the GIT quotient.
We will instead take a modern approach which avoids GIT and establish the existence of adequate
moduli spaces by applying the criteria of \cite{1812.01128}.

We summarize some of the properties of good and adequate moduli spaces that will be relevant for us.
Recall that~$S$ is assumed to be noetherian, hence quasi-separated, which is a standing assumption in the works we build on.

\begin{theorem}
\label{theorem:gms-properties}
  Let $\cX$ be an algebraic stack of finite type over $S$, and let $f\colon\cX \to X$ be an adequate moduli space.
  \begin{enumerate}[label=(\roman*)]
    \item\label{enumerate:ams-surj-univclosed-initial} \cite[Theorems 5.3.1 and 7.2.1]{MR3272912} The map $f$ is surjective, universally closed, and initial for maps to schemes and to separated algebraic spaces over $S$.

    \cite[Theorem 6.6]{MR3237451} If $f$ is a good moduli space, it is initial for maps to algebraic spaces.

    In particular, adequate moduli spaces which are schemes and good moduli spaces are unique up to a unique isomorphism.

    \item\label{enumerate:ams-finite-type} \cite[Theorem 6.3.3]{MR3272912} The algebraic space $X$ is of finite type over $S$.

    \item\label{enumerate:gms-points} \cite[Theorem 5.3.1]{MR3272912} For an algebraically closed $\cO_S$-field $k$, the map $f \times_S k$ identifies two $k$-points $x, y\colon \Spec k \to \cX \times_S k$ if and only if the closures of $\{x\}$ and $\{y\}$ in $|\cX \times_S k|$ intersect.

    \item\label{enumerate:ams-commutes-with-flat-base-change} \cite[Proposition 5.2.9~(1)]{MR3272912} If $X' \to X$ is a flat morphism of algebraic spaces, then $\cX \times_X X' \to X'$ is an adequate moduli space. In particular, if $S' \to S$ is a flat morphism of schemes, then $\cX \times_S S' \to X \times_S S'$ is an adequate moduli space.

      \cite[Proposition 4.7~(i)]{MR3237451} If $\cX \to X$ is a good moduli space, then for any morphism $X' \to X$ the base change $\cX \times_X X' \to X'$ is a good moduli space.

    \item\label{enumerate:ams-base-change-univ-homeo} \cite[Proposition 5.2.9~(3)]{MR3272912} More generally, if $X' \to X$ is any morphism of algebraic spaces, then the morphism $\cX \times_X X' \to X'$ is a universal homeomorphism.
    In particular, if $S' \to S$ is any morphism of schemes, then the closed points of $\cX \times_S S'$ are in natural bijection with those of $X \times_S S'$.
  \end{enumerate}
\end{theorem}

Condition \ref{enumerate:moduli-space-structure-sheaf} in \cref{definition:moduli-space} allows us to prove the following.
\begin{lemma}\label{lemma:descending-vector-bundles-to-ams}
Let $\cX$ be an algebraic stack over a scheme $S$ and let $f\colon \cX \to X$ be an adequate moduli space.
  \begin{enumerate}[label=(\roman*)]
      \item \label{item:descending-vector-bundles-to-ams-i} The pullback $f^*$ is fully faithful for vector bundles on $X$.
  \end{enumerate}
  In particular, given a vector bundle $\cF$ on $\cX$, there is up to isomorphism at most one vector bundle $F$ on $X$ such that $\cF \cong f^*F$. If such an $F$ exists, we say that $\cF$ \emph{descends} to $X$.
  \begin{enumerate}[label=(\roman*)]
    \setcounter{enumi}{1}
      \item \label{item:descending-vector-bundles-to-ams-ii} Let $X' \to X$ be an fpqc cover, let $\cX' \to X'$ be the base change of $\cX$ along $X'$ with the second projection morphism $u\colon \cX' \to \cX$, and let $\cF$ be a vector bundle on $\cX$.
      If $u^*\cF$ descends to $X'$, then $\cF$ descends to $X$.

      \item \label{item:descending-vector-bundles-to-ams-iii} If a line bundle $\cL$ on $\cX$ is generated by finitely many global sections, then $\cL$ descends to $X$.
  \end{enumerate}
\end{lemma}

\proof \ref{item:descending-vector-bundles-to-ams-i} Given vector bundles $F$ and $G$ on $X$, using the projection formula and the condition $f_* \cO_\cX = \cO_X$, we have
\[ \Hom_{\cX}(f^*F, f^*G) = \Hom_X(F, f_* f^* G) = \Hom_X(F, f_* \cO_\cX \otimes G) = \Hom_X(F, G). \]

\ref{item:descending-vector-bundles-to-ams-ii} The vector bundle $\cF' = u^*\cF$ comes equipped with a canonical descent datum; that is, there is a canonical isomorphism $\sigma \colon \mathrm{pr}_1^* \cF' \xrightarrow{\; \sim \;} \mathrm{pr}_2^* \cF'$, where $\mathrm{pr}_1, \mathrm{pr}_2 \colon \cX' \times_\cX \cX' \to \cX'$ are the projections, and $\sigma$ satisfies the cocycle condition on $\cX'\times_\cX \cX' \times_\cX \cX'$.
By assumption, $\cF'$ descends to a vector bundle $F'$ on $X'$, and by \cref{theorem:gms-properties}~\ref{enumerate:ams-commutes-with-flat-base-change},  $X'$ is an adequate moduli space of $\cX'$,
so it follows from \ref{item:descending-vector-bundles-to-ams-i} that $\sigma$ descends to an isomorphism $\tau\colon \text{pr}_1^* F' \cong \mathrm{pr}_2^*F'$, where $\mathrm{pr}_1, \mathrm{pr}_2 \colon X' \times_X X' \to X'$ are the projections.
Moreover $X'\times_X X' \times_X X'$ is an adequate moduli space for $\cX'\times_\cX \cX' \times_\cX \cX'$, so the isomorphism $\tau$ satisfies the cocycle condition on $X'\times_X X' \times_X X'$. In other words, we obtain a descent datum for $F'$, so there exists a unique vector bundle $F$ on $X$ whose restriction to $X'$ is $F'$, and since the descent datum for $F$ pulls back to that of $\cF$, it follows that $f^*F \cong \cF$.

\ref{item:descending-vector-bundles-to-ams-iii}  Let $s_0, \ldots, s_n \in \Gamma(\cX, \cL)$ be global sections that generate $\cL$ and denote by $\phi \colon\cX \to \PP^n$ the morphism induced by these sections. Since the adequate moduli space~$f \colon \cX \to X$ is initial for maps to separated algebraic spaces, there is a map $\psi \colon X \to \PP^n$ such that $\phi = \psi \circ f$. This implies that the line bundle $L:= \psi^*\cO_{\PP^n}(1)$ satisfies $f^*L = \cL$, proving the claim.
\endproof

If $f \colon \cX \to X$ is a good moduli space, then a more general result holds: the pullback $f^*$ is fully faithful for all quasi-coherent sheaves, and in fact quasi-coherent sheaves satisfy descent along $f$ \cite[Lemma 2.12]{Sveta}. We do not know if the analogous result holds for adequate moduli spaces.

\subsection{Existence criteria for moduli spaces}
\label{subsection:existence-criteria}

The existence criteria for good and adequate moduli spaces of \cite{1812.01128} are expressed in terms of two valuative criteria for algebraic stacks, known as $\Theta$-reductivity and S-completeness, that we now recall.
Both are certain codimension-2 filling conditions and,
in order to specify them,
we introduce two quotient stacks associated to a given discrete valuation ring.

For a DVR $R$ with uniformizer $\pi \in R$, fraction field $K$, and residue field~$\kappa$, we define
\[
  \Theta_R = [\Spec(R[t])/\Gm] \quad \mathrm{and} \quad \STbar_R = \left[
  \quot{
    \Spec\left(\frac{R[s,t]}{s t - \pi}\right)
    }
    {\Gm} \right],
\]
where $\Gm$ acts with weight $1$ on $s$ and weight $-1$ on $t$.
We denote by $0 \in \Theta_R$ and $0 \in \STbar_R$ the unique closed point in each stack.

\begin{definition}
  \label{def:ThetaS}
  An algebraic stack $\cX$ of finite type over $S$ is \emph{$\Theta$-reductive},
  respectively S\emph{-complete},
  if for any discrete valuation ring $R$ and any commutative diagram of solid arrows on the left,
  respectively on the right,
  \begin{equation}
    \begin{tikzcd}
      \Theta_R\setminus\{0\} \arrow[r] \arrow[d, hook] & \cX \arrow[d] \\
      \Theta_R \arrow[r] \arrow[ru, dashed, "\exists!"] & S
    \end{tikzcd}
  \qquad \qquad
    \begin{tikzcd}
      \overline{\mathrm{ST}}_R\setminus\{0\} \arrow[r] \arrow[d, hook] & \cX \arrow[d] \\
      \overline{\mathrm{ST}}_R \arrow[r] \arrow[ru, dashed, "\exists!"] & S
    \end{tikzcd}
  \end{equation}
  there exists a unique dashed arrow making the diagram commute.
\end{definition}

\begin{remark}
\label{remark:hartogs}
  A stack $\cX$ satisfies \emph{Hartogs's principle} if for any regular local ring $A$ of dimension 2
  with closed point $0 \in \Spec A$ and any map $\Spec{A} \setminus \{0\} \to \cX$, there exists a unique extension $\Spec{A} \to \cX$ \cite[Remark 3.51]{1812.01128}. In particular, it follows from descent that $\cX$ satisfying Hartogs's principle is both $\Theta$-reductive and S-complete. Moreover, if $\cX$ has affine diagonal, any such extension is unique if it exists.
\end{remark}

We can now state the existence criteria of Alper--Halpern-Leistner--Heinloth (in a slightly more restrictive version).

\begin{theorem}[{\cite[Theorem 5.4]{1812.01128}}] \label{theorem:AHLH}
  Let $\cX$ be an algebraic stack of finite type, with affine stabilizers and separated diagonal,
  over a noetherian scheme $S$.
  \begin{enumerate}[label=(\roman*)]
    \item\label{enumerate:gms-existence}
      If $S$ is a scheme of characteristic $0$,
      then $\cX$ admits a separated good moduli space over $S$ if and only if $\cX$ is $\Theta$-reductive and S-complete.
    \item\label{enumerate:ams-existence}
      If $\cX$ is locally reductive (see \cref{definition:locallyred} below),
      then $\cX$ admits a separated adequate moduli space if and only if $\cX$ is $\Theta$-reductive and S-complete.
  \end{enumerate}
\end{theorem}
When~$\cX$ has affine diagonal -- as it will have in our case -- the condition that it has affine stabilizers and separated diagonal is automatic.

\begin{definition}\label{definition:locallyred}
  A quasi-separated algebraic stack $\cX$ with affine stabilizers is \emph{locally reductive} if every point of $\cX$ specializes to a closed point and for every closed point $x \in \cX$, there exists a pointed \'etale morphism
  $([\Spec(A)/ \GL_n], w) \to (X,x)$
  inducing an isomorphism of stabilizers at $w$.
\end{definition}

If $\cX$ is S-complete, the automorphism group of any closed point is reductive \cite[Proposition 3.47]{1812.01128}.
In characteristic $0$, this implies that $\cX$ has \'{e}tale local presentations by quotient stacks of the form $[\Spec(A)/ G]$ with $G$ linearly reductive \cite{MR4088350}.
In positive characteristic, one instead has to assume the existence of such local presentations.
In \cref{subsection:ThetaS}, we show that the stacks $\modulistack{d,S}$ and $\modulistack[\semistable{\theta}]{d,S}$ are $\Theta$-reductive and S-complete, which implies that they have good moduli spaces when $S$ is a scheme of characteristic $0$.
However, to obtain adequate moduli spaces in positive characteristic, we show in \cref{subsection:local-reductivity} that $\modulistack[\semistable{\theta}]{d,S}$ is locally reductive under the additional assumption $\theta = \theta_\beta$ or $\theta=\eta_\beta$ for a dimension vector $\beta$.

\subsection{Points of moduli spaces of quiver representations}
\label{subsection: points of moduli spaces}

In this section, we describe closed points of the moduli stacks $\modulistack{d,S}$ and $\modulistack[\semistable{\theta}]{d,S}$.

\begin{lemma}\label{lemma:JH-filtration-closure}
  Let $k$ be a field and let $M$ be a $\theta$-semistable representation of dimension vector $d$. The point $|\gr{M}| \in |\modulistack[\semistable{\theta}]{d,k}|$ corresponding to the associated graded object $\gr{M}$ of the Jordan-H\"older filtration lies in the closure of $|M|$.
  Consequently, a closed point of $|\modulistack[\semistable{\theta}]{d,k}|$ corresponds to a geometrically polystable representation.
\end{lemma}

\begin{proof}
  Given a nonsplit short exact sequence $0 \to N' \to N \to N'' \to 0$ of representations where $\dimvect{N} = d$, the line $\bbA^1$ in $\Ext(N'', N')$
  spanned by this class parameterizes a family of representations $\mathcal{N}$ such that $\mathcal{N}_t \cong N$ for $t \neq 0$ and $\mathcal{N}_0 \cong N' \oplus N''$. Considering the induced map $\bbA^1 \to \modulistack{d}$, we see that the point $|N' \oplus N''| \in |\modulistack{d}|$ is in the closure of the point $|N|$.

  If now
  \[
  0 = M^0 \subset M^1 \subset M^2 \subset \cdots \subset M^{r-1} \subset M^r = M
  \]
  is a Jordan-H\"older filtration of $M$, then by inductively applying the above argument, we see that the closure of $|M|$ contains the point corresponding to
  \[ \left( \bigoplus_{\ell=1}^i M^\ell/M^{\ell-1} \right) \oplus M/M^i \]
  for each $i = 1, \ldots, r$, hence in particular the point $|\gr{M}|$.

  Now if $M$ corresponds to a closed point in $|\modulistack[\semistable{\theta}]{d,k}|$, then it defines the same point as its base change $M'$ to an algebraically closed field. By the above, $|\gr M'|$ lies in the closure of $|M'|$ which only contains one point, hence $M' = \gr M'$, and so $M$ is geometrically polystable.
\end{proof}

\begin{proposition}
  \label{proposition:closed-points-arbitrary-base}
  Let $\pi\colon \modulistack[\semistable{\theta}]{d,S} \to S$ denote the structure morphism and let $p \in |\modulistack[\semistable{\theta}]{d,S}|$ be a point. If $\pi(p) \in S$ is closed and $p$ is represented by a geometrically polystable representation $M$ defined over a finite extension $L$ of the residue field $\kappa(\pi(p))$, then $p$ is closed. If $S$ is a Jacobson scheme, the converse holds.
\end{proposition}

\begin{proof}
  Suppose that $\pi(p)$ is closed and that $p$ is represented by a geometrically polystable representation $M$ defined over a finite extension $L$ of $\kappa(\pi(p))$. Since $\pi$ is continuous, the fiber $\pi^{-1}(\pi(p))$ is closed, so we may assume $S = \Spec{k}$ where $k = \kappa(\pi(p))$. Suppose for a contradiction that $p$ is not closed. The closed substack with underlying set $\overline{\{p\}}$ with its reduced substack structure is of finite type over $k$, hence contains a closed point $p'$ represented by a map $\Spec{L'} \to \modulistack[\semistable{\theta}]{d,k}$ corresponding to a representation $M'$. If $M'$ is not geometrically polystable, then there is a finite extension $L''$ such that $M' \otimes_{L'} L''$ is not polystable, and by \cref{lemma:JH-filtration-closure}, the point $|\gr(M' \otimes_{L'} L'')|$ is in the closure of $p'$, contrary to the fact that $p'$ is closed. Thus, $M'$ is geometrically polystable.

  Let now $K$ be a compositum of $L$ and $L'$ over $k$. On the one hand, by assumption $p \neq p'$, so $M \otimes K$ and $M' \otimes K$ are not isomorphic. On the other hand, for any stable summand $E$ of $M \otimes K$, by upper semicontinuity we have
  \[ \dim \Hom(E, M \otimes K) \le \dim \Hom(E, M' \otimes K). \]
  Since both $M \otimes K$ and $M' \otimes K$ are polystable of the same dimension vector and the above dimensions give the multiplicities of the stable summand $E$, we conclude they must be isomorphic, which gives a contradiction. Thus, $p$ is closed.

  Conversely, suppose that $S$ is Jacobson and $p \in |\modulistack[\semistable{\theta}]{d}|$ is closed.
  Since $\pi$ is of finite type, the image point $\pi(p) \in S$ is closed by \cite[Tag \href{https://stacks.math.columbia.edu/tag/01TB}{01TB}]{stacks-project}.
  On the other hand since $p$ is in particular closed in the fiber $\pi^{-1}(\pi(p)) = \modulistack[\semistable{\theta}]{d,\kappa(\pi(p))}$, by the Nullstellensatz $p$ is represented by a map $\Spec{L} \to \modulistack[\semistable{\theta}]{d}$ corresponding to a representation $M$ over $L$, where $L$ is a finite extension of $\kappa(\pi(p))$. If $M$ is not geometrically polystable, then there is a finite extension $L'$ of $L$ such that the point $|\gr(M \otimes L')|$ is distinct from $p$ but lies in the closure of $p$ by \cref{lemma:JH-filtration-closure}. Thus, $M$ must be geometrically polystable.
\end{proof}

If $S$ is not Jacobson, the image of a closed point $p \in |\modulistack[\semistable{\theta}]{d,S}|$ in $S$ may not be closed. As an example, given a discrete valuation ring $R$ with uniformizer $\pi$ and fraction field $K$, the representation $K \to K$ of the Jordan quiver given by multiplication by $\pi^{-1}$ corresponds to a closed point in $\modulistack{1,\Spec{R}}$ whose image in $\Spec{R}$ is not closed.

\begin{corollary}\label{corollary:points-are-semisimple-reps}
  Over an algebraically closed field $k$, the closed points of $\modulistack{d}$ are in bijection with isomorphism classes of semisimple $k$-representations.
\end{corollary}

\begin{proof}
  We have $\modulistack{d} = \modulistack[\semistable{\theta}]{d}$ for the stability function $\theta = 0$,
  with respect to which a representation is polystable if and only if it is semisimple.
\end{proof}

\begin{remark}
  As an application of the classification of closed points of $\modulistack[\semistable{\theta}]{d}$,
  one can show that the line bundle $\cL_\theta$ on $\modulistack[\semistable{\theta}]{d}$
  descends to the moduli space $\modulispace[\semistable{\theta}]{d}$
  (whose existence we will obtain in \cref{corollary:Md-ss-separated-gms,remark:stability-function}),
  provided we work over a noetherian base of characteristic~0,
  so that we can apply \cite[Theorem 10.3]{MR3237451}.
  Indeed it suffices to show the stabilizer of any closed geometric point $x\colon \Spec k \to \modulistack[\semistable{\theta}]{d}$
  acts trivially on $x^* \cL_\theta$ and since such closed points correspond to $\theta$-polystable representations,
  whose automorphism groups are products of general linear groups,
  one can directly check the action on $x^* \cL_\theta$ is trivial.
  We do not give the details, as we will prove in greater generality that $\cL_\theta$ descends in \cref{section:projectivity}.
\end{remark}

\subsection{Local reductivity}
\label{subsection:local-reductivity}

We first give a criterion to ensure that points on a stack specialize to closed points.

\begin{lemma}
\label{lemma:specializing to closed points}
If $\cX$ is a quasi-compact algebraic stack with quasi-compact diagonal, then every point in the topological space $|\cX|$ specializes to a closed point.
\end{lemma}

\begin{proof}
  By \cite[Tag \href{https://stacks.math.columbia.edu/tag/0DQN}{0DQN}]{stacks-project}, $|\cX|$ is a spectral topological space,
  and by \cite[Theorem 6]{Hochster}, a spectral topological space underlies some affine scheme, meaning that there is a homeomorphism $|\cX| \cong | \Spec A |$ for some commutative ring $A$.
  But on an affine scheme, every point specializes to a closed point.
\end{proof}

We remark that \cref{lemma:specializing to closed points} is not true for arbitrary quasi-compact stacks, not even algebraic spaces. For example, if $k$ is a field of characteristic $0$ and $X$ is the quotient of $\bbA^1_k$ by the free action of $\ZZ$ given by the dual action $n\cdot x \mapsto x + n$ on $k[x]$, then 
$X$ has infinitely many points but the trivial topology.

Now we will start applying \cref{lemma:specializing to closed points} by proving that the moduli stack of all representations is locally reductive.

\begin{lemma}
\label{lemma:full-stack-loc-red}
  The stack $\modulistack{d,S}$ is locally reductive.
\end{lemma}

\begin{proof}
  Since the stack $\modulistack{d,S}$ is of finite type over the noetherian base scheme $S$, it is quasi-compact, and by \cref{proposition:affine-diagonal} the diagonal of $\modulistack{d,S}$ is affine, hence quasi-compact,
  so it follows from \cref{lemma:specializing to closed points} that every point of $|\cX|$ specializes to a closed point.

  Thus, it remains to find \'{e}tale local quotient presentations.
  By covering $S$ by affine open subschemes, we may assume that $S$ itself is affine. Since $\modulistack{d,S}$ is now a global quotient stack of the affine scheme $\Rep{d,S}$ by the reductive group scheme $\textrm{G}_{d,S}$, by choosing an embedding $\textrm{G}_{d,S} \hookrightarrow \GL_{N,S}$, we have
  \[
    \modulistack{d,S} \cong [\Rep{d,S}/\textrm{G}_{d,S}] \cong [X/\GL_{N,S}],
  \]
  where $X$ is the quotient of the affine scheme $\Rep{d,S} \times_S \GL_{N,S}$ by the free diagonal action of $\textrm{G}_{d,S}$. Since $X$ is an affine scheme, this presentation shows that $\modulistack{d,S}$ is locally reductive.
\end{proof}

We next show that $\modulistack[\semistable{\theta}]{d,\ZZ}$ is locally reductive using the determinantal sections constructed in \cref{subsection:characterizing-semistable-reps}.
We comment on the (non-)necessity of the condition on the stability function~$\theta$ in \cref{remark:stability-function}. First we make an observation that will 
be used in the proof of 
\cref{proposition:locally-reductive} and 
\cref{theorem:proj-over-noeth-S}.

\begin{remark}\label{remark:spread-out-etale-locally}
  \begin{enumerate}[label=(\roman*)]
      \item Let $k$ be a field, $B \twoheadrightarrow k$ a surjection of rings, and let $V$ be a representation of dimension vector $d$ over $k$. We can extend $V$ to a representation $\cV$ over $\Spec{B}$ as follows. Since  $V_i$ is free for each vertex $i$, we can take $\cV_i = B^{\oplus d_i}$, and since each map $V_a: V_{s(a)} \to V_{t(a)}$ is represented by a matrix with entries in $k$, we can lift these matrices to $B$ to obtain $\cV_a$.
      
      \item If $x \in \Spec{A}$ is a point in an affine scheme and $k$ is a finite separable extension of the residue field $\kappa(x)$, we can find an \'etale morphism $\Spec{B} \to \Spec{A}$ whose fiber over $x$ is $\Spec{k}$. For instance, we can take $A = A_a[T]/(f)$, where $f \in A[T]$ is a monic polynomial whose image in $\kappa(x)[T]$ is the minimal polynomial of a primitive element of $k$ over $\kappa(x)$ and $a \in A$ is chosen suitably.
  \end{enumerate}
\end{remark}

\begin{proposition}
  \label{proposition:locally-reductive}
  If $\theta = \theta_\beta$ or $\theta = \eta_\beta$ for a dimension vector $\beta$, then the stack
  $\modulistack[\semistable{\theta}]{d,S}$ is locally reductive.
\end{proposition}

\begin{proof}
We consider the case $\theta = \eta_\beta$. 
By \cref{corollary:open-substacks}, $\modulistack[\semistable{\theta}]{d,S}$ is quasi-compact with quasi-compact diagonal, so by \cref{lemma:specializing to closed points} every point specializes to a closed point.
We are again left with finding \'etale local quotient presentations as in \cref{definition:locallyred}.
The question is Zariski local on $S$, so we may assume $S = \Spec{C}$ is affine.
We begin by reducing to the case $S = \Spec{\ZZ}$. Namely, suppose $[\Spec{A}/\GL_N] \to \modulistack[\semistable{\theta}]{d,\ZZ}$ is \'etale and consider the diagram
\begin{equation*}
\begin{tikzcd}
    {[}\Spec{A}/\GL_N] \times_{\Spec{\ZZ}} \Spec{C} \arrow[r] \arrow[d] & {[}\Spec{A}/\GL_N] \arrow[d] \\
    \modulistack[\semistable{\theta}]{d,C} \arrow[r] \arrow[d] & \modulistack[\semistable{\theta}]{d,\ZZ} \arrow[d] \\
    \Spec C \arrow[r] & \Spec \ZZ
\end{tikzcd}
\end{equation*}
The top-left vertical arrow is \'etale, and moreover
\[ [\Spec{A}/\GL_N] \times_{\Spec{\ZZ}} \Spec{C}  \cong [\Spec(A \otimes_\ZZ C)/ \GL_N] \]
is of the required form, so base changing an \'etale cover of $\modulistack[\semistable{\theta}]{d,\ZZ}$ by quotient presentations yields the desired cover of $\modulistack[\semistable{\theta}]{d,C}$. Thus, we may assume $S = \Spec{\ZZ}$.

Since $\Spec\ZZ$ is Jacobson, by \cref{proposition:closed-points-arbitrary-base} any closed point of $\modulistack[\semistable{\theta}]{d,\ZZ}$ is represented by
a map $x\colon \Spec k \to \modulistack[\semistable{\theta}]{d,\ZZ}$ corresponding to a $\theta$-semistable 
representation $M$ defined over a finite field $k$.
Using \cref{proposition:sst}, we can find a representation $\overline{V}$ over $\bar{k}$
with $\dimvect{\overline{V}} = m \beta$ for some $m > 0$ such that
\[ \Hom(M \otimes \bar{k}, \overline{V}) = \Ext(M \otimes \bar{k}, \overline{V}) = 0. \]
The representation $\overline{V}$ is defined over some finite extension $k'$ of $k$, meaning that $\overline{V} \cong V' \otimes_{k'} \bar{k}$ for some representation $V'$ over $k'$. Note that
\[ \Hom(M \otimes k', V') \otimes \bar{k} = \Hom(M \otimes \bar{k}, \overline{V}) = 0, \]
hence $\Hom(M \otimes k', V') = 0$. We now replace $k$ by $k'$ and $M$ by $M \otimes_k k'$, as the two define the same point of $\modulistack[\semistable{\theta}]{d,\ZZ}$.
Since $k$ is separable over its prime field $\FF_p$,
using \cref{remark:spread-out-etale-locally} we can find an \'etale map $\Spec{B} \to \Spec{\ZZ}$ whose fiber over $\Spec{\FF_p}$ is $\Spec{k}$ and an extension $\cV$ of $V$ to $\Spec{B}$.

Now consider the cartesian diagram
\begin{equation*}
  \begin{tikzcd}
    \modulistack[\semistable{\theta}]{d,B} \arrow[r, "f"] \arrow[d] & \modulistack[\semistable{\theta}]{d,\ZZ} \arrow[d] \\
    \Spec B \arrow[r] & \Spec \ZZ
  \end{tikzcd}
\end{equation*}
The bottom morphism is \'etale, hence so is the top morphism $f$.
Moreover, $f$ induces an isomorphism on automorphism groups
because it arises as the base change of a morphism of schemes \cite[Tag \href{https://stacks.math.columbia.edu/tag/0DUB}{0DUB}]{stacks-project}.
Observe that the map $x$ can be factored through $f$ since we chose $B$ that admits a quotient isomorphic to $k$.
Now the representation $\cV$ defines a section $\sigma$ of the line bundle $\cL_{\theta}^{\otimes m}$ on the larger stack $\modulistack{d,B}$ and by \cref{proposition:vanishing-sections} this section is nonzero at $x$. Let $\cU \subset \modulistack{d,B}$ be the nonvanishing locus of $\sigma$. Using \cref{proposition:sst}(b) we see that in fact $\cU \subseteq \modulistack[\semistable{\theta}]{d,B}$.

We claim that $\cU$ is of the desired form $[\Spec A/\GL_N]$. To see this, recall that $\modulistack{d,B} \cong [\Rep{d,B}/\mathrm{G}_{d,B}]$ and let $\varphi \colon \Rep{d,B} \to \modulistack{d,B}$ denote the quotient map. The preimage
$U = \varphi^{-1}(\cU) \subset \Rep{d,B}$  is $\mathrm{G}_{d,B}$-invariant, and moreover
$U$ is the nonvanishing locus of the section $\varphi^*\sigma$, hence affine since $\varphi^*\cL_{\theta} \cong \cO_{\Rep{d,B}}$. Thus, we see that
$\cU \cong [
U/\mathrm{G}_{d,B}]$. Finally, since we can embed $\mathrm{G}_{d,B}$ into $\GL_{N,B}$ as a closed subgroup for a suitable $N$, we can write
$[U/\mathrm{G}_{d,B}] \cong [\Spec A/\GL_{N,B}]$ as in \cref{lemma:full-stack-loc-red}. In conclusion, we have exhibited an \'etale neighborhood $[\Spec A/\GL_N] \to \modulistack[\semistable{\theta}]{d,\ZZ}$ of the point $x$.
\end{proof}

\subsection{\texorpdfstring{$\Theta$}{Theta}-reductivity and S-completeness for quiver representations}
\label{subsection:ThetaS}
In this section we give a direct moduli-theoretic proofs that
the stacks $\modulistack{d,S}$ and $\modulistack[\semistable{\theta}]{d,S}$
are $\Theta$-reductive and S-complete.
First, we show that the stack $ \modulistack{d,S}$ satisfies Hartogs's principle.
For this, we use the following version of Hartogs's lemma.

\begin{lemma}
\label{lemma:Hartogs}
Let $A$ be a regular local ring of dimension $2$ with closed point $0 \in \Spec{A}$.
\begin{enumerate}[label=(\roman*)]
    \item \label{item:hartogs1} Any locally free sheaf $\cF$ of finite rank on $\Spec{A} \setminus \{ 0 \}$ is free.
    \item \label{item:hartogs2} If $\cF$ is a free sheaf of finite rank on $\Spec{A}$, then any automorphism of $\cF$ over $\Spec{A} \setminus \{0\}$ extends uniquely to $\Spec{A}$.
\end{enumerate}
\end{lemma}

\begin{proof}
  For \ref{item:hartogs1}, note that since $\Spec{A}$ is noetherian, the coherent sheaf $\cF$ on $\Spec{A} \setminus \{0\}$   admits a coherent extension $\cG$ to all of $\Spec{A}$.
  The double dual $\cG^{\vee\vee}$ is reflexive, hence free since $A$ is regular local of dimension $2$. Since $\cG^{\vee\vee}$ agrees with $\cG$ wherever $\cG$ is locally free, we see that $\cG^{\vee\vee}$ is also an extension of $\cF$, and so $\cF$ is free itself.

  For \ref{item:hartogs2}, let $n = \rk \cF$. After choosing an isomorphism $\cF \cong \cO^{\oplus n}$, an automorphism $\phi$ of $\cF|_{\Spec{A} \setminus \{0\}}$ corresponds to an invertible $n \times n$ matrix consisting of functions on $\Spec{A} \setminus \{0\}$. By the usual Hartogs's lemma, these functions extend uniquely to functions on all of $\Spec{A}$. The determinant of the resulting matrix is nonzero at every codimension 1 point of $\Spec{A}$, hence is a unit of $A$, so the corresponding endomorphism of $\cF$ is invertible.
\end{proof}

\begin{proposition}\label{proposition:full-stack-Hartogs}
  For any regular local ring $A$ of dimension $2$ with closed point $0 \in \Spec{A}$, any morphism $\Spec{A} \setminus \{ 0 \} \rightarrow \modulistack{d, S}$ extends uniquely to $\Spec{A} \rightarrow \modulistack{d, S}$. In particular, $ \modulistack{d, S}$ is both $\Theta$-reductive and S-complete by \cref{remark:hartogs}.
\end{proposition}

\begin{proof} The morphism $\Spec{A} \setminus \{ 0 \} \rightarrow \modulistack{d}$ corresponds to a family \[\cF = ( (\cF_i)_{i \in Q_0}, (\cF_a \colon \cF_{s(a)} \rightarrow \cF_{t(a)})_{a \in Q_1})\] over $\Spec{A} \setminus \{ 0 \}$. By
\cref{lemma:Hartogs}, each vector bundle $\cF_i$ on $\Spec{A} \setminus \{ 0 \}$ extends uniquely to $\widetilde{\cF}_i$ on $\Spec{A}$ and each map $\cF_a$ extends uniquely to $\widetilde{\cF}_a \colon \widetilde{\cF}_{s(a)} \rightarrow \widetilde{\cF}_{t(a)}$ on $\Spec{A}$. This family $\widetilde{\cF}$ thus defines the unique morphism $\Spec{A} \rightarrow \modulistack{d}$ extending the given one.
\end{proof}

We are now in a position to apply \cref{theorem:AHLH} to $\modulistack{d,S}$,
which we know is locally reductive by \cref{lemma:full-stack-loc-red},
and has affine diagonal by \cref{proposition:affine-diagonal}.

\begin{corollary}
  \label{corollary:Md-separated-gms}
  The stack $\modulistack{d, S}$ admits a separated adequate moduli space $\modulispace{d, S}$.
\end{corollary}

We next turn our attention to the stack $\modulistack[\semistable{\theta}]{d,S}$ of $\theta$-semistable representations. As the following example shows, the stack $\modulistack[\semistable{\theta}]{d,S}$ does not in general satisfy Hartogs's principle.

\begin{example}
  Consider the $2$-Kronecker quiver
  \begin{equation*}
    Q\colon
    \begin{tikzcd}[every label/.append style={font=\small}]
      1 \arrow[r,bend left] \arrow[r,bend right,swap] & 2
    \end{tikzcd}
  \end{equation*}
  and the stability function $\theta(n_1, n_2) = n_1 - n_2$. Let $k$ be a field and $\cF$ be the family of representations of $Q$ of dimension vector $d = (1,1)$
  over $\bbA^2 = \Spec{k[x,y]}$, where the arrows are multiplication by $x$ and multiplication by $y$. The representation $\cF_t$ is stable when $t \in \bbA^2 \setminus \{0\}$ but unstable when $t = 0$, as it is destabilized by the subrepresentation $k \rightrightarrows 0$. In other words, the family $\cF$ gives a map $\bbA^2\setminus\{0\} \to \modulistack[\stable{\theta}]{d}$ whose unique extension to a map $\bbA^2 \to \modulistack{d}$ does not factor through $\modulistack[\semistable{\theta}]{d}$.
\end{example}

In order to establish $\Theta$-reductivity and S-completeness, we need a moduli-theoretic interpretation of families of representations over $\Theta_R$ and $\STbar_R$. Throughout we will use $R$ to denote a DVR with fraction field $K$ and residue field $\kappa = R/\pi$, where $\pi$ is a uniformizer of $R$.

\begin{proposition}\label{proposition:theta-reductive}
  The stack $ \modulistack[\semistable{\theta}]{d,S}$ is $\Theta$-reductive.
\end{proposition}
\begin{proof}
  Since $\Theta_R \setminus \{ 0 \}$ is the union of $\Spec(R)$ and $\Theta_K$ over $\Spec(K)$,  a morphism $\Theta_R \setminus \{ 0 \} \rightarrow \modulistack[\semistable{\theta}]{d}$   is given by a family $\cF$ of $\theta$-semistable representations over $\Spec(R)$ with a filtration
  \[
    0 = \cF_K^{0} \subset \cF_K^{1} \subset \cdots \subset \cF_K^{r} = \cF_K
  \]
  of the generic fiber $\cF_K$ such that the successive quotients $\cF_K^{\ell}/\cF_K^{\ell-1}$ are $\theta$-semistable.

  Since $\modulistack{d}$ is $\Theta$-reductive by \cref{proposition:full-stack-Hartogs}, the filtration $\cF_K^{\bullet}$ extends uniquely to a filtration $\cF^{\bullet}$ of $\cF$. Thus it suffices to verify that the associated graded of this filtration over the special fiber is $\theta$-semistable, or equivalently that $\cF_\kappa^{\ell}/\cF_\kappa^{\ell-1}$ are all $\theta$-semistable. Since $\theta$ is constant in flat families, we have $\theta(\cF_\kappa^{\ell}) =\theta(\cF_K^{\ell}) = 0$, and as $\cF_\kappa^{\ell}$ is a subrepresentation of the $\theta$-semistable representation $\cF_\kappa$ of the same slope, it is also $\theta$-semistable and consequently all the $\cF_\kappa^{\ell}/\cF_\kappa^{\ell-1}$ are $\theta$-semistable.
\end{proof}

\begin{proposition}\label{proposition:S-complete}
  The stack $\modulistack[\semistable{\theta}]{d,S}$ is S-complete.
\end{proposition}

\begin{proof}
  Let $\STbar_R \setminus \{0\} \to \modulistack[\semistable{\theta}]{d}$ be a morphism. By \cref{proposition:full-stack-Hartogs} the morphism extends uniquely to $\STbar_R \to \modulistack{d}$, so we must show that the image of $0 \in \STbar_R$ is contained in $\modulistack[\semistable{\theta}]{d}$.

  The morphism $\STbar_R \to \modulistack{d}$ is equivalent to a diagram of representations
  \[\begin{tikzcd}
   \cdots \arrow[r,yshift=0.5ex,"s"]{s}  & \cF^{\ell-1}\arrow[r,yshift=0.5ex,"s"] \arrow[l,yshift=-0.5ex,"t"] &\cF^{\ell}\arrow[r,yshift=0.5ex,"s"] \arrow[l,yshift=-0.5ex,"t"] &\cF^{\ell+1}\arrow[r,yshift=0.5ex,"s"] \arrow[l,yshift=-0.5ex,"t"] & \cdots \arrow[l,yshift=-0.5ex,"t"]
 \end{tikzcd}\]
  of $Q$ over $\Spec R$, such that
  \begin{itemize}
    \item each map $s$ and $t$ is injective,
    \item $s\circ t = t\circ s = \pi$,
    \item $s$ is an isomorphism for $\ell \gg 0$ and $t$ is an isomorphism for $\ell \ll 0$,
    \item the induced maps $s\colon \cF^{\ell-1}/t \cF^{\ell} \to \cF^{\ell}/t \cF^{\ell+1}$
      and $t\colon\cF^{\ell+1}/s \cF^{\ell} \to \cF^{\ell}/s \cF^{\ell-1}$ are injective.
  \end{itemize}
  The restriction to $\STbar_R \setminus \{ 0 \} \cong \Spec R \coprod_{\Spec K} \Spec R$
  corresponds to the two $\theta$-semistable representations
  \[
    \cE \coloneqq \mathrm{colim} (\cF^{\ell-1} \xleftarrow{t} \cF^{\ell}) \quad \mathrm{and} \quad \cF \coloneqq \mathrm{colim} (\cF^{\ell} \xrightarrow{s} \cF^{\ell+1})
  \]
  such that the restrictions $\cE_K$ and $\cF_K$ to $\Spec K$ are isomorphic. Over the closed subsets $\Theta_\kappa \xhookrightarrow{s = 0} \STbar_R$ and $\Theta_\kappa \xhookrightarrow{t = 0} \STbar_R$ we obtain filtrations
  \[
    \cdots \xhookrightarrow{t} \cF^{\ell+1}/s \cF^{\ell} \xhookrightarrow{t}  \cF^{\ell}/s \cF^{\ell-1} \xhookrightarrow{t} \cdots \hookrightarrow \cE_\kappa
  \]
  and
  \[
    \cdots \xhookrightarrow{s} \cF^{\ell-1}/t \cF^{\ell} \xhookrightarrow{s} \cF^{\ell}/t \cF^{\ell+1} \xhookrightarrow{s} \cdots \hookrightarrow \cF_\kappa
  \]
  respectively. The image of $0 \in \STbar_R$ corresponds to the common associated graded
  \[
    \bigoplus_{\ell \in \ZZ} \frac{\cF^{\ell}/t \cF^{\ell+1}}{s (\cF^{\ell-1}/t \cF^{\ell})} \cong \bigoplus_{\ell \in \ZZ} \frac{\cF^{\ell}/s \cF^{\ell-1}}{t (\cF^{\ell+1}/s \cF^{\ell})} \cong \bigoplus_{\ell \in \ZZ} \frac{\cF^{\ell}}{s \cF^{\ell-1} + t \cF^{\ell+1}},
  \]
  which we must show is $\theta$-semistable.

  By assumption we have for all $\ell$
  \[
    0 = \theta(\cE_\kappa) \ge \theta(\cF^{\ell}/s \cF^{\ell-1}) = \theta(\cF^{\ell}) - \theta(\cF^{\ell-1}),
  \]
  since $\cF^{\ell-1} \cong s \cF^{\ell-1} \subseteq \cF^{\ell}$. Similarly,
  \[
    0 = \theta(\cF_\kappa) \ge \theta(\cF^{\ell}/t \cF^{\ell+1}) = \theta(\cF^{\ell}) - \theta(\cF^{\ell+1}).
  \]
  Thus, we must have $\theta(\cF^{\ell}) = \theta(\cF^{\ell+1})$ for all $\ell$. Moreover, as the value of $\theta$ is constant in flat families and $\cF^{\ell}_K \cong \cF_K$, we must have $\theta(\cF^{\ell}) = 0$ for all $\ell$. Thus, the quotients $\cF^{\ell}/t \cF^{\ell+1}$ are all semistable with $\theta(\cF^{\ell}/t \cF^{\ell+1}) = 0$, and the same is true for the quotients $(\cF^{\ell}/t \cF^{\ell+1})/s (\cF^{\ell-1}/t \cF^{\ell})$.
\end{proof}

Similarly to \cref{corollary:Md-separated-gms} we obtain the following.
\begin{corollary}
  \label{corollary:Md-ss-separated-gms}
  If $\theta = \theta_\beta$ or $\theta = \eta_\beta$ for a dimension vector $\beta$,
  then the stack $\modulistack[\semistable{\theta}]{d, S}$ admits an adequate moduli space $\modulispace[\semistable{\theta}]{d, S}$,
  separated over $S$.
\end{corollary}

\begin{proof}
  We have verified that
  $\modulistack[\semistable{\theta}]{d,S}$ is locally reductive in \cref{proposition:locally-reductive},
  $\Theta$-reductive in \cref{proposition:theta-reductive},
  and S-complete in \cref{proposition:S-complete}.
  It moreover has affine diagonal,
  and thus affine stabilizers and separated diagonal.
  Therefore, it admits an adequate moduli space $\modulispace[\semistable{\theta}]{d,S}$, separated over $S$, by \cite[Theorem 5.4]{1812.01128} (see \cref{theorem:AHLH}).
\end{proof}

Recall that, if $Q$ is acyclic, then any stability function $\theta$ for which there exists a semistable representation supported on $Q_0$ can be written in this form by \cref{lemma:stability-conditions-are-dim-vectors}.

\begin{remark}
  \label{remark:stability-function}
  It follows from the GIT construction that $\modulistack[\semistable{\theta}]{d,S}$ is locally reductive for an arbitrary stability function $\theta$. We are currently unable to remove the hypothesis $\theta = \theta_\beta$ or $
  \theta = \eta_\beta$ using our methods.

  When the base scheme~$S$ has characteristic $0$, S-completeness implies local reductivity \cite[Proposition~3.47, Theorem 2.2]{1812.01128}, and it follows that the stack $\modulistack[\semistable{\theta}]{d, S}$ admits a separated \textit{good} moduli space for any choice of $\theta$.
\end{remark}

\cref{corollary:Md-ss-separated-gms} also follow from GIT,
as explained in \cite[Theorem~1.5]{MR4226557}.
We have given a purely moduli-theoretic argument
by appealing to the existence result for adequate moduli spaces.

\subsection{Langton's semistable extension theorem for quiver representations}

In this section we will show that when $Q$ is an acyclic quiver, the
adequate moduli space $\modulispace[\semistable{\theta}]{d,S}$ is proper over~$S$,
where~$S$ is a noetherian scheme.
This is a particular instance of \cref{proposition:gms-map-proper} describing properness of maps between
adequate moduli spaces.

\begin{proposition}
  \label{proposition:gms-map-proper}
  Let $Q$ be a (not necessarily acyclic) quiver and let $\theta$ be a stability function. The morphism
  \[
    \modulispace[\semistable{\theta}]{d,S} \to \modulispace{d,S}
  \]
  on adequate moduli spaces is proper whenever the adequate moduli space $\modulispace[\semistable{\theta}]{d,S}$ exists (see \cref{corollary:Md-ss-separated-gms}).
\end{proposition}

The proof of \cref{proposition:gms-map-proper} relies on the following result,
which is an analogue of the main result of \cite{MR0364255}.

\begin{proposition}\label{proposition:langton}
  Let $R$ be a DVR with uniformizer $\pi$, fraction field $K$ and residue field $\kappa$. Let $M$ be a representation over $R$ such that the generic fiber $M \otimes_R K$ is $\theta$-semistable. There exists a subrepresentation $M' \subseteq M$ such that $M' \otimes_R K$ and $M \otimes_R K$ are isomorphic, and $M' \otimes_R \kappa$ is $\theta$-semistable.
\end{proposition}

\proof If $\overline{M}\colonequals M\otimes_R \kappa$ is $\theta$-semistable, then there is nothing to prove. If this is not the case, let $\overline{F}$ be the maximal destabilizing subrepresentation of $\overline{M}$. This defines a subrepresentation $M^{(1)} \subset M$ of $Q$ in the following way. For every $i \in Q_0$, let $(f^1_i,\dots, f^{s_i}_i, e_i^{s_i+1},\dots, e^{d_i}_i)$  be a basis of $\overline{M}_i$ extending a basis $f^1_i,\dots, f^{s_i}_i$ of $\overline{F}$. We can further lift these to bases of each $M_i$, which we denote by $(\tilde{f}^1_i,\dots, \tilde{f}^{s_i}_i, \tilde{e}_i^{s_i+1},\dots, \tilde{e}^{d_i}_i)$. For every $i\in Q_0$, we define $M^{(1)}_i$ as the subset of $M_i$ spanned by $(\tilde{f}^1_i,\dots, \tilde{f}^{s_i}_i, \pi \tilde{e}_i^{s_i+1},\dots, \pi \tilde{e}^{d_i}_i)$. For every $a \colon i \to j$, the restriction of $M_a$ to $M^{(1)}_i$ lands in $M^{(1)}_j$, thus this defines a representation $M^{(1)}$ of $Q$. If $\overline{M^{(1)}}$ is $\theta$-semistable, then we are done. Otherwise, let $\overline{F}_{1}$ be the maximal destabilizing subrepresentation of $\overline{M^{(1)}}$ which, following the above procedure, defines a subrepresentation $M^{(2)} \subset M^{(1)}$. We can apply the arguments of \cite[Section~5]{MR0364255} to show that this procedure will terminate, i.e., there is an $n$ such that $\overline{M^{(n)}}$ is $\theta$-semistable.
\endproof

\begin{proof}[Proof of \cref{proposition:gms-map-proper}]
  The map  $\modulispace[\semistable{\theta}]{d, S} \to \modulispace{d, S}$ is separated and of finite type over $S$ since both  $\modulispace[\semistable{\theta}]{d, S}$ and $\modulispace{d, S}$ are, so it suffices to show that it is universally closed.
  Moreover, since the adequate moduli space map $\modulistack[\semistable{\theta}]{d, S} \to \modulispace[\semistable{\theta}]{d, S}$ is surjective,
  it is enough to show that the map $\modulistack[\semistable{\theta}]{d, S} \to \modulispace{d, S}$ is universally closed. This is local on the base scheme $S$, so since the formation of the adequate moduli space commutes with base change along open embeddings, we may assume that $S = \Spec{B}$ for a noetherian ring $B$.

  To show that $\modulistack[\semistable{\theta}]{d, S} \to \modulispace{d, S}$ is universally closed, we verify the valuative criterion for universal closedness \cite[Tag \href{https://stacks.math.columbia.edu/tag/0H2C}{0H2C}]{stacks-project}: if for any DVR $R$ with fraction field $K$ and the square of solid arrows in the diagram
  \begin{equation*}
  \begin{tikzcd}
    \Spec K' \arrow [r, dashed] \arrow[d, dashed] & \Spec K \arrow[r] \arrow[d] & \modulistack[\semistable{\theta}]{d, S} \arrow[d] \\
    \Spec R' \arrow[r, dashed] \arrow[rru, dashed] & \Spec R \arrow[r] & \modulispace{d, S}
  \end{tikzcd}
  \end{equation*}
  commutes, there exists a field extension $K'$ of $K$, a DVR $R' \subset K'$ dominating $R$, and a dashed diagonal arrow $\Spec{R'} \to \modulistack[\semistable{\theta}]{d, S}$ making the whole diagram commute, then the rightmost morphism is universally closed.

  Since $S = \Spec{B}$ is affine, we have by \cref{proposition:smooth-cover-by-Rep} that $\modulistack{d, S} = [\Spec{A}/\GL_N]$, where $A$ is a polynomial ring over $B$. Moreover, since $B$ is noetherian, the map $\pi\colon \modulistack{d, S} \to \modulispace{d, S}$ is of finite type by \cref{theorem:gms-properties}~\ref{enumerate:ams-finite-type}.
  Thus, we may apply \cite[Theorem A.8]{1812.01128} to find a finite extension $K'$ of $K$, a DVR $R' \supseteq R$ dominating $R$, and a morphism $\psi\colon \Spec{R'} \to \modulistack{d, S}$ such that the diagram of solid arrows
  \begin{equation*}
  \begin{tikzcd}
    \Spec K' \arrow [r] \arrow[dd] & \Spec K \arrow[r] \arrow[dd] & \modulistack[\semistable{\theta}]{d, S} \arrow[d, "\iota"] \\
    & & \modulistack{d, S} \arrow[d, "\pi"] \\
    \Spec R' \arrow[rru, "\psi", near end] \arrow[r] \arrow[rruu, dashed, near end, "\psi'"] & \Spec{R} \arrow[r] & \modulispace{d, S}
  \end{tikzcd}
  \end{equation*}
  commutes.

  The map $\psi\colon \Spec R' \to \modulistack{d, S}$ corresponds to a family of representations over $R'$ with $\theta$-semistable generic fiber. By \cref{proposition:langton}, there exists another morphism $\psi'\colon \Spec R' \to \modulistack[\semistable{\theta}]{d, S}$ such that the restrictions of $\psi$ and $\psi'$ to $\Spec K'$ agree.
  Since
  $R'$ is a DVR
  and $\modulispace{d, S}$ is separated, this implies that the two morphisms $\pi \circ \psi$ and $\pi \circ \iota \circ \psi'$ are equal \cite[Tag \href{https://stacks.math.columbia.edu/tag/03KU}{03KU}]{stacks-project}. Thus, the top and bottom rows together with the arrow $\psi'$ in the above diagram are what we set out to construct.
\end{proof}

If $Q$ is not acyclic, then $\modulispace{d,k}$ is not proper over $\Spec(k)$.
Consider for instance the Jordan quiver:
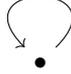
\begin{equation*}
  \begin{tikzcd}[every label/.append style={font=\small}]
    \bullet \arrow[out=45, in=135, loop]
  \end{tikzcd}
\end{equation*}
In this case all representations are $\theta$-semistable,
since the only stability function is the zero function.
For an explicit example (using the notation of \cref{proposition:langton})
that illustrates how the valuative criterion for properness fails it suffices to consider
the one-dimensional representation $M_a \colon K \to K, 1 \mapsto \pi^{-1}$ over $K$.
Then $M$ has no lift to any representation over~$R$. In fact, the space of $d$-dimensional representations of $Q$ where~$d=(1)$
is represented by the affine line,
because the group $\mathrm{G}_{d,k}\cong\Gm$ acts trivially on $\Rep{d,k}=\mathbb{A}^1$.

\begin{proposition}\label{proposition:acyclic-stack-universally-closed}
  If $Q$ is acyclic, the stack $\modulistack[\semistable{\theta}]{d,S}$ is universally closed over $S$.
\end{proposition}
\begin{proof}
  We check the valuative criterion \cite[Tag \href{https://stacks.math.columbia.edu/tag/0H2C}{0H2C}]{stacks-project},
  which translates to: for a discrete valuation ring $R$ with uniformizer $\pi$ and fraction field $K$, if $M$ is a semistable representation
  over $K$, there exists a semistable representation $N$ over $R$ and an isomorphism $\phi\colon M \xrightarrow{\sim} N|_K$. By \cref{proposition:langton}, it suffices to find such a family $N$ without requiring that $N \otimes_R (R/\pi)$ is semistable.

  Choose an admissible ordering of 
  $Q_0 = \{ 1, \dots, n \}$ and a $K$-basis of $M_i$ for each $i \in Q_0$.
  In these bases, the maps $M_a\colon M_{s(a)} \to M_{t(a)}$ are given by matrices $A_a$ over $K$.
  For each $i$, let $N_i$ be a free $R$-module with the same basis. 
  We define integers $m_i$ for $i \in Q_0$ by setting $m_1 = 0$ and
  inductively choosing $m_i$ in such a way that for each arrow $a$ with $t(a) = i$,
  the matrix $N_a = \pi^{m_i - m_{s(a)}} A_a$ has entries in $R$.
  Now we can set $N = (\{N_i\}_{i \in Q_0}, \{N_a\}_{a \in Q_1})$,
  and $\phi\colon M \xrightarrow{\sim} N|_K$ is given
  by taking $\phi_i$ to be multiplication by $\pi^{m_i}$.
\end{proof}

\begin{corollary}
  \label{corollary:adequate-proper-arbitrary-base}
  If~$Q$ is acyclic, the adequate moduli space~$\modulispace[\semistable{\theta}]{d,S}$
  is proper over~$S$.
\end{corollary}

\begin{proof}
  The map $\modulispace[\semistable{\theta}]{d,S} \to S$ is separated by \cref{corollary:Md-ss-separated-gms},
  so it suffices to show that it is universally closed,
  and since the map $\modulistack[\semistable{\theta}]{d,S} \to \modulispace[\semistable{\theta}]{d,S}$
  is surjective,
  this is equivalent to $\modulistack[\semistable{\theta}]{d,S} \to S$ being universally closed,
  which is \cref{proposition:acyclic-stack-universally-closed}.
\end{proof}

The following gives a complete description of the moduli space $\modulispace{d,S}$ when $Q$ is acyclic.

\begin{proposition}
  \label{proposition:acyclic-full-ams-is-S}
  If $Q$ is acyclic, the structure morphism $\modulispace{d,S} \to S$ is an isomorphism.
\end{proposition}
\begin{proof}
  Consider first the case $S = \Spec{k}$ for a field $k$.
  We claim that the stack $\modulistack{d,k}$ has a unique closed point.
  Indeed, by \cref{proposition:closed-points-arbitrary-base},
  any closed point is represented by a semisimple representation $M$ defined over a finite extension $L$ of $k$.
  Since $Q$ is acyclic, the only such semisimple representation is $\bigoplus_{i \in Q_0} S(i)^{\oplus d_i}$ which is already defined over $k$.
  Thus, by \cref{theorem:gms-properties}~\ref{enumerate:gms-points}, the adequate moduli space $\modulispace{d}$
  has a unique closed point which is defined over $k$,
  and since $\modulispace{d}$ is of finite type over $k$,
  this point must be the only one.
  Finally, as the stack $\modulistack{d}$ is reduced,
  so is $\modulispace{d}$, and thus we conclude that $\modulispace{d} \cong \Spec{k}$.

  Let now $S$ be a noetherian scheme. For any point $x \in S$, the base change map
  \[ \modulispace{d,\kappa(x)} \to \modulispace{d,S} \times_S \Spec{\kappa(x)} \]
  is bijective by \cref{theorem:gms-properties}~\ref{enumerate:ams-base-change-univ-homeo},
  and by the above $\modulispace{d,\kappa(x)}$ is isomorphic to $\Spec{\kappa(x)}$. Since $\modulispace{d,S} \to S$ is also proper by \cref{corollary:adequate-proper-arbitrary-base},
  it is finite by Zariski's Main Theorem \cite[Tag \href{https://stacks.math.columbia.edu/tag/0A4X}{0A4X}]{stacks-project}.
  Thus, we have $M_{d,S} \cong \underline{\Spec}_{\cO_S} \cA$ for a sheaf $\cA$ of finite $\cO_S$-algebras.
  Moreover, since for any $x \in S$, the induced map $\cO_S|_x \to \cA|_x$ is an isomorphism,
  the structure map $\cO_S \to \cA$ is surjective, so $\modulispace{d,S} \to S$ is a closed embedding.
  On the other hand, the composition
  $\modulistack{d,S} \to \modulispace{d,S} \to S$ is
  scheme-theoretically surjective,
  hence so is $\modulispace{d,S} \to S$,
  so it must be an isomorphism.
\end{proof}

A more general result regarding the structure of $\modulispace{d,S}$ is proved using GIT-methods by Donkin in \cite[Theorem and Remark]{MR1259609},
generalizing the result for fields in characteristic $0$ due to Le Bruyn--Procesi \cite[Theorem 1]{MR958897}.

\section{Projectivity of the adequate moduli space}
\label{section:projectivity}

The aim of this section is to prove that the moduli space $\modulispace[\semistable{\theta}]{d,S}$ is projective over $S$ when the quiver $Q$ is acyclic.
Recall that we defined a line bundle $\cL_\theta$ on $\modulistack[\semistable{\theta}]{d}$ in Section~\ref{subsection:determinantal-line-bundles}.
We begin by showing that $\cL_\theta$ is semiample even if $Q$ is not acyclic, which will imply \cref{theorem:effective} from the introduction. After this, we show with increasing generality that $\cL_\theta$ is relatively ample over $S$.

\subsection{Global generation over a field}
\label{subsection:global-generation}

Let $Q$ be a quiver, $d$ a dimension vector, and $\theta$ a stability function such that $\theta(d) = 0$. 

\begin{proposition}
  \label{proposition:semiampleness}
  Suppose $k$ is a field and the stability function $\theta$ is of the form $\theta = \theta_\beta$ or $\theta = \eta_\beta$ for a dimension vector $\beta$.
  The line bundle $\cL_\theta$ on $\modulistack[\semistable{\theta}]{d, k}$ is semiample
  and descends to a line bundle $L_\theta$ on the moduli space $\modulispace[\semistable{\theta}]{d,k}$.
  In fact, if $m \in \NN$ satisfies the inequality \eqref{eq:m-inequality-beta},
  then $\cL_\theta^{\otimes m}$ is generated by finitely many global sections, and if $k$ is algebraically closed, these sections can be taken to be of the form $\sigma_V$ for a representation $V$ of dimension vector $m\beta$.
\end{proposition}

\begin{proof}
  We give the proof when $\theta = \eta_\beta$, the other case follows similarly. Assume first that $k$ is algebraically closed. Let $p \in \modulistack[\semistable{\theta}]{d,k}$ be a closed point corresponding to a $\theta$-semistable representation $M$.
  By \cref{corollary:hom-vanishing}~\ref{enumerate:hom-vanishing-beta}, a general representation $V$ of dimension vector $m\beta$ satisfies $\Hom(M, V) = 0$, and by \cref{proposition:vanishing-sections}, the associated section $\sigma_V$ of $\cL_\theta^{\otimes m}$ is nonzero at the point $p$. Since $\modulistack[\semistable{\theta}]{d,k}$ is quasi-compact, the non-vanishing loci of finitely many such sections cover $\modulistack[\semistable{\theta}]{d,k}$.

  Let now $k$ be an arbitrary field. By the above, the pullback of $\cL_\theta^{\otimes m}$ to $\modulistack[\semistable{\theta}]{d,\bar{k}}$ is generated by finitely many global sections when $m$ satisfies \eqref{eq:m-inequality-beta}, so the same holds for $\cL_\theta^{\otimes m}$ by for example \cite[Exercise 19.2.I]{rising-sea}.

  To prove that $\cL_\theta$ descends to $\modulispace[\semistable{\theta}]{d,k}$, we let $m > 0$ be an integer satisfying $\eqref{eq:m-inequality-beta}$ so that both $\cL_\theta^{\otimes m}$ and $\cL_\theta^{\otimes m+1}$ are generated by finitely many globally sections.
  It follows from \cref{lemma:descending-vector-bundles-to-ams}~\ref{item:descending-vector-bundles-to-ams-iii}
  that $\cL_\theta^{\otimes m}$ and
  $\cL_\theta^{\otimes m+1}$ descend to line bundles
  $L_m$ and $L_{m+1}$ on $\modulispace[\semistable{\theta}]{d,k}$.
  Now $L_\theta \coloneqq L_{m+1} \otimes L_m^\vee$ pulls back to
  $\cL_\theta^{\otimes m+1} \otimes (\cL_\theta^{\otimes m})^\vee = \cL_\theta$,
  so we see that $\cL_\theta$ itself descends.
\end{proof}

\begin{proof}[Proof of \cref{theorem:effective}]
  This now follows from \cref{proposition:effective-bounds}
  and \cref{proposition:semiampleness},
  with the bound in \cref{theorem:effective} being derived just from the Euler matrix for~$Q$
  and the dimension vector~$d$,
  and not the more implicit bound in \eqref{eq:m-inequality-beta}.
\end{proof}

\begin{remark}
  \label{remark:effective-bpf}
  Effective basepoint-freeness results as in \cref{theorem:effective}
  are of interest in general, and for moduli spaces in particular.
  For moduli of vector bundles on (smooth projective) curves there has been significant progress;
  for an overview, see \cite[Section~7.2]{MR3135438}.
  The moduli space~$\mathrm{M}_C(r,\mathcal{L})$
  of semistable vector bundles of rank~$r$ and fixed determinant~$\mathcal{L}$
  on a curve~$C$ of genus~$g$
  has Picard rank~1
  and its Picard group is generated by a determinantal line bundle.
  The best known bound on the basepoint-freeness of
  the linear system associated to the~$k$th multiple of the generator
  is quadratic in the rank~$r$ (but independent of the genus~$g$).
  These bounds are similar to the one in \cref{theorem:effective},
  which are also quadratic in the entries of the dimension vector.

  Conjecturally \cite[Section~7.5]{MR3135438} the true bound for
  basepoint-freeness on moduli of vector bundles is \emph{linear} in the rank,
  and thus of the same order as the \emph{square root} of the dimension of the moduli space.

  For moduli of quiver representations one also expects room for improvement.
  Consider the following two ways in which~$\mathbb{P}^n$ can be realized as a moduli space of quiver representations.
  First, using dimension vector~$(1,1)$ for the~$(n+1)$-Kronecker quiver as in \cref{example:nKronecker}
  and stability function~$(1,-1)$
  we obtain a bound linear in~$n$,
  yet the Picard group is generated by the very ample line bundle~$\mathcal{O}(1)$,
  hence the bound should be constant.
  On the other hand,
  following~\cite[page 218]{MR1906875} we can also realize it as the moduli space for the~$2$-Kronecker quiver
  using dimension vector~$(n,n)$.
  Again using \cref{example:nKronecker} we see that~$\lambda=0$,
  and thus the effective basepoint-freeness bound says that the generator is globally generated.
  See also \cref{remark:extended-dynkin} for the case of general Dynkin and extended Dynkin quivers.

  In general, Fujita's conjecture predicts a bound linear in the dimension of the moduli space.
  We obtain a bound that is quadratic in the entries of the dimension vector, and thus of the same order as the dimension of the moduli space which also grows quadratically in the entries of the dimension vector.
  This can be compared to Koll\'ar's general effective basepoint-freeness
  result \cite[Theorem~1.1]{MR1233485},
  which is very far from the predicted bound.
\end{remark}

\subsection{Projectivity over a field}

From now on, we assume that $Q$ is acyclic, in which case \cref{lemma:stability-conditions-are-dim-vectors}~\ref{enumerate:stability-conditions-are-dim-vectors-beta}
implies that $\theta = \eta_\beta$ for a unique dimension vector $\beta \in \NN^{Q_0}$.
We now prove \cref{theorem:main-thm}~\ref{enumerate:main-ample} over a field.

\begin{theorem}\label{theorem:projectivity}
  Let $k$ be a field and assume that $Q$ is acyclic.
  The line bundle $\cL_\theta$ descends to an ample line bundle $L_\theta$ on the moduli space $\modulispace[\semistable{\theta}]{d,k}$.
  In particular, the moduli space $\modulispace[\semistable{\theta}]{d,k}$ of $\theta$-semistable representations with dimension vector $d$ is a projective variety.
\end{theorem}

\begin{proof}
  Suppose first that $k$ is algebraically closed. By \cref{proposition:semiampleness}, the line bundle $\cL_\theta$ is semiample and descends to a line bundle $L_\theta$. To show that $L_\theta$ is ample, it suffice to show that for $m$ satisfying $\eqref{eq:m-inequality-beta}$, the map $\phi: \modulispace[\semistable{\theta}]{d} \to \PP^n$ induced by the complete linear series of $L_\theta^{\otimes m}$ is finite. For convenience, we denote $\cL = \cL_\theta^{\otimes m}$ and $L = L_\theta^{\otimes m}$.
  We first claim that $\phi$ has finite fibers.
  If not, there exists a smooth, proper, connected curve $C$
  and a nonconstant map $\gamma \colon C \to \modulispace[\semistable{\theta}]{d,k}$
  such that the composition $\phi \circ \gamma \colon C \to \PP^n$ is constant.
  This means that the line bundle $\gamma^* L$ has degree $0$ on $C$,
  so any section of any power of $\gamma^* L$ is constant.
  We will show that this is impossible.

  By \cref{theorem:gms-properties}~\ref{enumerate:gms-points} and \cref{proposition:closed-points-arbitrary-base}, the $k$-points of $\modulispace[\semistable{\theta}]{d,k}$ correspond to $\theta$-polystable representations under the adequate moduli space map $f: \modulistack[\semistable{\theta}]{d,k} \to \modulispace[\semistable{\theta}]{d,k}$. Given a polystable representation,
  there are only finitely many polystable representations of the same dimension vector
  with the same isomorphism classes of stable summands.
  Thus, since the image of $C$ in $\modulispace[\semistable{\theta}]{d,k}$
  contains infinitely many $k$-points,
  it in particular contains two points $p$ and $p'$ corresponding to polystable representations $M$ and $M'$
  such that one of the stable summands of $M$ does not appear in $M'$.
  By \cref{theorem:separation}, there exists $m' > 0$
  and a representation $V$ of dimension vector $m m' \beta$ such that
  \[
    \Hom(M, V) \neq 0 \quad \text{ and } \quad  \Hom(M', V) = 0.
  \]
  The representation $V$ induces a section $\sigma_V$ of $\cL^{\otimes m'}$,
  and by \cref{proposition:vanishing-sections} we have $\sigma_V(M) = 0$ but $\sigma_V(M') \neq 0$.
  There is a section $t \in \Gamma(\PP^n, \cO_{\PP^n}(m'))$ such that $\sigma_V = f^* \phi^* t$, and the section $s = \phi^*t \in \Gamma(\modulispace[\semistable{\theta}]{d,k}, L^{\otimes m'})$ has the property that
  $s(p) = 0$ but $s(p') \neq 0$. Hence, $\gamma^*(s)$ is a nonconstant section of $\gamma^* L^{\otimes m'}$, which gives a contradiction. Thus, $\phi$ has finite fibers.

  Since $\modulispace[\semistable{\theta}]{d,k}$ is proper by \cref{corollary:adequate-proper-arbitrary-base}, the map $\phi$ is proper, hence finite by Zariski's Main Theorem \cite[Tag \href{https://stacks.math.columbia.edu/tag/0A4X}{0A4X}]{stacks-project}. Thus, $\modulispace[\semistable{\theta}]{d,k}$ is projective and $L_\theta$ is ample. This concludes the case when $k$ is algebraically closed.

  Now let $k$ be an arbitrary field and let $\bar{k}$ be an algebraic closure.
  By the case of an algebraically closed field, $\mathcal{L}_{\theta,\bar{k}}$ descends to an ample line bundle $L_{\theta,\bar{k}}$ on $\modulispace[\semistable{\theta}]{d,\bar{k}}$. Consider the diagram
  \begin{equation*}
      \begin{tikzcd}
      \modulistack[\semistable{\theta}]{d,\bar{k}} \arrow[r] \arrow[d] & \modulistack[\semistable{\theta}]{d,k} \arrow[d] \\
      \modulispace[\semistable{\theta}]{d,\bar{k}} \arrow[r] \arrow[d] & \modulispace[\semistable{\theta}]{d,k} \arrow[d] \\
      \Spec{\bar{k}} \arrow[r] & \Spec{k}
    \end{tikzcd}
  \end{equation*}
  By \cref{theorem:gms-properties}~\ref{enumerate:ams-commutes-with-flat-base-change}, the base change morphism $\modulispace[\semistable{\theta}]{d,\bar{k}} \to \modulispace[\semistable{\theta}]{d,k} \times_{\Spec{k}} \Spec{\bar{k}}$ is an isomorphism,
  so it follows from \cref{lemma:descending-vector-bundles-to-ams}~\ref{item:descending-vector-bundles-to-ams-ii} that there exists a line bundle $L_\theta$ on $\modulispace[\semistable{\theta}]{d,k}$ whose pullback to $\modulispace[\semistable{\theta}]{d,\bar{k}}$ is $\overline{L}_\theta$.

  Finally, we claim that $L_\theta$ is ample. Since $\modulispace[\semistable{\theta}]{d,k}$ is proper over $k$, by \cite[Tag \href{https://stacks.math.columbia.edu/tag/0D38}{0D38}]{stacks-project}, it suffices to show that for any coherent sheaf $F$ on $\modulispace[\semistable{\theta}]{d,k}$, there exists $n_0$ such that
  \[ \mathrm{H}^i(\modulispace[\semistable{\theta}]{d,k}, F \otimes L_\theta^{\otimes n}) = 0 \]
  for all $i > 0$ and all $n \ge n_0$. By flat base change, we have
  \[ \mathrm{H}^i(\modulispace[\semistable{\theta}]{d,k}, F \otimes L_\theta^{\otimes n}) \otimes_k \bar{k} \cong \mathrm{H}^i(\modulispace[\semistable{\theta}]{d,\bar{k}}, \overline{F} \otimes \overline{L}_\theta^{\otimes n}), \]
  where $\overline{F}$ denotes the pullback of $F$ to $\modulispace[\semistable{\theta}]{d,\bar{k}}$. By the first part of the proof, $\overline{L}_\theta$ is ample, and so such an $n_0$ exists.
\end{proof}

\subsection{Projectivity over a general base}

We are now ready to prove \cref{theorem:main-thm}~\ref{enumerate:main-ample}, namely that $\modulispace[\semistable{\theta}]{d,S}$ is projective over an arbitrary noetherian base scheme $S$.
Here we use the notion of projectivity from \cite[Tag \href{https://stacks.math.columbia.edu/tag/01W8}{01W8}]{stacks-project}
(and not the stronger notion of H-projectivity).

\begin{theorem}\label{theorem:proj-over-noeth-S} \label{theorem:projZ}
  Suppose $Q$ is acyclic and  $S$ is a noetherian scheme.
  The line bundle $\cL_\theta$ descends to an $S$-ample line bundle $L_\theta$ on the moduli space $\modulispace[\semistable{\theta}]{d,S}$.
  In particular, $\modulispace[\semistable{\theta}]{d,S}$ is projective over~$S$.
\end{theorem}

\begin{proof}
  Recall that $\modulispace[\semistable{\theta}]{d,S}$ is proper over $S$ \cref{corollary:adequate-proper-arbitrary-base}.
  We begin by reducing to the case when the base scheme is $\Spec{\ZZ}$. Consider the commuting diagram
  \begin{equation}\label{diagram:base-change-of-semistable-ams}
      \begin{tikzcd}
      \modulistack[\semistable{\theta}]{d,S} \arrow[rr, "\iota_\cM"] \arrow[d, "f_{S}"] \arrow[dr, "f_\ZZ|_S"]  && \modulistack[\semistable{\theta}]{d,\ZZ} \arrow[d, "f_\ZZ"] \\
      \modulispace[\semistable{\theta}]{d,S} \arrow[r, "g"] \arrow[dr, swap, "\pi_{S}"] & \modulispace[\semistable{\theta}]{d,\ZZ} \times S \arrow[r, "\iota_{\mathrm{M}}"] \arrow[d, "\pi_\ZZ|_S"] & \modulispace[\semistable{\theta}]{d,\ZZ} \arrow[d, "\pi_\ZZ"] \\
      & S \arrow[r, "\iota"] & \Spec{\ZZ}
    \end{tikzcd}
  \end{equation}
  Suppose we know that $\cL_{\theta,\ZZ}$ descends to a $\ZZ$-ample line bundle $L_{\theta,\ZZ}$ on $\modulispace[\semistable{\theta}]{d,\ZZ}$. This means that there exists a closed embedding $j\colon \modulispace[\semistable{\theta}]{d,\ZZ} \hookrightarrow \PP_\ZZ^{n}$ such that $j^*\cO(1) = L_\theta^{\otimes m}$ for some $m > 0$, and in particular $\modulispace[\semistable{\theta}]{d,\ZZ}$ is a scheme. By base change, we obtain a closed embedding $j_S\colon \modulispace[\semistable{\theta}]{d,\ZZ} \times S \hookrightarrow \PP_S^{n}$ such that $j_S^*\cO(1) = \iota_M^*L_{\theta,\ZZ}^{\otimes m}$.
  Moreover, since
  \[ f_S^* g^* \iota_M^* L_{\theta, \ZZ} = \iota_{\cM} f_\ZZ^* L_{\theta,\ZZ} = \iota_{\cM} \cL_{\theta, \ZZ} = \cL_{\theta, S}, \]
  we see that $\cL_{\theta,S}$ descends to the line bundle $L_{\theta, S} = g^*\iota_M^* L_{\theta,\ZZ}$ and that $L_{\theta,S}^{\otimes m} = g^*j_S^*\cO(1)$.

  Now by \cref{theorem:gms-properties}~\ref{enumerate:ams-base-change-univ-homeo}, the map $g$ has finite fibers, and it is proper since $\modulispace[\semistable{\theta}]{d,S}$ is.
  Thus, $g$ is finite by \cite[Tag \href{https://stacks.math.columbia.edu/tag/0A4X}{0A4X}]{stacks-project}. This implies firstly that $\modulispace[\semistable{\theta}]{d,S}$ is affine over the scheme $\modulispace[\semistable{\theta}]{d,\ZZ} \times S$, hence itself a scheme, and secondly by \cite[Tag \href{https://stacks.math.columbia.edu/tag/0D3A}{0B5V}]{stacks-project} that $L_{\theta,S}$ is ample.

  We now proceed to prove the theorem over $\Spec{\ZZ}$. First of all, we show that $\cL_{\theta,\ZZ}$ descends to the moduli space $\modulispace[\semistable{\theta}]{d,\ZZ}$. As in the proof \cref{proposition:semiampleness}, it suffices to show that $\cL_{\theta,\ZZ}^{\otimes m}$ descends for all sufficiently large integers $m$, and by combining \cref{lemma:descending-vector-bundles-to-ams}~\ref{item:descending-vector-bundles-to-ams-ii} and \cref{lemma:descending-vector-bundles-to-ams}~\ref{item:descending-vector-bundles-to-ams-iii},
  it is enough to show that for all such $m > 0$, there exists an \'etale cover of $\Spec \ZZ$ by
  affine schemes $\Spec A \to \Spec \ZZ$ such that
  $\cL_{\theta,A}^{\otimes m}$ is globally generated.

  By \cref{proposition:semiampleness}, for all sufficiently large integers $m > 0$ and for all primes $p$ the line bundle $\cL_{\theta, \overline{\FF}_p}^{\otimes m}$ is generated by determinantal sections $\sigma_0, \ldots, \sigma_n$ corresponding to representations $V_0,\ldots,V_n$ of dimension vector $m\beta$ over $\overline{\FF}_p$. These representations are defined over a finite extension $k$ of $\FF_p$, and using \cref{remark:spread-out-etale-locally} we can find an \'etale neighborhood $\Spec{B} \to \Spec{\ZZ}$ of $\Spec{k}$ and extensions $\cV_i$ of each $V_i$ to $B$. The representations $\cV_i$ define global sections $\widetilde{\sigma}_i$ of $\cL_{\theta, B}^{\otimes m}$ over $\modulistack[\semistable{\theta}]{d,B}$ which pull back to $\sigma_i$ in $\modulistack[\semistable{\theta}]{d,\overline{\FF}_p}$. Thus, the locus $\cU$ on $\modulistack[\semistable{\theta}]{d, B}$ over which the $\widetilde{\sigma}_i$ generate $\cL_{\theta,B}^{\otimes m}$ contains $\modulistack[\semistable{\theta}]{d, k}$.

  Since the structure morphism $\modulistack[\semistable{\theta}]{d,B}\to\Spec B$ is closed by \cref{proposition:acyclic-stack-universally-closed}, the image of the complement of $\cU$ is closed in $\Spec{B}$ and does not contain $\Spec{k}$, so replacing $\Spec{B}$ by an affine open neighborhood of $\Spec{k}$, we may assume that the sections $\widetilde{\sigma}_i$ generate $\cL_{d,B}^{\otimes m}$. Choosing such an \'etale neighborhood $\Spec{B}$ for each prime $p$ provides us with the required \'etale cover of $\Spec{\ZZ}$.
  
  Let $L_\ZZ$ denote the line bundle on $\modulispace[\semistable{\theta}]{d,\ZZ}$ whose pullback to $\modulistack[\semistable{\theta}]{d,\ZZ}$ is $\cL_\ZZ=\cL_{\theta,\ZZ}$, and similarly define $L_{\overline{\FF}_p}$ on $\modulispace[\semistable{\theta}]{d,\overline{\FF}_p}$. We know that $\modulispace[\semistable{\theta}]{d,\ZZ} \to \Spec{\ZZ}$ is proper, so by \cite[Tag \href{https://stacks.math.columbia.edu/tag/0D3A}{0D3A}]{stacks-project}
  it suffices to show that the restriction of $L_\ZZ$ to $\modulispace[\semistable{\theta}]{d,\ZZ}\times_\ZZ \FF_p$ is ample, and to do this, it suffices to show that the pullback of $L_\ZZ$ to $\modulispace[\semistable{\theta}]{d,\ZZ} \times_\ZZ \overline{\FF}_p$ is ample.

  Now consider the diagram \eqref{diagram:base-change-of-semistable-ams} with $S = \Spec{\overline{\FF}_p}$. As above, the base change morphism $g$ is finite, so it follows that if $g^* \iota_{\mathrm{M}}^* L_\ZZ$ is ample on $\modulispace[\semistable{\theta}]{d,\overline{\FF}_p}$, then $\iota_{\mathrm{M}}^* L_\ZZ$ is ample on $\modulispace[\semistable{\theta}]{d,\ZZ} \times_\ZZ \overline{\FF}_p$. Now, on $\modulistack[\semistable{\theta}]{d,\overline{\FF}_p}$ we have isomorphisms of line bundles
  \[ f_{\overline{\FF}_p}^* L_{\overline{\FF}_p} = \cL_{\overline{\FF}_p} = \iota_{\cM}^* \cL_{\ZZ} = \iota_\cM^* f_\ZZ^* L_\ZZ = f_{\overline{\FF}_p}^* g^* \iota_{\mathrm{M}}^* L_\ZZ, \]
  which by \cref{lemma:descending-vector-bundles-to-ams}~\ref{item:descending-vector-bundles-to-ams-i} implies that $L_{\overline{\FF}_p} = g^* \iota_{\mathrm{M}}^* L_\ZZ$. However, we know from \cref{theorem:projectivity} that $L_{\overline{\FF}_p}$ is ample on $\modulispace[\semistable{\theta}]{d,\overline{\FF}_p}$, so $\iota_{\mathrm{M}}^* L$ is ample on $\modulispace[\semistable{\theta}]{d,\ZZ} \times_\ZZ \overline{\FF}_p$.
\end{proof}

\appendix

\section{Projectivity using Theta-stability}
\label{section:appendix}

\numberwithin{theorem}{section}
\numberwithin{corollary}{section}
\numberwithin{definition}{section}
\numberwithin{example}{section}
\numberwithin{exercise}{section}
\numberwithin{lemma}{section}
\numberwithin{proposition}{section}
\numberwithin{remark}{section}
\numberwithin{claim}{section}

In this section we present a short argument for projectivity that uses Halpern--Leistner's theory of stability for stacks.
Using this theory gives a shorter argument than King's or the GIT-free approach outlined in the body of this paper,
albeit it relies on a theorem (\cite[Theorem~5.6.1 (2)]{halpernleistner2022structure})
that we treat as a black box, and only holds in characteristic $0$.
The theorem says that
if a stack $\modulistack{}$ admits a good moduli space,
then its semistable locus does too,
and the latter good moduli space is projective over the former.
In particular, if the good moduli space of $\modulistack{}$ is a point,
then the good moduli space of the semistable locus is a projective variety.

Let $k$ be a field, which we assume to be of characteristic~0,
so that in the context of \cref{theorem: BGIT proof of projectivity},
we get a good (and not merely adequate) moduli space. 
As in \cref{subsection:ThetaS}, we define the stack $\Theta_k$ as $[\bbA^1_k/\Gm]$.
We will denote the closed point of $\Theta_k$ by $0$,
and let $1$ denote the point corresponding to the open orbit $\bbA^1 \setminus \{0\}$.
As explained in \cite[Corollary~7.13]{1812.01128},
a morphism $f\colon\Theta_k\to\modulistack{}$ to a moduli stack of objects in an abelian category
corresponds to a filtration, and the closed point corresponds to the associated graded of the filtration.
As in \cite[Section~3.2]{halpernleistner2022structure},
we will say that a filtration is \emph{non-degenerate}
if the induced morphism~$\Gm\to\Aut(f(0))$ of sheaves of groups has finite kernel.

For a $k$-scheme $T$ we write $\Theta_T = \Theta_k \times T$.
Following {\cite[Tag \href{https://book.themoduli.space/tag/00F3}{00F3}]{HLbook} and \cite[Corollary 7.13]{1812.01128}}, for an algebraic stack $\cX$, we define $\Filt \cX$, called the  \emph{stack of $\ZZ$-weighted filtrations}, to be the stack corresponding to the pseudo\-functor:
  \[\Filt \cX\colon T \to\Maps(\Theta_T, \cX).\]

From now on, let $\cX$ be an algebraic stack locally of finite type and with affine automorphism groups over a noetherian base scheme $S$.

If $\cL$ is an invertible sheaf on $\cX$, then we can define a locally constant \emph{weight function} $\wt_\cL\colon | \Filt \cX | \to \ZZ$:
\[
\wt_\cL\colon \big(f\colon \Theta_k \to \cX \big)\mapsto\wt_{\Gm} \big( \cL_{|f(0)} \big),
\]
where the $\Gm$-action on $\cL_{|f(0)}$ is induced by $\Gm = \Aut_{\Theta_k} (0) \to \Aut_\cX (f(0))$.

Having a weight function allows one to define a notion of semistability for points of stacks;
see \cite[Section~4.1]{halpernleistner2022structure} or \cite[Section~7.3]{1812.01128}.
We say that a point $p \in |\cX|$ is:
\begin{enumerate}[label=(\roman*)]
    \item \emph{$\cL$-unstable} (or \emph{$\Theta$-unstable} as in \cite[Section~7.3]{1812.01128}) if there is a nondegenerate filtration $f \in \Filt \cX$ with $f(1) = p$ and $\wt_\cL(f) < 0$;
    \item \emph{$\cL$-semistable} if it is not $\cL$-unstable.
\end{enumerate}

Let~$Q$ be a quiver, possibly with oriented cycles. We fix a dimension vector~$d$ and a stability function~$\theta$ such that~$\theta(d)=0$.
Let $\modulistack{d,S}$ be the moduli stack of representations of $Q$ of dimension vector $d$ over a noetherian scheme $S$
(\cref{definition: moduli stack M_d}).
Let $\cF^{\mathrm{univ}}$ be the universal family on $\modulistack{d,S}$.
Recall that we have the following line bundle on $\modulistack{d}$ (\cref{subsection:determinantal-line-bundles}):
\[
  \cL_\theta = \bigotimes_{i \in Q_0} (\det \cF^{\mathrm{univ}}_i)^{\otimes - \theta_i}.
\]

This line bundle induces a weight function and hence a notion of semistability on $\modulistack{d,S}$ which we will show coincides with King's notion of $\theta$-stability (\cref{subsection: background on stability of reps}).

\begin{lemma}
\label{lemma: theta is Theta}
  Fix a dimension vector~$d$ and a stability function~$\theta$ such that~$\theta(d)=0$.
  \begin{enumerate}[label=(\roman*)]
      \item
      \label{item: weight filtration formula}
      The weight function for a filtration $f\colon \Theta_k \to \modulistack{d,S}$ has the following formula:
        \[\wt_{\cL_\theta}(f)
        = - \sum_{n \in \ZZ} n \cdot
        \theta \left( \gr_n M \right)
        = - \sum_{n \in \ZZ}
        \theta \left( F_n M \right),
        \]
        where $f$ corresponds to a representation $M$ with filtration $\ldots \subset F_{n+1} M \subset F_{n} M \subset \ldots$ and $\gr_n M:=F_n M/F_{n+1}M$  so that $f(0)=\gr M$.
      \item
      \label{item: theta is Theta}
      A representation $M$ is $\theta$-semistable
      if and only
      if it is $\cL_\theta$-semistable.
  \end{enumerate}
\end{lemma}

\begin{proof}
    To prove the first part, the weight calculation is as follows:
    \begin{equation*}
      \begin{aligned}
        \wt_{\cL_\theta}(f)
        &=  \wt_{\Gm} \left( \bigotimes_{i \in Q_0} (\det \cF^{\mathrm{univ}}_i)^{\otimes - \theta_i} \right)\Bigg|_{f(0)}
        =  \wt_{\Gm} \left(
            \bigotimes_{i \in Q_0} (\det (\gr M)_i)^{\otimes - \theta_i}
            \right) \\
        &=- \sum_{i \in Q_0} \theta_i \cdot
            \wt_{\Gm} \det \left((\gr M)_i \right).
      \end{aligned}
    \end{equation*}
    The $\Gm$-weight corresponds to the grading weight, so $\wt_{\Gm} (\det \gr_n M_i) = n \dim M_i$, and therefore:
    \[
        \wt_{\cL_\theta}(f) =
        - \sum_{i \in Q_0}
        \theta_i \cdot
        \sum_{n \in \ZZ}
            n \dim \gr_n M_i
        = - \sum_{n \in \ZZ} n \cdot
        \theta \left( \gr_n M \right).
    \]
    The second equality follows from the fact that $\theta$ is additive in short exact sequences, so $\theta \left( \dimvect \gr_n M \right) = \theta \left( \dimvect F_n M \right) - \theta \left( \dimvect F_{n+1} M \right)$.

    Both equalities are used to prove the equivalence in the second part.
    Suppose first that $M$ is $\theta$-unstable, so that there is a subrepresentation $M' \subset M$ such that $\theta( \dimvect M' ) > 0$.
    We can view it as a two-step filtration $f = (M' \subset M)$ with $f(1) = M$
    whose associated graded factors are $\gr_0 = M / M'$, $\gr_1 = M'$
    while all other components are zero. Then by \ref{item: weight filtration formula}:
    \[
      \wt_{\cL_\theta}(f) = - \theta (\gr_1 M )
      =- \theta ( M' ) < 0.
    \]
    This shows that $M$ is $\cL_\theta$-unstable.

    Conversely, assume that $M$ is $\cL_\theta$-unstable and let $f$ be the destabilizing filtration with $\wt_{\cL_\theta}(f) < 0$. By \ref{item: weight filtration formula}, this is equivalent to
    \[ \sum_{n \in \ZZ} \theta(F_n M) > 0, \]
    so there exists at least one value of $n$ for which $\theta(F_n M) > 0$. Since $\theta(M) = \theta(d) = 0$ by assumption, we see that $F_n M \subset M$ is a proper subrepresentation that destabilizes $M$.
\end{proof}

\begin{theorem}
\label{theorem: BGIT proof of projectivity}
  Let $Q$ be a quiver, $d$ a dimension vector and $\theta$ a stability function with $\theta(d)=0$.
  Fix a noetherian scheme $S$ defined over $\QQ$.
  Then $\modulispace[\semistable{\theta}]{d,S}$ is an algebraic space which is
  projective over $\modulispace{d,S}$.
  In particular, if $S = \Spec k$ for a field $k$ of characteristic $0$ and $Q$ is acyclic,
  then $\modulispace[\semistable{\theta}]{d,k}$ is a projective variety.
\end{theorem}

\begin{proof}
  By \cref{corollary:Md-separated-gms} and since $S$ is of characteristic $0$, the stack $\modulistack{d,S}$ admits a good moduli space $\modulispace{d,S}$.
  By \cref{lemma: theta is Theta}~\ref{item: theta is Theta}, the substack $\modulistack[\semistable{\theta}]{d,S} \subset \modulistack{d,S}$
  coincides with the substack of $\cL_\theta$-semistable objects,
  so we can apply \cite[Theorem~5.6.1(2)]{halpernleistner2022structure} to $\pi = \mathrm{id}_{\modulistack{d}}\colon\modulistack{d} \to \modulistack{d}$ and conclude that
  this substack admits a good moduli space $\modulispace[\semistable{\theta}]{d,S}$, yielding the commutative diagram
    \[\begin{tikzcd}
        \modulistack[\semistable{\theta}]{d,S} &
        \modulistack{d,S}\\
        \modulispace[\semistable{\theta}]{d,S} & \modulispace{d,S}
        \arrow[hook, from=1-1, to=1-2]
        \arrow[from=1-2, to=2-2]
        \arrow[from=1-1, to=2-1]
        \arrow[from=2-1, to=2-2]
    \end{tikzcd}\]
    where vertical arrows are good moduli spaces
    and the bottom morphism $\modulispace[\semistable{\theta}]{d,S} \to \modulispace{d,S}$ is projective. 
\end{proof}

Using the description $\modulistack{d,S} \cong [\mathrm{R}_{d,S}/\mathrm{G}_{d,S}]$ from \cref{proposition:smooth-cover-by-Rep}
and the fact that the good moduli space of such a global quotient stack
is given by the ring of invariants,
then we can also get a more streamlined proof for quivers with cycles that
the good moduli space $\modulispace[\semistable{\theta}]{d,S}$ is projective-over-affine.
It would be interesting to find a proof that $\modulispace{d,S}$ is affine without methods from GIT.

{\small
\bibliographystyle{abbrv}
\bibliography{automatic}
}

\setlength\parindent{0pt}
\setlength\parskip{6pt}

\texttt{pieter.belmans@uni.lu} \\
University of Luxembourg, Department of Mathematics,
6, Avenue de la Fonte,
L-4364 Esch-sur-Alzette,
Luxembourg

\texttt{chiarad@sas.upenn.edu} \\
University of Pennsylvania, Department of Mathematics,
David Rittenhouse Lab,
209 South 33rd Street
Philadelphia, PA 19104-6395,
United States

\texttt{hans.franzen@math.upb.de} \\
Paderborn University, Institute of Mathematics,
Warburger Stra\ss e 100,
33098 Paderborn,
Germany

\texttt{v.hoskins@math.ru.nl} \\
Radboud University, IMAPP, PO Box 9010, 6525 AJ Nijmegen, Netherlands

\texttt{murmuno@sas.upenn.edu} \\
University of Pennsylvania, Department of Mathematics,
David Rittenhouse Lab,
209 South 33rd Street
Philadelphia, PA 19104-6395,
United States

\texttt{tuomas.tajakka@math.su.se} \\
Stockholm University, Department of Mathematics, Albanov\"agen 28, 114 19 Stockholm,
Sweden

\end{document}